\documentclass[final,onefignum,onetabnum]{siamart171218}

\usepackage{lipsum}
\usepackage{enumitem}
\usepackage{amsfonts}
\usepackage{graphicx}
\graphicspath{{images_sicon/}}
\usepackage{epstopdf}
\usepackage{algorithmic}
\ifpdf
  \DeclareGraphicsExtensions{.eps,.pdf,.png,.jpg}
\else
  \DeclareGraphicsExtensions{.eps}
\fi
\usepackage{mathtools}
\usepackage{amssymb}
\usepackage{caption}
\usepackage{subcaption}

\usepackage{xcolor}
\usepackage[normalem]{ulem}

\newcommand{\RevAdd}[1]{{\color{blue}#1}}
\newcommand{\RevDel}[1]{{\color{red}\sout{#1}}}

\newsiamremark{remark}{Remark}
\newsiamremark{hypothesis}{Hypothesis}
\crefname{hypothesis}{Hypothesis}{Hypotheses}
\newsiamthm{claim}{Claim}

\newsiamremark{assumption}{Assumption}
\crefname{assumption}{Assumption}{Assumption}

\headers{Feedback Stabilization of McKean-Vlasov PDEs}{D. Kalise, L. M. Moschen, and G. A. Pavliotis}

\title{Linearization-Based Feedback Stabilization of McKean-Vlasov PDEs\thanks{Submitted to the editors November 3rd, 2025.
\funding{D. Kalise is partially supported by the EPSRC Standard Grant EP/T024429/1. 
L. M. Moschen is funded by the Roth Scholarship at Imperial College London with travel funds from the ICL-CNRS Lab. 
G. A. Pavliotis is partially supported by an ERC-EPSRC Frontier Research Guarantee through Grant No. EP/X038645, ERC Advanced Grant No. 247031 and a Leverhulme Trust Senior Research Fellowship, SRF$\backslash$R1$\backslash$241055.
}}}

\author{Dante Kalise\thanks{Department of Mathematics, Imperial College London, London, UK (\email{d.kalise-balza@imperial.ac.uk}, \email{lucas.moschen22@imperial.ac.uk}, \email{g.pavliotis@imperial.ac.uk}).}
\and Lucas M. Moschen\footnotemark[2]
\and Grigorios A. Pavliotis\footnotemark[2]}

\usepackage{amsopn}

\makeatletter
\newcommand*{\addFileDependency}[1]{% argument=file name and extension
  \typeout{(#1)}% latexmk will find this if $recorder=0 (however, in that case, it will ignore #1 if it is a .aux or .pdf file etc and it exists! if it doesn't exist, it will appear in the list of dependents regardless)
  \@addtofilelist{#1}% if you want it to appear in \listfiles, not really necessary and latexmk doesn't use this
  \IfFileExists{#1}{}{\typeout{No file #1.}}% latexmk will find this message if #1 doesn't exist (yet)
}
\makeatother

\newcommand{\R}{\mathbb{R}}
\newcommand{\C}{\mathbb{C}}
\newcommand{\T}{\mathbb{T}}
\newcommand{\A}{\mathcal{A}}
\newcommand{\Nop}{\mathcal{N}}
\newcommand{\W}{\mathcal{W}}
\newcommand{\B}{\mathcal{B}}
\newcommand{\M}{\Mcal}
\newcommand{\D}{\mathcal{D}}
\newcommand{\K}{\mathcal{K}}
\newcommand{\Hcal}{\mathcal{H}}
\newcommand{\Lcal}{\mathcal{L}}
\newcommand{\Pcal}{\mathcal{P}}
\newcommand{\Pac}{\mathcal{P}_{\mathrm{ac}}(\Omega)}
\newcommand{\Qcal}{\mathcal{Q}}
\newcommand{\Mcal}{\mathcal{M}}
\newcommand{\unitary}{\mathcal U}

\newcommand{\Space}{\mathcal{X}}
\DeclarePairedDelimiter\ceil{\lceil}{\rceil}
\DeclarePairedDelimiter\floor{\lfloor}{\rfloor}

\renewcommand{\RevAdd}[1]{#1}
\renewcommand{\RevDel}[1]{}

\ifpdf
\hypersetup{
  pdftitle={Linearization-Based Feedback Stabilization of McKean-Vlasov PDEs},
  pdfauthor={D. Kalise, L. M. Moschen, and G. A. Pavliotis}
}
\fi

\begin{document}

\maketitle

% REQUIRED
\begin{abstract}
    We develop a feedback control framework for stabilizing the McKean-Vlasov PDE on the torus.
    Our goal is to steer the dynamics toward a prescribed stationary distribution or accelerate convergence to it using a time-dependent control potential. 
    We reformulate the controlled PDE in a weighted, zero-mean space and apply the ground-state transform to obtain a Schrödinger-type operator. 
    The resulting operator framework enables spectral analysis, verification of the infinite-dimensional Hautus test, and construction of a Riccati-based feedback law derived from the linearized dynamics, yielding local exponential stabilization with a chosen convergence rate.
    We rigorously prove local exponential stabilization via maximal regularity arguments and nonlinear estimates. 
    Numerical experiments on well-studied models in one and two dimensions (the noisy Kuramoto model for synchronization, the $O(2)$ spin model in a magnetic field, and the von Mises attractive interaction potential) showcase the effectiveness of our control strategy, demonstrating convergence acceleration and stabilization of unstable equilibria.
\end{abstract}

% REQUIRED

% REQUIRED

\section{Introduction}

The {\em McKean-Vlasov equation} is a nonlinear and nonlocal Fok\-ker-Planck-type partial differential equation (PDE) describing the evolution of a probability \RevAdd{distribution with density} $\mu(t,x)$ that arises as the mean-field limit of a system of interacting particles subject to noise~\cite{mckean1969propagation, mckean1966class, sznitman1991topics}.
In this paper, we consider the overdamped McKean-Vlasov PDE on the torus $\Omega = \T^d \coloneqq (\R / 2\pi \mathbb{Z})^d$:
\begin{equation}
    \label{eq:mckean_vlasov}
    \partial_t \mu(t,x) = \nabla \cdot \Bigl[\sigma\nabla\mu(t,x) + \mu(t,x)\bigl(\nabla V(x) + (\nabla W * \mu)(x)\bigr)\Bigr], \quad x \in \Omega, \; t > 0
\end{equation}
with periodic boundary conditions, \RevAdd{i.e., $\mu(\cdot,t)$ is $2\pi$-periodic in each coordinate for every $t>0$}, where $*$ denotes the convolution $(f * g)(x) \coloneqq \int_\Omega f(x-x') \, g(x') \, dx'$, $V : \Omega \to \R$ is \RevAdd{a periodic} external potential, $W : \Omega \to \R$ is \RevAdd{a periodic} interaction potential, and $\sigma>0$ is the diffusion coefficient.
Equation~\eqref{eq:mckean_vlasov} appears in many applications to physics~\cite{bavaud1991equilibrium, frank2005nonlinear}, models for collective behavior~\cite{naldi2010mathematical}, and machine learning~\cite{sirignano2020mean}.
It governs the evolution of the probability \RevAdd{distribution} of \RevAdd{$X_t$, where $X_t$ solves} the McKean-Vlasov stochastic differential equation (SDE)
\begin{equation}\label{e:mckean_sde}
    dX_t = -\nabla V(X_t) \, dt - (\nabla W * \mu(t,\cdot))(X_t) \, dt + \sqrt{2\sigma} \, dB_t,
\end{equation}
where $\mu(t,\cdot)$ \RevAdd{denotes} the probability density of the \RevAdd{law of $X_t$, with initial law admitting density $\mu_0$}, and $\{B_t\}_{t\ge 0}$ is a standard $d-$dimensional Brownian motion\RevAdd{~\cite{mckean1966class, pavliotis2014stochastic}}.

We are interested in the stationary solutions $\bar\mu$ of the McKean-Vlasov PDE:
\[
\nabla \cdot \bigl[\sigma\nabla\bar\mu + \bar\mu\bigl(\nabla V+(\nabla W*\bar\mu)\bigr)\bigr]=0.
\]
As is well known, the stationary equation above is equivalent to the Kirkwood-Monroe integral equation~\cite{kirkwood1941statistical, tamura1984asymptotic}
\[
\bar\mu(x) = \frac{1}{Z}\exp\Bigl(-\frac{1}{\sigma}\bigl(V(x)+ (W*\bar\mu)(x)\bigr)\Bigr), \quad Z=\int_{\T^d} \exp \Bigl(-\frac{1}{\sigma}(V(y)+ (W*\bar\mu)(y))\Bigr)\,dy.
\]
Under uniform convexity of $V$ and convexity of $W$ (or $H$-stability on $\T^d$, see \Cref{sec:stationary_states} for the definition), one obtains a unique $\bar\mu$ and exponential convergence to the steady state~\cite{carrillo2003kinetic, malrieu2001logarithmic}. 
In the absence of such convexity, or when $W$ is not $H$-stable, multiple equilibria and bifurcations may occur at sufficiently high interaction strengths~\cite{bertoli2024stability, bertoli2025phase, carrillo2020long, cayes2010mckean, dawson1983critical, tamura1984asymptotic}. 
The goal of this paper is to study feedback stabilization for stationary states of the McKean-Vlasov dynamics.

\RevAdd{Research on controlling McKean-Vlasov dynamics has predominantly followed stochastic variational routes, where the control problem is formulated at the level of the McKean-Vlasov SDE. 
This literature can be divided into mean-field games and mean-field type control, corresponding respectively to non-cooperative Nash equilibria in large populations and to centralized formulations of controlled McKean-Vlasov dynamics; see, e.g.,~\cite{bensoussan2013mean,carmona2013control}.
In mean-field games, the backward equation describes the value of a representative agent, while the forward equation describes the induced population distribution, together with the fixed-point condition that the two coincide~\cite{carmona2018probabilistic,carmona2013control,huang2006large,lasry2007mean}.
In particular, the monograph~\cite{cardaliaguet2019master} establishes well-posedness of the master equation and studies the convergence of the $N$-player Nash system, as $N \to \infty$, to the mean-field game limit on $\T^d$.
In mean-field type control, by contrast, the population is controlled by a central planner, and the corresponding forward-backward systems and master equations encode optimality conditions for a centralized problem, with additional terms reflecting the dependence on the distribution~\cite{andersson2010maximum,bensoussan2013mean,pham2017dynamic}. 
We also refer to~\cite{djete2022mckean} for a discussion of strong, weak, and related formulations, and to~\cite{cardaliaguet2023algebraic} and the references therein for analyses of mean-field optimal control and $N$-particle approximations.
A recent controllability result for McKean-Vlasov SDEs appears in~\cite{barbu2023exact}.

Closer to our framework, direct control of the nonlocal Fokker-Planck equation at the PDE level has received comparatively less attention. 
For comparison, the linear Fokker-Planck equation (i.e., $W \equiv 0$) has been extensively studied from a PDE-control point of view, including controllability and Schrödinger bridge connections~\cite{blaquiere1992controllability}, open-loop optimal control~\cite{aronna2021first, Fleig2017}, and Riccati-based feedback stabilization~\cite{breiten2018control, Breiten2023}. 
In the McKean-Vlasov setting, first-order optimality conditions for the PDE were derived in~\cite{albi2016mean, bongini2017mean}, and an optimal control problem for the noisy Kuramoto model is studied in~\cite{chen2025optimal}.
The noisy Kuramoto model is a canonical McKean-Vlasov PDE for mean-field coupled phase oscillators on the circle~\cite{kuramoto1984chemical}.
Related feedback and backstepping-type designs for parabolic drift-diffusion PDEs can be found in~\cite{coron2007control}. 
Two recent contributions closer to the present PDE-control setting are~\cite{bicego2025computation}, which studies the computation and control of unstable steady states through nonlinear model predictive control, and~\cite{albi2025control}, which studies the design of control strategies based on mesoscopic feedback laws.
A complementary information-theoretic perspective on feedback control for the noisy Kuramoto model is developed in~\cite{sowinski2024information}.}

However, a theoretical stabilization framework for deterministic feedback laws in $\mathbb{T}^d$ that handles nonlocal mean-field coupling and multiple equilibria remains lacking.
\RevAdd{Our results provide such a framework for stationary solutions of the nonlinear McKean-Vlasov PDE on $\T^d$, based on a finite-dimensional feedback design, spectral analysis, and an operator Riccati equation.} 
We achieve this by extending the linearization-based method of~\cite{breiten2018control} from linear to nonlinear Fokker-Planck equations by linearizing the McKean-Vlasov PDE around a steady state and analyzing the resulting generator in the weighted space $L^2(\Omega,\bar\mu^{-1})$ (see~\cite{pavliotis2025linearization}). 
\RevAdd{This places our problem closer to mean-field type control, since the control acts on the population law via a single decision maker.}
The main challenge is that the mean-field term $\nabla \cdot (\mu \nabla W * \mu)$ leads to a non-self-adjoint linearized operator in the natural function space of the problem, so the spectrum does not need to be real (in contrast to the linear Fokker-Planck case). 
We address this by proving that only finitely many unstable modes need to be controlled, and show that feedback designed from the linearization locally stabilizes the full nonlinear McKean-Vlasov PDE.

\subsection{Contributions and main results}

We develop a framework for feedback stabilization of the McKean-Vlasov PDE on the torus through linearization-based control design.

We begin by introducing a control term into the McKean-Vlasov equation \eqref{eq:mckean_vlasov} through \RevAdd{a finite-dimensional family of} time-dependent potentials of the form
\[
V(x) \mapsto V(x) + \sum_{j=1}^m u_j(t) \alpha_j(x),
\]
where $u_j(t) \in \mathbb{R}$ are scalar controls to be determined and $\alpha_j$ are prescribed shape functions \RevAdd{on $\Omega$}. 
The controlled PDE becomes
\[
\partial_t \mu = \nabla \cdot \left[ \sigma \nabla \mu + \mu \left( \nabla V + \nabla W * \mu + \sum_{j=1}^m u_j(t) \nabla \alpha_j \right) \right].
\]
For a target steady state $\bar{\mu}$, we analyze the perturbation $y = \mu - \bar{\mu}$ and derive the linearized dynamics
\[
\partial_t y = \mathcal{L} y + \sum_{j=1}^m \mathcal{B}_j u_j(t),
\]
where $\mathcal{L}$ is the linearized McKean-Vlasov operator incorporating both local drift-diffusion and the linearized nonlocal mean-field coupling, and $\mathcal{B}_j = \nabla \cdot (\bar{\mu} \nabla \alpha_j)$ represents the influence of the $j$-th control on the linearized system.
\RevAdd{The definition of the operator $\Lcal$ is given later in~\eqref{eq:linearized_operator}.}

The central theoretical contribution is establishing that feedback control based on this linearization achieves local exponential stabilization. 
\RevAdd{For a prescribed target decay rate $\delta > 0$, we focus on the finitely many eigenvalues of $\Lcal$ with real part at least $-\delta$, i.e., the modes that prevent or slow exponential decay at rate $\delta$.} 
In~\Cref{thm:hautus}, we prove that by choosing the shape functions $\{\alpha_j\}_{j=1}^m$ to solve the elliptic equations
\[
\nabla \cdot (\bar{\mu} \nabla \alpha_j) = \psi_j, \quad j = 1, \ldots, m,
\]
where $\{\psi_j\}_{j=1}^m$ are the eigenfunctions of $\mathcal{L}$ corresponding to \RevAdd{those eigenvalues,} the pair $(\mathcal{L}, \mathcal{B})$ satisfies the infinite-dimensional Hautus test, \RevAdd{the spectral criterion used to establish $\delta$-stabilizability}. 
This choice of shape functions ensures stabilization of precisely those modes that would otherwise prevent or slow exponential decay.

\RevAdd{This finite-dimensional control strategy can be compared with mean field type control formulations such as~\cite{bensoussan2013mean}. 
In those formulations, the drift in the forward equation is determined through a coupled forward-backward system. 
If the value function is restricted to an ansatz of the form $\phi(t,x) = \sum_{j=1}^m p_j(t)\alpha_j(x)$, then the resulting controlled drifts lie in the same finite-dimensional space spanned by $\{\nabla \alpha_j\}_{j=1}^m$ as in our setting.
The difference is that here the coefficients are obtained from a linear-quadratic regulator problem for the linearized dynamics, which leads to a more tractable operator Riccati equation instead of a full forward-backward system.}

The spectral analysis underlying this result relies on a ground-state transformation that maps the linearized operator $\mathcal{L}$ to a Schrödinger-type operator $\mathcal{H} = \mathcal{H}_{\text{loc}} + \mathcal{K}$ acting on $L^2(\mathbb{T}^d)$. 
As established in~\Cref{prop:pure-spectrum}, this operator has compact resolvent and discrete spectrum consisting of isolated eigenvalues of finite multiplicity. 
Crucially, \Cref{lem:spectrum_finite} proves that for any $\delta > 0$, only finitely many eigenvalues satisfy $\Re(\lambda) \geq -\delta$, reducing the stabilization problem to controlling a finite number of modes despite the infinite-dimensional nature of the PDE.

Given the Hautus test verification, we solve the algebraic Riccati equation
\[
(\mathcal{L}^* + \delta I)\Pi + \Pi(\mathcal{L} + \delta I) - \Pi \mathcal{B} \mathcal{B}^* \Pi + \Mcal = 0,
\]
where $\mathcal{M}$ is a positive weight operator and $\mathcal{B} = [\mathcal{B}_1, \ldots, \mathcal{B}_m]$. 
\Cref{prop:ARE} establishes the existence of a unique self-adjoint, nonnegative, bounded solution $\Pi$ and shows that the feedback law
\[
u(t) = -\mathcal{B}^* \Pi y(t)
\]
renders the linearized operator $\mathcal{L}_\Pi = \mathcal{L} + \delta I - \mathcal{B} \mathcal{B}^* \Pi$ exponentially stable.

Our main stabilization result, \Cref{thm:nonlin-existence}, proves that this linear feedback law achieves local exponential stabilization of the full nonlinear McKean-Vlasov equation. 
Specifically, if the initial distribution $\mu_0$ satisfies \RevAdd{$\|\mu_0 - \bar\mu\|_{\bar{\mu}^{-1}} \leq \varepsilon$} for sufficiently small $\varepsilon > 0$, then the closed-loop solution under the feedback $u(t) = -\mathcal{B}^* \Pi y(t)$ satisfies
\RevAdd{\[
\|y(t)\|_{\bar{\mu}^{-1}} \leq Ce^{-\delta t}, \quad \forall t \geq 0,
\]}
where the constant $C$ depends on the initial condition. 
The proof employs maximal regularity theory (\Cref{thm:max-reg-Hpi}) to handle the linear evolution and establishes Lipschitz estimates for the nonlinear terms (\Cref{lem:feedback-nonlinear} and \Cref{lem:remainder}) to complete a contraction mapping argument in the evolution space $W(0,\infty;\mathbb{T}^d)$ \RevAdd{defined in~\eqref{e:evol_space}}.
 
\Cref{fig:kuramoto_control_vs_uncontrol_profiles} provides a qualitative illustration of the proposed methodology. 
For a noisy Kuramoto model ($V=0$, $W=-K\cos (x)$) with supercritical coupling ($K>1$), where the uniform distribution $\bar{\mu} = (2\pi)^{-1}$ is unstable. 
Our feedback law successfully stabilizes this equilibrium that would otherwise be inaccessible. 
Similarly, the control can accelerate convergence to stable equilibria or steer the system between different steady states in models with multiple equilibria.
 
\begin{figure}[!htbp]
  \centering
  \includegraphics[width=0.7\textwidth]{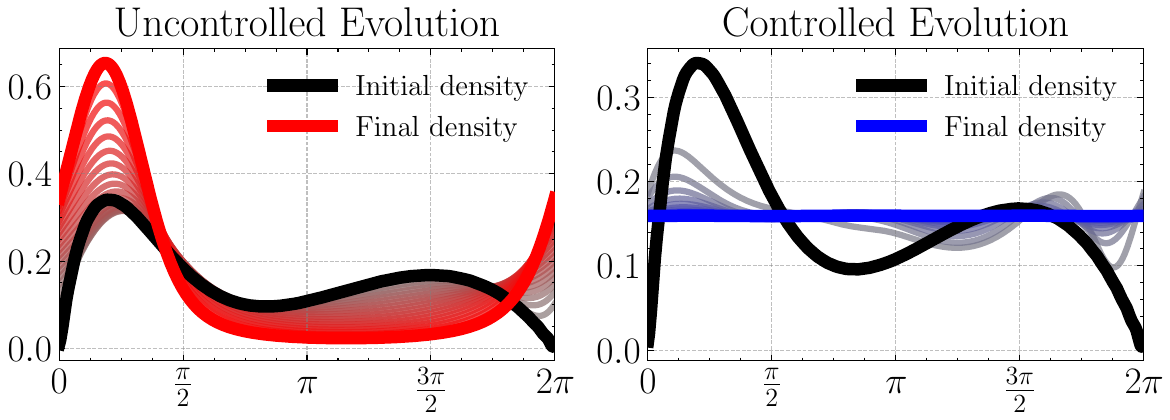}
  \caption{{\bf Time-evolution to the unstable steady state.}
  Uncontrolled (left) vs. controlled (right) evolution of \RevAdd{$\mu(t,x)$} for the noisy Kuramoto model with \(\sigma=0.5\), \(V=0\), \(W=-K\cos(x)\).  
  Black: initial density;  gray: intermediate snapshots; red/blue: final density (uncontrolled/controlled).}
  \label{fig:kuramoto_control_vs_uncontrol_profiles}
\end{figure}
%
%%%%%%%%%%%%%%%%%%%%%%%%%%%%%%%%%%%%%%%%%%%%%%%%%%%%%%%%
%
\subsection{Outline of the paper}

\Cref{sec:wellposed} reviews the variational formulation and well-posedness of the McKean-Vlasov equation on the torus, including convergence to equilibrium under convexity assumptions.
\Cref{sec:operator} introduces the linearization around a steady state, the Schrödinger-type reformulation via ground-state transformation, and the spectral properties of the linearized operator. 
\Cref{sec:control} formulates the optimal control problem, derives a Riccati-based feedback law via the Hautus test, and proves local exponential stability of the nonlinear system.
\Cref{sec:numerics} presents numerical experiments on several models illustrating the efficacy of our feedback design.

\section{Preliminaries and well-posedness}
\label{sec:wellposed}

We begin by introducing the notation and function spaces used throughout the paper.
We then recall the well-posedness theory for the McKean-Vlasov equation under periodic boundary conditions and discuss the structure and convergence properties of its stationary states.
Finally, we linearize the dynamics around the equilibrium and reformulate the resulting operator using a ground-state transformation, which will serve as the basis for the feedback control design in subsequent sections.

\subsection{Notation}

Given a real interval $[0,T]$, a normed space $X$ and $1 \le p \le \infty$, we let $L^p(0,T; X)$ denote the set of $L^p$ functions $f : [0,T] \to X$.
We denote by $L^p(\Omega)$ the set of $L^p$ periodic functions defined on $\Omega$.
In addition, we denote by $W^{k,p}(\Omega)$ the periodic Sobolev spaces with $H^k(\Omega) = W^{k,2}(\Omega)$, $L^1_{+}(\Omega) \coloneqq  \{f \in L^1(\Omega) : f \ge 0\}$, $L^2_{0}(\Omega) \coloneqq  \{f \in L^2(\Omega) : \int_{\Omega} f = 0\}$, and for a positive weight $w(x) > 0$, $L^p(\Omega, w) \coloneqq  \{f : fw^{1/p} \in L^p(\Omega)\}$.
\RevAdd{For the set of probability densities on $\Omega$, we denote $\Pac \coloneqq \{f \in L^1_+(\Omega) : \|f\|_1 = 1\}$.}
We write $\|\cdot\|_p$ for the norm in $L^p(\Omega)$ and $\|\cdot\|_{w}$ for the norm in $L^2(\Omega, w)$.
The inner product in $L^2(\Omega)$ is denoted by $\langle \cdot, \cdot \rangle$ and the inner product in $L^2(\Omega, w)$ by $\langle \cdot, \cdot \rangle_w$.
The space ${[H^1(\Omega)]}^*$ denotes the topological dual of $H^1(\Omega)$, with respect to $L^2(\Omega)$ as the pivot space.
We use the differential operators $\nabla = (\partial_{x_1},\dots,\partial_{x_d})$, $\nabla \cdot F = \sum_{i=1}^d\partial_{x_i}F_i$, $\Delta = \nabla \cdot \nabla$, and the time derivative $\partial_t\mu=\frac{\partial\mu}{\partial t}$.
Over a Hilbert space $H$, the set of bounded linear operators is denoted by $\mathbb{L}(H)$.
\RevAdd{For a closed linear operator $\mathcal{T}$ defined on $H$, $\sigma(\mathcal{T})$ and $\rho(\mathcal{T})$ denote its spectrum and resolvent set, respectively. 
For $\lambda \in \mathbb{C}$, $\Re(\lambda)$ and $\Im(\lambda)$ denote the real and imaginary parts of $\lambda$, respectively.}

\subsection{Model setup and variational formulation}

We now introduce the weak formulation of the McKean--Vlasov equation considered in this work, along with the standing assumptions on the model parameters. 
\RevAdd{Then, we state the well-posedness result, and a complete proof is provided in~\Cref{app:full-wellposedness-proof}.}

\begin{assumption}[Model assumptions]\label{assumption:model}
    Assume the confining potential $V \in W^{1,\infty}(\Omega) \cap W^{2,p}(\Omega)$ (for $p=\max\{d,2+\epsilon\}$ and some $\epsilon > 0$) and the interaction potential $W \in W^{2,\infty}(\Omega)$ \RevAdd{are $2\pi$-periodic in each coordinate}; $W$ is coordinate-wise even; the diffusion coefficient $\sigma > 0$ is fixed; the initial condition $\mu_0 \in \Pac$ \RevAdd{is $2\pi$-periodic in each coordinate}.
\end{assumption}

Under \Cref{assumption:model}, we consider the McKean-Vlasov equation on the torus
\begin{equation}
    \label{eq:mv_wellposed}
    \partial_t \mu = \sigma \Delta \mu + \nabla \cdot \bigl(\mu \nabla v[\mu]\bigr), 
    \quad \text{with } v[\mu](x) = V(x) + (W * \mu)(x),
\end{equation}
subject to periodic boundary conditions 
\RevAdd{\[
  \mu(t, x + 2\pi e_k) = \mu(t,x), \quad \forall x \in \R^d,\; t>0,\; k=1,\dots,d,
\]}
and initial data $\mu(0,\cdot) = \mu_0$.

\begin{definition}[Weak solution]\label{def:varsol}
    For $T > 0$, a function $\mu \in L^2(0, T; H^1(\Omega)) \cap L^{\infty}(0, T; L^2(\Omega))$ with $\partial_t \mu \in L^2(0, T; {[H^1(\Omega)]}^*)$ is a {\em weak solution} of \eqref{eq:mv_wellposed} if for every $\varphi \in L^2(0, T; H^1(\Omega))$,
    \[
    \int_0^T \langle \partial_t \mu(t,\cdot), \varphi(t,\cdot) \rangle \, dt
    + \int_0^T \langle \sigma \nabla \mu(t,\cdot) + \mu(t,\cdot)  \nabla v[\mu(t,\cdot)], \nabla \varphi(t,\cdot) \rangle \, dt = 0
    \]
\end{definition}

\RevAdd{Although probability densities are naturally $L^1$ objects, we formulate weak solutions as in \Cref{def:varsol} because the existence theory developed here relies on standard $L^2$ and $H^1$ estimates.
Additionally, \Cref{thm:wellposed} shows that, for $\mu_0 \in L^2(\Omega) \cap \Pac$, the solution remains in $\Pac$ for all $t>0$.}
This notion of solution ensures sufficient time regularity to make the initial condition meaningful. 
In particular, using arguments similar to~\cite[Thm. 3, Sec. 5.9]{evans2022partial}, the mapping $t \mapsto \mu(t, \cdot)$ admits a continuous representative in $L^2(\Omega)$ when $\mu \in L^2(0, T; H^1(\Omega))$ and $\partial_t \mu \in L^2(0, T; {[H^1(\Omega)]}^*)$. 
This guarantees the well-definedness of the initial condition \( \mu(0) = \mu_0 \) in the weak topology.

\begin{theorem}[Global well-posedness]
    \label{thm:wellposed}
    Let $\mu_{0}\in \RevAdd{L^{2}(\Omega)} \cap \Pac$.
    Under \Cref{assumption:model}, there exists a unique weak solution 
    \RevAdd{\[
    \mu \in L^\infty(0,T;L^2(\Omega)) \cap L^2(0,T;H^1(\Omega)),
    \qquad
    \partial_t \mu \in L^2(0,T;[H^1(\Omega)]^*),
    \]}
    to~\eqref{eq:mv_wellposed} with the estimate
    \begin{equation*}
        \|\mu\|_{{L}^{\infty}\left(0,T;{L}^{2}(\Omega)\right)} + \|\mu\|_{{L}^{2}\left(0,T;H^{1}(\Omega)\right)}+\|\partial_t \mu\|_{{L}^{2}\left(0,T;{[H^{1}(\Omega)]}^*\right)}\leq C(T)\|\mu_0\|_{2}.
    \end{equation*}
    for a time-dependent constant $C(T) > 0$.
    Additionally, $\mu(t) \in \Pac$ for all $t > 0$.
\end{theorem}

\begin{proof}
    The proof proceeds via a Picard iteration adapted from~\cite{chazelle2017well}. 
    Specifically, from a smooth initial condition, we construct a sequence $\{\mu_n\}_{n \in \mathbb{N}} \subseteq L^2(0,T; H^1(\Omega))$ of solutions to
    \begin{align}
    	&\begin{cases}
    	\begin{aligned}
    		\partial_t \mu_{n} &= \sigma \Delta \mu_{n} + \nabla \cdot \left(\mu_{n} v[\mu_{n-1}] \right) &\text{in } (0,T] \times \Omega,  \\
    	\mu_{n}(0, \cdot) &= \mu_{0} &\text{on } \Omega,
    	\end{aligned}
    	\end{cases}\label{eq:sequence_evolution}
    \end{align}
    which are linear Fokker-Planck equations with the drift vector field frozen from the previous iterate.
    Existence and uniqueness of solution to~\eqref{eq:sequence_evolution} at each step in $C^{\infty}(0, T; C^{\infty}(\Omega))$ follow from standard results for linear parabolic equations in the torus; see, e.g.,~\cite[Ch. 6]{pavliotis2014stochastic}, ~\cite[Thm. 2.6.13 and Thm. 3.2.1]{pazy2012semigroups}.

    To pass to the limit, we establish uniform (on $n$) a priori estimates for $\mu_n$ in $L^2(0, T; H^1(\Omega)) \cap L^{\infty}(0, T; L^2(\Omega))$ and for $\partial_t \mu_n$ in $L^2(0, T; {[H^1(\Omega)]}^*)$ by exploiting the regularity of $V$ and $W$. 
    These ensure that the sequence $\{\mu_n\}_{n \in \mathbb{N}}$ is bounded in the appropriate Sobolev-Bochner space and admits a subsequence converging weakly to a limit $\mu$.
    Furthermore, we establish that the sequence converges strongly in $L^1(0, T; L^1(\Omega))$, which suffices to produce a weak solution of~\eqref{eq:mv_wellposed}.
    The limiting function $\mu$ \RevAdd{satisfies the same a priori estimate as the approximating sequence, by weak lower semicontinuity of the corresponding norms, and it satisfies the weak formulation of \Cref{def:varsol}.} %inherits the energy bounds from the approximating sequence and satisfies the weak formulation of \Cref{def:varsol}. 
    Existence, therefore, follows. 
    The uniqueness of weak solutions is obtained via a Grönwall argument applied to the difference of two solutions in $L^2$.
    The result is then extended to an $L^2$ initial data using mollification.
    \RevAdd{A complete proof is given in~\Cref{app:full-wellposedness-proof}}.
    We also refer to~\cite{nickl2025bayesian, pavliotis2025linearization} for additional arguments.
\end{proof}

\begin{remark}[Classical regularity under smooth data]
    If $V, W \in W^{2k + 1, \infty}(\Omega)$ and $\mu_0 \in H^{2k}(\Omega) \cap \Pac$ for $k = \ceil*{2 + d/4}$, then the unique weak solution $\mu$ to~\eqref{eq:mv_wellposed} admits a representative in 
    \[
    \mu \in C^1((0,T);\, C^2(\Omega)),
    \]
    after modification on a set of measure zero. 
    This follows from standard parabolic regularity theory and the Sobolev embedding on the torus.
    A complete argument is provided in~\Cref{sup:sec:regularity}.
\end{remark}

\subsection{Stationary states and exponential convergence}
\label{sec:stationary_states}

We recall that a stationary solution $\bar\mu$ of \eqref{eq:mv_wellposed} satisfies
\begin{equation}
    \label{eq:stationary}
    \nabla\cdot\bigl(\bar\mu \nabla v[\bar\mu]\bigr)+ \sigma \Delta \bar\mu = 0\quad\Longleftrightarrow\quad \bar\mu(x) = \frac{1}{Z} \exp\bigl(-\tfrac{1}{\sigma} v[\bar\mu](x)\bigr),
\end{equation}
with $\displaystyle Z = \int_{\Omega} \exp\bigl(-\tfrac{1}{\sigma} v[\bar\mu](x')\bigr) \, dx'$~\cite{tamura1984asymptotic}.

The following \Cref{thm:equilibria} follows from the logarithmic Sobolev inequality induced by displacement convexity, which is rigorously proved in~\cite{malrieu2001logarithmic} for the Euclidean space $\R^d$ and is carried over to the torus with periodic boundary conditions.
We refer to~\cite[Thm. 2.1]{carrillo2003kinetic} and~\cite[Sec. 2.3]{pavliotis2025linearization} for extensions and recall the definition of $H$-stability: a function $W \in L^2(\Omega)$ is said to be $H$-stable if for all $\nu \in L^2(\Omega)$, 
\[
\int_{\Omega \times \Omega} W(x-y) \nu(x) \nu(y) \, dx \, dy \ge 0.
\]

\begin{proposition}[Existence and uniqueness of stationary states]\label{thm:equilibria}
    Under \Cref{assumption:model} and $V, W \in \mathcal C^2(\Omega)$ being either $H$-stable or such that $\|{W}\|_{\infty} < \sigma$.
    Then, 
    \begin{enumerate}
        \item[(i)] there exists a solution $\bar\mu \in \Pac$ to~\eqref{eq:stationary};
        \item[(ii)] if the initial condition has finite relative entropy with respect to $\bar\mu$\RevAdd{, namely $\int_\Omega \mu_0 \log (\mu_0 / \bar\mu) \, dx < \infty$,}
        the solution of \eqref{eq:mv_wellposed} satisfies
        \[
        \|\mu(t,\cdot)-\bar\mu\|_{1} \le e^{-\alpha t} \|\mu_0-\bar\mu\|_{1}
        \]
        for an appropriate constant $\alpha > 0$.
    \end{enumerate}  
\end{proposition}

\begin{remark}[Non-convex case and bifurcations]
    If the interaction potential $W$ is not $H$-stable, equation~\eqref{eq:stationary} \RevAdd{may} have multiple solutions at sufficiently low temperatures \RevAdd{(equivalently, for sufficiently small diffusion $\sigma$)}, leading to phase transitions \RevAdd{(qualitative changes in the set or stability of stationary states)} and metastability \RevAdd{(transient states that persist for a long time before converging to a stable state)}~\cite{carrillo2020long}.
    A classic illustration is the noisy Kuramoto model on $\T$, with $V\equiv0$ and $W(x)=-K\cos(x)$.  
    For $K \le 1$, the uniform density $\mu\equiv(2\pi)^{-1}$ is the unique (and stable) equilibrium.  
    At the critical coupling $K=1$, it loses stability, and for $K>1$, multiple nonuniform steady states bifurcate, all \RevAdd{spatial} translations \RevAdd{on the torus} of
    \begin{equation}
        \label{eq:steady_state_kuramoto}
        \bar\mu(\theta) =\frac{\exp\bigl(2Kr\cos\theta\bigr)}{\int_\T \exp\bigl(2Kr\cos\phi\bigr)\,d\phi}, \quad r = \Psi(2Kr),  \quad \Psi(\theta)=\frac{\int_\T\cos x e^{\theta\cos x}\,dx}{\int_\T e^{\theta\cos x}\,dx},
    \end{equation}
    while the uniform state remains a steady state but is now unstable~\cite{strogatz1991stability, bertini2010dynamical}.
\end{remark}
    
\subsection{Spectral analysis of the linearized dynamics}
\label{sec:operator}

We now analyze the spectral structure of the linearized McKean-Vlasov dynamics around a stationary distribution $\bar\mu$. 
By understanding the spectral properties, we formulate a feedback design that uses information about the modes requiring stabilization.

Fix a stationary state $\bar\mu$ \RevAdd{satisfying the hypotheses of \Cref{thm:equilibria}}, which will serve as the desired target distribution for the feedback control law developed in \Cref{sec:control}.
\RevAdd{Set $\Space \coloneqq L^2\bigl(\Omega,\bar\mu^{-1}\bigr)$} and define the operators, for all $y \in \Space$,
\RevAdd{\begin{equation}
  \label{eq:operators_mckean_vlasov}
  \A \mu \coloneqq \nabla \cdot (\mu \nabla V + \sigma \nabla \mu) \quad \text{and} \quad \W(\mu) \coloneqq \nabla \cdot (\mu\nabla W*\mu),
\end{equation}}%
so that equation~\eqref{eq:mv_wellposed} can be written as $\partial_t \mu = \A \mu + \W(\mu)$, where $\W$ is a nonlinear and nonlocal operator.
The domain of $\A$ is given by $D(\A)$, while the domain of $\W$ is given by $H^1(\Omega) \cap \Space = H^1(\Omega)$.
\RevAdd{Because $V \in W^{1,\infty}(\Omega)$ and $W \in W^{2,\infty}(\Omega)$, both $V$ and $W$ are continuous. Since $\bar\mu \in \Pac \subset L^1(\Omega)$, the convolution $W * \bar\mu$ is continuous on $\Omega$. Hence $\bar\mu \propto \exp(-(V + W*\bar\mu) / \sigma)$
is continuous and strictly positive on the compact domain $\Omega$. 
Therefore it is bounded above and below. 
Consequently,} the weighted and unweighted Sobolev spaces coincide as sets with equivalent norms.
In particular, $H^d(\Omega) \cap \Space = H^d(\Omega)$ for every nonnegative integer $d$.
Hence, well-posedness on $\Space$ immediately follows.

Denote by $\D_W$ the Fréchet derivative of $\W$ at $\bar\mu$, which can be computed as 
\[
\D_W y \coloneqq \nabla \cdot (y \nabla W * \bar\mu + \bar\mu \nabla W * y), \quad y \in \Space,
\]
by noticing that $\W(y + \bar\mu) = \D_W y + \W(y) + \W(\bar\mu)$.
\RevAdd{Then, the perturbation variable 
\begin{equation}
  \label{eq:y_definition}
  y \coloneqq \mu - \bar\mu,
\end{equation}
satisfies 
\begin{equation} 
  \label{eq:y_evolution}
  \partial_t y = \Lcal y + \W(y),
\end{equation}
where 
\begin{equation}
  \label{eq:linearized_operator}
  \Lcal \coloneqq \A + \D_W
\end{equation}}
is the linearized operator around $\bar\mu$.

\subsubsection*{Ground-state transformation}

We perform the (ground-state) unitary transformation \RevAdd{(see, e.g.,~\cite[Sec 4.9]{pavliotis2014stochastic} and~\cite{breiten2018control,Ris84, Stroock93, tomita1976eigenvalue})}, to map the Fokker-Planck operator $\Lcal$ to a {\em Schr\"odinger operator} $\Hcal$ defined in $L^2(\Omega)$.  
Specifically, define the {\em unitary operator}
\[
\unitary\colon \Space \longrightarrow L^2(\Omega), \qquad (\unitary y)(x)=\frac{y(x)}{\sqrt{\bar\mu(x)}}
\]
Verifying that $\unitary$ is an isometric isomorphism with a well-defined inverse is straightforward.
Then, set 
\[
\Hcal \coloneqq -\unitary \Lcal \,\unitary^{-1} \eqqcolon \Hcal_{\mathrm{loc}} + \K,
\]
where
\[
\Hcal_{\mathrm{loc}} = -\sigma\,\Delta + \Psi(x),
\]
is a Schrödinger operator with potential $\Psi(x)=\tfrac1{4\sigma}|\nabla v|^2 - \tfrac12\Delta v$ and $v[\bar\mu] = V + W*\bar\mu$, and $\K$ is the integral operator $\K[\phi](x) \coloneqq \int_{\Omega} k(x,y) \phi(y) \, dy$, where
\[
k(x,y) \coloneqq -\frac{\bar\mu(y)^{1/2}}{\bar\mu(x)^{1/2}} \Bigl(\nabla\bar\mu(x) \cdot \nabla W(x-y) + \bar\mu(x)\,\Delta W(x-y)\Bigr), \quad x,y\in\Omega.
\]

The operator 
\[
\Lcal_{\mathrm{loc}}y \coloneqq -\unitary^{-1} \Hcal_{\mathrm{loc}} \, \unitary y = \sigma \Delta y + \nabla \cdot \bigl(y \nabla v[\bar\mu]\bigr)
\]
is a standard linear Fokker-Planck operator with diffusion $\sigma > 0$ and \RevAdd{frozen} drift \RevAdd{field} $v[\bar\mu]$.
\RevAdd{Once the stationary state $\bar\mu$ is fixed, the same integration by parts argument as in the no-flux case (see \cite[Lem. 3.1]{breiten2018control}) shows that, on the compact manifold $\Omega$, the operator $\Lcal_{\mathrm{loc}}$ is self-adjoint on $\Space$ with compact resolvent.}
In particular, its spectrum is purely discrete and nonpositive, and zero is a simple eigenvalue.

\begin{lemma}[Spectral properties of $\Lcal_{\mathrm{loc}}$]
    \label{lem:L0-spectrum}
    The operator $\Lcal_{\mathrm{loc}}$ is self-adjoint with compact resolvent, $\sigma(\Lcal_{\mathrm{loc}}) \subseteq (-\infty, 0]$, and $\ker(\Lcal_{\mathrm{loc}})=\operatorname{span}\{\bar\mu\}$; $0$ is a simple eigenvalue and $\bar\mu$ is the unique (normalized) steady state.
\end{lemma}

In the following, the nonlocal part $\K$ is shown to be Hilbert-Schmidt and hence compact (see the \hyperref[proof:K-bounded]{proof} in the Appendix).  
By \Cref{lem:perturb-pure}, this implies that the full operator $\Hcal = \Hcal_{\rm loc} + \K$ also has compact resolvent and a purely discrete spectrum consisting only of eigenvalues of finite algebraic multiplicity accumulating only at $\infty$ in modulus.

\begin{lemma}\label{lem:K-bounded}
    The operator $\K$ is Hilbert-Schmidt, hence compact, on $L^2(\Omega)$.
\end{lemma}

\begin{remark}
    The kernel $k$ of $\K$ decomposes as
    \[
    k = k_{\mathrm{sym}} + k_{\mathrm{asym}}, \qquad 
    k_{\mathrm{asym}}(x,y) = \frac{[\bar\mu(x)\bar\mu(y)]^{1/2}}{\sigma}\, \bigl(\nabla v[\bar\mu](x)\cdot \nabla W(x-y)\bigr).
    \]
    As shown in \Cref{lem:spectrum_finite}, the operator norm of $\K_{\mathrm{asym}}$ bounds the imaginary parts of the spectrum of $\Hcal$.
    In particular, if $V \equiv 0$ and $\bar\mu$ is uniform, $\K$ is self-adjoint.
\end{remark}

\begin{proposition}[Discrete spectrum of $\Hcal$]\label{prop:pure-spectrum}
    The operator $\Hcal$ has compact resolvent. 
    Its spectrum consists of isolated (possibly complex) eigenvalues of finite multiplicity accumulating only at $\infty$ in modulus, i.e., $\sigma(\Hcal) = \{\lambda_n\}_{n \ge 1}$, $\lambda_n$ isolated, finite algebraic multiplicity, and $|\lambda_n| \to \infty$.
    Furthermore, there exist (generalized) eigenfunctions $\{\varphi_n\}$ of $\Hcal$ and corresponding eigenfunctions $\{\varphi_n^*\}$ of $\Hcal^*$ such that $\langle \varphi_n,\varphi_m^*\rangle = \delta_{nm}$, and the linear spans of $\{\varphi_n\}$ and $\{\varphi_n^*\}$ are both dense in $L^2(\Omega)$.
\end{proposition}

\begin{proof}
    On the compact torus $\T^d$, the Laplacian $-\Delta$ with domain $H^2(\Omega)$ is self-adjoint in $L^2(\Omega)$ and has compact resolvent, since $\sigma(-\Delta) = \{|k|^2 : k\in\mathbb{Z}^d\}$ is discrete and unbounded above.
    Adding the operator $\psi \mapsto \Psi \psi$ to $-\sigma \Delta$, with the bounded function $\Psi$, yields $\Hcal_{\mathrm{loc}}$, which remains self-adjoint on $H^2(\Omega)$ and still has compact resolvent by \Cref{lem:perturb-pure}.
    Finally, since $\K$ is compact by \Cref{lem:K-bounded}, the full operator $\Hcal$ still has compact resolvent. 
    \cite[Ch. III, Sec. 8, Thm. 6.29]{kato2013perturbation} concludes the theorem.
\end{proof}

Because $\Hcal$ is unitarily equivalent up to a minus sign to $\Lcal$, its spectrum is clearly the negative of that of $\Lcal$, i.e., $\sigma(\Hcal)=-\sigma(\Lcal)$.
Indeed, for any $\lambda \in \C$, 
\[
(\Hcal-\lambda I)\varphi=0 \iff (\Lcal + \lambda I)(\unitary^{-1}\varphi)=0.
\]
Moreover, mass conservation 
\[
0 = \int_\Omega \Lcal y\,dx = \langle (\A+\D_W)y,\bar\mu\rangle_{\bar\mu^{-1}}, \quad \forall y \in \Space
\]
implies $\Lcal^* \bar\mu = 0$.
Hence $\bar\mu$ is a nontrivial element of the kernel of the adjoint, and $0 \in \sigma(\Lcal)$.
Thus, we have proved:

\begin{proposition}\label{prop:spectrum_relation}
    The spectrum of the operator $\Hcal$ is the negative of the spectrum of the operator $\Lcal$, i.e., $\sigma(\Hcal) = -\sigma(\Lcal)$.
    Further, $\varphi$ is an eigenfunction of $\Hcal$ with eigenvalue $\lambda$ if and only if $\unitary^{-1} \varphi$ is an eigenfunction of $\Lcal$ with eigenvalue $-\lambda$.
    Finally, $0 \in \sigma(\Lcal)$.
\end{proposition}

Before turning to the control design, we state a spectral property that will be necessary in the proof of \Cref{thm:hautus}.  
Namely, although the full operator $\Hcal$ acts on an infinite-dimensional space, only finitely many of its eigenvalues can lie above any prescribed real level.
This ensures that only a finite number of modes need to be stabilized by a feedback control.

\begin{lemma}[Finiteness of spectrum above $-\delta$]\label{lem:spectrum_finite}
  For any $\delta>0$, the set
  \[
  \bigl\{\lambda\in\sigma(\Lcal) : \Re(\lambda) \ge -\delta\bigr\}
  \]
  is finite.
\end{lemma}

\begin{proof}
    Since $\Hcal_{\rm loc}$ is self-adjoint, bounded below, and has compact resolvent, its eigenvalues $\{\mu_n\}_{n \in \mathbb{N}}$ are real, of finite multiplicity, and satisfy $\mu_n \to+\infty$.
    By a standard perturbation argument,
    \[
    \sigma(\Hcal) \subseteq \bigcup_{n=1}^{\infty} \bar{D}(\mu_n, \|\mathcal{K}\|),
    \]
    where $D(a,r) = \{z \in \C:|z-a|<r\}$.
    Thus each $\lambda \in \sigma(\Hcal)$ satisfies $|\lambda - \mu_n| \le \|\mathcal{K}\|$ for some $n \in \mathbb{N}$, which implies $|\Im(\lambda)| \le \|\mathcal{K}\|$ and $\Re(\lambda) \ge \min_{n} \mu_n - \|\mathcal{K}\| \ge C > -\infty$ for a constant $C$.
    By \Cref{prop:pure-spectrum}, the eigenvalues of $\Hcal$ satisfy $|\lambda_k| = \sqrt{\Im(\lambda_k)^2 + \Re(\lambda_k)^2} \to \infty$ as $k\to \infty$, hence $\Re(\lambda_k) \to \infty$.
    Consequently, only finitely many $\eta \in \sigma(\Hcal)$ can satisfy \RevAdd{$\Re(\eta) \le \delta$}, and, therefore, only finitely many $\lambda \in \sigma(\Lcal)$ can satisfy $\Re(\lambda) \ge -\delta$.
\end{proof}

\begin{remark}
    Since $\mathcal{L}$ is the sum of $-\sigma \Delta$ and a bounded operator, 
    \[
    D(\Lcal) = D(-\sigma \Delta) = H^2(\Omega) \cap \Space = H^2(\Omega).
    \]
\end{remark}

From the spectral analysis presented in this section, it follows that the linearized McKean-Vlasov operator $\Lcal$ has a discrete spectrum with finitely many eigenvalues with real part larger than $-\delta$, which are the only modes that need to be controlled.
In the next section, we exploit this modal decomposition to formulate and solve a finite-dimensional Riccati equation for feedback stabilization.

\subsection{Decoupling via the zero-mean projection}
\label{subsec:projection}

The linear operator $\Lcal$ \RevAdd{defined in~\eqref{eq:linearized_operator}} has an eigenvalue at $0$ arising from mass conservation (see \Cref{prop:spectrum_relation}).
In particular, if $y(t,x)=\mu(t,x)-\bar\mu(x)$, then $\int_{\Omega}y(t,x)\,dx=0$ for all $t \ge 0$.
Consequently, instead of the full space $\Space$, our dynamics and control are restricted to the zero-mean subspace
\[
\Space_0 \coloneqq \left\{y \in \Space : \int_{\Omega} y(x) \,dx = 0 \right\}.
\]
In the sequel we introduce the projection $\Pcal$ onto $\Space_0$ so that 
$\Space_0 = \operatorname{im}(\Pcal)$, and derive the reduced dynamics on this subspace.

Define the projection on $\Space$
\[
\Pcal y \coloneqq  y - \left(\int_{\Omega}y(x) \, dx\right) \bar\mu,
\] 
so that $\Space_0 \coloneqq \operatorname{im}(\Pcal)$, and decompose $y$ via the projections $\Pcal$ and $\Qcal \coloneqq I - \Pcal$, i.e.,
\[
y_p \coloneqq \Pcal y, \quad y_q \coloneqq \Qcal y, \quad \Space = \operatorname{im}(\Pcal) \oplus \operatorname{im}(\Qcal).
\]
Applying $\Pcal$ and $\Qcal$ to the evolution PDE \RevAdd{for $y$ given in~\eqref{eq:y_evolution}}
\[
\partial_t y_p + \partial_t y_q = \Lcal(y_p + y_q) + \W(y_p + y_q)
\]
yields 
\[
\partial_t y_q = 0 \implies y_q(t) = y_q(0) = \Qcal\mu_0 - \bar\mu
\]
since $\Qcal \Lcal = 0$ and $\Qcal \W = 0$, and 
\[
\partial_t y_p = \Pcal\Lcal y_p + \Pcal \W(y_p), \qquad y_p(0) = \Pcal\mu_0.
\]
Given that $\Pcal\Lcal = \Lcal$ on $H^2(\Omega)$ and $\Pcal\W = \W$ on $H^1(\Omega)$,
\[
\partial_t y_p = \Lcal_0\,y_p + \W_0(y_p), \qquad y_p(0)= \Pcal\mu_0,
\]
where
\[
\Lcal_0 = \Lcal I_{\Pcal}, \qquad \W_0(y_p) = \W(I_{\Pcal}y_p), 
\]
and $I_{\Pcal} : \Space_0 \to \Space$ is the canonical inclusion.

\begin{remark}\label{remark:kernel_L0}
    By working on the space $\Space_0 = \operatorname{im}(P)$, we have that, for all $f \in \Space_0$,
    \[
    \langle f, \bar\mu \rangle_{\bar\mu^{-1}} = \int_{\Omega} f(x) \, dx = 0
    \]
    so $\bar\mu \in \Space_0^{\perp}$.
    The eigenfunctions of $\Lcal_0$ cannot belong to $\operatorname{span}(\bar\mu)$.
\end{remark}

Under the ground-state transform $\unitary$, the zero-mean projection $\Pcal$ on $\Space$ induces the orthogonal projector
\[
\Pcal_H  = \unitary \Pcal\,\unitary^{-1} \colon L^2(\Omega) \to \{\sqrt{\bar\mu}\}^\perp.
\]
Since $\Pcal$ commutes with $\Lcal$ (mass conservation), it follows that $\Pcal_H$ commutes with $\Hcal$.  
Hence the reduced Schrödinger operator
\[
\Hcal_0  = \Pcal_H \Hcal
\]
is well-defined on the codomain $\operatorname{im}(\Pcal_H) = \{\sqrt{\bar\mu}\}^\perp \subseteq L^2(\Omega)$, and corresponds to the projection $\Pcal \Lcal$ on the weighted side.

\section{Control Problem Formulation and Riccati Feedback}
\label{sec:control}

In this section, we pose the feedback-stabilization problem for the McKean-Vlasov dynamics on the torus and state the Riccati-based control design using the linearization step discussed in the previous section.

\subsection{Formulation}

We steer equation~\eqref{eq:mv_wellposed} by adding a time-dependent potential of the form
\[
 V(x) \mapsto V(x) + \sum_{j=1}^m u_j(t) \alpha_j(x),
\]
where each $\alpha_j\in W^{1,\infty}(\Omega) \cap W^{2,p}(\Omega)$ is a prescribed {\em shape function}, which will be chosen ad hoc to meet the stability condition, and $u_j(t) \in \R$ are scalar controls.
\RevAdd{The admissible class for $u = (u_1,\dots,u_m)$ is introduced below in the linear-quadratic formulation.}
\RevAdd{Using the operators $\A$ and $\W$ defined in \eqref{eq:operators_mckean_vlasov}, the} controlled PDE becomes
\begin{equation}\label{eq:controlled_mckean_vlasov}
  \partial_t\mu = \A\mu + \W(\mu) + \sum_{j=1}^m u_j(t)\,\Nop_j\mu,
\end{equation}
where $\Nop_j : D(\Nop_j) \subseteq \Space \to \Space$ is defined as $\Nop_j \mu \coloneqq \nabla \cdot (\mu \nabla \alpha_j)$ with $D(\Nop_j) = H^1(\Omega)$.
The equation for \RevAdd{the perturbation $y$~\eqref{eq:y_definition}} around a steady state $\bar\mu$ is
\begin{equation}\label{eq:abstract_formulation}
  \partial_t y = \Lcal y + \sum_{j=1}^m u_j(t) \Nop_j y + \sum_{j=1}^m \B_j u_j(t) + \W(y),
\end{equation}
where \RevAdd{$\Lcal$ is defined in~\eqref{eq:linearized_operator} and} $\B_j \coloneqq \Nop_j\bar\mu$.
The first order approximation of this equation is given by
\begin{equation}\label{eq:linearized_formulation}
  \partial_t y = \Lcal y + \sum_{j=1}^m \B_j u_j(t).
\end{equation}

We apply the zero-mean projection $\Pcal$ onto $\Space_0$ and define
\[
  \Nop_{0,j} \coloneqq \Nop_j I_{\Pcal}, \qquad \B_{0,j} \coloneqq \B_j \quad j=1,\dots,m
\]
to obtain operators defined on $\Space_0$. 
Since $\Pcal \Nop_j = \Nop_j$, $\Pcal \B_j = \B_j$, $\Qcal \Nop_j = 0$, and $\Qcal \B_j = 0$, on the respective domains, the projected dynamics read as
\[
\partial_t y_p = \mathcal{L}_0 y_p + \sum_{j=1}^m u_j(t) \Nop_{0,j} y_p + \sum_{j=1}^m \B_{0,j} u_j(t) + \W_0(y).
\]

\noindent\textbf{Notational convention.} From now on we write
\[
y \equiv y_p,\quad \Lcal \equiv \Lcal_0,\quad \W \equiv \W_0, \quad \Hcal_0 \equiv \Hcal, \quad \Nop_j \equiv \Nop_{0,j}, \quad \B_{0,j} \equiv \B_j
\]
with the understanding that all of these objects act on the zero-mean space $\Space_0$, or, after the ground-state transform, on $L_0^2(\Omega)$.

\subsection{Linear-quadratic control problem}

\RevAdd{For a prescribed target exponential decay rate $\delta > 0$} and a self-adjoint, positive operator $\Mcal : \Space_0 \to \Space_0$, define the cost
\begin{equation}\label{eq:cost}
  \mathcal J(y,u) \coloneqq \frac12\int_0^\infty e^{2\delta t} \Bigl(\langle y(t),\Mcal y(t)\rangle_{\bar\mu^{-1}} + |u(t)|_2^2\Bigr) \, dt.
\end{equation}
\RevAdd{The exponential weight reflects the goal of achieving stabilization with decay rate at least $\delta$.
Accordingly, the admissible control space is
\[
L^2_\delta(0,\infty;\mathbb R^m) \coloneqq \left\{u:(0,\infty)\to\mathbb R^m : \int_0^\infty e^{2\delta t}|u(t)|_2^2 \, dt < \infty\right\}.
\]}%
Under the (zero-mean) linearized dynamics~\eqref{eq:linearized_formulation}, we set $z(t) \coloneqq  e^{\delta t}y(t)$, $v_j(t) \coloneqq  e^{\delta t}u_j(t)$, obtaining the standard linear-quadratic problem
\begin{align}
  \min_{v \in L^2(0,\infty; \R^m)} \;& \frac12\int_0^\infty \bigl(\langle z(t),\M z(t)\rangle_{\bar\mu^{-1}} + |v(t)|_2^2\bigr)\,dt, \label{eq:cost_transformed}
\\
  &\partial_t z = (\Lcal + \delta I) z + \B v, \label{eq:equation_transformed}
\end{align}
where $\B u \coloneqq \sum_{j=1}^{m}\mathcal B_j u_j\;(t)$.

\begin{theorem}[Hautus $\delta$-stabilization]
  \label{thm:hautus}
  Fix $\delta>0$.  
  Assume the first $m$ eigenpairs $\{(\lambda_j,\psi_j)\}_{j=1}^m$ of $\Lcal$ satisfy $\Re\lambda_j \ge -\delta$.  
  If the shape functions $\alpha_j$ solve
  \begin{equation}
      \label{eq:B_definition}
      \nabla \cdot\bigl(\bar\mu \nabla\alpha_j\bigr) = \psi_j, \quad j=1,\dots,m,
  \end{equation}
  with periodic boundary conditions, then $(\Lcal, \B)$ satisfies the infinite-dimensional Hautus test and is $\delta$-stabilizable.
\end{theorem}

\begin{proof}
    By \Cref{lem:spectrum_finite}, there are finitely many eigenvalues of $\Lcal$ in the half-plane $\Re\lambda\ge -\delta$, namely $\{\lambda_1,\dots,\lambda_m\}$.
    Set $\beta_j = \unitary^{-1} \alpha_j$.
    \RevAdd{Applying} the ground-state transform and rewriting in Schrödinger form, equation~\eqref{eq:B_definition} becomes
    \begin{equation}
        \label{eq:B_definition_equiv}
        \Hcal_{\mathrm{loc}} \beta_j = -\sigma \unitary \psi_j = -\sigma \varphi_j,
    \end{equation}
    where $\varphi_j$ denotes the corresponding eigenfunction of the operator $\Hcal = \Hcal_{\mathrm{loc}} + \mathcal{K}$.
    %This equation has a solution $\beta_j \in W^{2,p}(\Omega)$ for any $p > 0$ if and only if the right-hand side is orthogonal to the kernel of $\Hcal_{\mathrm{loc}}^\dagger = \Hcal_{\mathrm{loc}}$ given that $\Hcal_{\mathrm{loc}}$ is self-adjoint over $L^2(\Omega)$.
    On the torus, $\Hcal_{\rm loc}$ is self-adjoint on $D(\Hcal_{\rm loc}) = H^2(\Omega)$ and has closed range. 
    By the Fredholm alternative, ~\eqref{eq:B_definition_equiv} is solvable iff the right-hand side is orthogonal to the kernel of $\Hcal_{\mathrm{loc}}^\dagger = \Hcal_{\mathrm{loc}}$. 
    Under this condition, there exists $\beta_j \in W^{2,p'}(\Omega)$ for any $p' > 0$.
    In particular, $\beta_j \in W^{1,\infty}(\Omega) \cap W^{2,p}(\Omega)$~\cite[Sec. 6.2]{evans2022partial}.
    By \Cref{lem:L0-spectrum}, we know that
    \[
    \ker(\Hcal_{\mathrm{loc}}^\dagger) = \ker(\Hcal_{\mathrm{loc}}) = \operatorname{span}\{ \sqrt{\bar\mu} \}.
    \]
    Therefore, the equation~\eqref{eq:B_definition_equiv} is solvable if and only if \RevAdd{$\varphi_j \perp \operatorname{span}\{\sqrt{\bar\mu}\}$}, or equivalently, \RevAdd{$\psi_j \perp \operatorname{span}\{\bar\mu\}$}, which is true by \Cref{remark:kernel_L0}.
    Additionally, the solution is unique up to a constant, which we do not define since we are only interested in its derivative.

    Let $\{(\lambda_i,\phi_i)\}_{i=1}^m$ be the corresponding eigenpairs of \(\Lcal^*\), chosen so that 
    \[
    \langle\psi_j,\phi_i\rangle_{\bar\mu^{-1}} = \delta_{ij}.
    \]
    Then, we compute
    \[
    \B^* \phi_i = \Bigl(\langle \psi_1, \phi_i \rangle_{\bar\mu^{-1}}, \cdots, \langle \psi_m, \phi_i \rangle_{\bar\mu^{-1}}\Bigr) = e_i \neq 0,
    \]
    where $\{e_i\}_{i=1}^m \subseteq \R^m$ is the canonical basis.
    So, no nonzero vector in 
    \[ 
    \bigoplus_{\Re(\lambda) \ge -\delta} \operatorname{ker}(\lambda I - \Lcal^*)
    \]
    lies in $\operatorname{ker}(\B^*)$. 
    This is exactly the infinite-dimensional Hautus condition.
    Since we assumed $\Mcal$ to be positive definite, the pair $(\Lcal, \Mcal)$ is trivially detectable, and the proof follows from \RevAdd{\cite[Thm. 5.2.11]{curtain2012introduction}}.
\end{proof}

\subsection{Algebraic Riccati equation and feedback law}
\label{subsec:ARE}

We now derive the Riccati equation whose solution yields the optimal control for problem~\eqref{eq:cost_transformed}-\eqref{eq:equation_transformed}.
Recall that $\Lcal_0 = \A_0+\D_{0,W}$ acts on the zero-mean space $\Space_0$, and $\B_0$ is the control operator on $\Space_0$.  
We consider two equivalent algebraic Riccati equations:
\begin{align}
  &(\Lcal_0^* + \delta I) \Pi_0  + \Pi_0 (\Lcal_0 + \delta I) - \Pi_0 \B_0 \B_0^* \Pi_0 + \Mcal = 0, \quad \Pi_0 = \Pi_0^* \in \mathbb{L}(\Space_0),
  \tag{$R_0$}\label{eq:riccati_reduced} \\[6pt]
  &(\Pcal^*\Lcal^* + \delta \Pcal^*) \Pi + \Pi(\Lcal\Pcal + \delta \Pcal) - \Pi \B\B^* \Pi +\Pcal^* \M \Pcal = 0, \quad \Pi = \Pi^*  \in \mathbb{L}(\Space),
  \tag{$R$}\label{eq:riccati_full}
\end{align}
where equation~\eqref{eq:riccati_reduced} is posed directly on the zero-mean subspace, while \eqref{eq:riccati_full} is its extension to the full space.

A standard decomposition argument shows that any solution $\Pi$ of~\eqref{eq:riccati_full} restricts to a solution of~\eqref{eq:riccati_reduced} via $\Pi_0 = I_{\Pcal}^* \Pi I_{\Pcal}$ and conversely any solution $\Pi_0$ of~\eqref{eq:riccati_reduced} lifts to a family of solutions of \eqref{eq:riccati_full} by \RevAdd{$\Pi = \Pcal^*\Pi_0 \Pcal + \gamma \Qcal$}, for all $\gamma \in \R$.
In either case, the optimal feedback
\[
v(t) = -\B^* \Pi z(t) \quad \Longrightarrow \quad u(t)=e^{-\delta t}v(t) = -\B^* \Pi y(t)
\]
is the same since \RevAdd{$\B^* \Qcal = 0$}.
For more details, we refer to~\cite[Sec. 4.3]{breiten2018control}.

We conclude by stating the connection between the Riccati equation and the optimal control problem~\eqref{eq:cost_transformed}-\eqref{eq:equation_transformed}, which is guaranteed by \Cref{thm:hautus} coupled with \RevAdd{\cite[Thm. 5.2.11]{curtain2012introduction}}.

\begin{proposition}[Algebraic Riccati Equation]\label{prop:ARE}
  Under \Cref{thm:hautus}'s hypotheses, there exists a unique self-adjoint, nonnegative $\Pi \in \mathbb{L}(\Space_0)$ solving \eqref{eq:riccati_reduced}.
  The resulting feedback law
  \[
  v(t) = -\B^* \Pi z(t), \quad u(t) = e^{-\delta t}v(t) = -\B^* \Pi y(t).
  \]
  is optimal, and the operator $\Lcal_{\Pi} \coloneqq \Lcal + \delta I - \B \B^* \Pi$ generates an exponentially stable semigroup on $\Space_0$.
\end{proposition}

\subsection{Local exponential stabilization}
\label{subsec:local}

After computing the optimal control in feedback form $v(t) = -\B^* \Pi z(t)$ for the linearized model, we insert it into the full nonlinear dynamics to obtain the closed-loop equation
\[
\dot{z} = \Lcal_{\Pi}z - (\B^* \Pi z)\Nop^{\delta}(t) z  + \W^{\delta}(t)(z),
\]
where $\Nop^{\delta}(t)z = (\Nop_1^{\delta}(t)z, \cdots,\Nop_m^{\delta}(t)z)$, $\Nop_j^{\delta}(t) = e^{-\delta t}\Nop_j$ and $\W^{\delta}(t) = e^{-2\delta t}\W$.
By doing the change of variables $\psi = \unitary z$, we obtain the equation in $L^2(\Omega)$:
\begin{equation}\label{eq:transformed_mckean_vlasov}
    \dot\psi = -\Hcal_{\Pi}\psi - (\hat{\B}^* \hat{\Pi}\psi)\hat{\Nop}^{\delta}(t) \psi + \hat{\W}^{\delta}(\psi),
\end{equation}
where
\[
\Hcal_{\Pi} = \Hcal - \delta I + \hat\B \hat\B^{*} \hat\Pi, \quad \hat\B=\unitary \B, \quad \hat\Pi=\unitary\,\Pi\,\unitary^{-1},
\]
\[
\hat\Nop^{\delta}_j(t) = \unitary \Nop^{\delta}_j\,\unitary^{-1}, \qquad \hat\W^{\delta} = \unitary\,\W^{\delta}\,\unitary^{-1}.
\]
By \Cref{prop:ARE} and \Cref{lem:K-bounded}, both $\hat\B \hat\B^{*} \hat\Pi$ and the integral perturbation $\K$ are bounded on $L^2(\Omega)$.  
Hence
\begin{equation}
  \label{eq:Hpi-domain}
  D(\Hcal_{\Pi}) = D(\Hcal_{\mathrm{loc}})\,\cap\,L^2_0(\Omega), \quad D(\Hcal_{\Pi}^*) = D(\Hcal_{\mathrm{loc}}^*)\,\cap\,L^2_0(\Omega).
\end{equation}

Denote $H_{\sharp}(\Omega) \coloneqq  H^1(\Omega)\cap L^2_0(\Omega)$ and notice that $[H_{\sharp}(\Omega)]^* \cong \{F\in [H^1(\Omega)]^*: \langle F,1\rangle_{[H^{1}(\Omega)]^*,H^{1}(\Omega)}=0\}$.
By standard real-interpolation theory for sectorial operators, taking $\lambda \in \rho(\Hcal_{\mathrm{loc}})$ and setting $\Hcal_{\mathrm{loc},\lambda} \coloneqq \lambda I - \Hcal_{\mathrm{loc}}$, 
\[
[D(\Hcal_{\mathrm{loc},\lambda})\cap L^2_0(\Omega), L^2_0(\Omega)]_{1/2} = H^1_{\sharp}(\Omega) = [D(\Hcal_{\mathrm{loc}, \lambda}^*)\cap L^2_0(\Omega), L^2_0(\Omega)]_{1/2},
\]
since the fractional powers of $\Hcal_{\mathrm{loc},\lambda}$ are well-defined.
Moreover, by~\cite[Vol. 1, Sec. 12]{lions2012non}, we have
\[
[[D(\Hcal_{\lambda}), L^2_0(\Omega)]_{1/2}, [D(\Hcal_{\lambda}^*), L^2_0(\Omega)]_{1/2}^*]_{1/2} = L^2_0(\Omega).
\]
Defining the evolution space
\begin{equation}\label{e:evol_space}
W(0, \infty; \Omega) \coloneqq L^2\bigl(0,\infty; H^1_{\sharp}(\Omega)\bigr) \cap H^1\bigl(0,\infty; [H^1_{\sharp}(\Omega)]^*\bigr),
\end{equation}
equipped with the norm
\[
\|y\|_{W(0, \infty; \Omega)} \coloneqq \sqrt{\int_{0}^{\infty} \|y(t)\|_{H^1(\Omega)}^2
    +\|\partial_t y(t)\|_{[H^1(\Omega)]^*}^2
    \,dt},
\]
one obtains the following maximal-regularity result.

\begin{theorem}[Maximal regularity]
    \label{thm:max-reg-Hpi}
    Let $f \in L^2(0,\infty; [H_{\sharp}(\Omega)]^*)$ and $\psi_0 \in L^2_0(\Omega)$.
    Then the Cauchy problem
    \[
    \begin{cases}
      \dot\psi(t) + \Hcal_{\Pi} \, \psi(t) = f(t), & t>0,\\
      \psi(0)=\psi_0,
    \end{cases}
    \]
    admits a unique mild solution $\psi \in W(0, \infty; \Omega)$ satisfying
    \[
    \|\psi\|_{W(0, \infty; \Omega)} \le C\Bigl(\|f\|_{L^2(0,\infty;[H^1(\Omega)]^*)} + \|\psi_0\|_2\Bigr)
    \]
    \RevAdd{for some constant $C>0$}.
    In particular, $\psi \in C_b\bigl([0,\infty);L^2_0(\Omega)\bigr)$.
\end{theorem}

\begin{proof}
  Write 
  \[
    \Hcal_{\Pi} = A + B, \quad A \coloneqq \Hcal_{\mathrm{loc},\lambda}, \quad B \coloneqq -\bigl(\delta I-\hat\B\hat\B^*\hat\Pi+\K\bigr),
  \]
  acting on $L^2_0(\Omega)$.  
  By \Cref{prop:pure-spectrum}, $A$ is sectorial with angle $0$, has compact resolvent and $D(A) = D(\Hcal_{\mathrm{loc}}) \cap L^2_0(\Omega)$.  
  Since $B \in \mathbb{L}(L^2_0(\Omega))$, and by \eqref{eq:Hpi-domain} we have the interpolation identity
  \[
  [D(A),L^2_0(\Omega)]_{1/2} = H_{\sharp}^1(\Omega) =[D(A^*),L^2_0(\Omega)]_{1/2},
  \]
  so that the operator $A+B$ meets the hypotheses of the maximal-$L^2$-regularity theorem (\cite[Ch. 3, Thm. 2.2]{bensoussan2007representation}; see also \cite[I.4.2]{lions2012non}).
  The conclusion follows immediately.
\end{proof}

The closed-loop evolution in equation~\eqref{eq:transformed_mckean_vlasov} contains two nonlinear (quadratic-type) perturbations:
\[
  F_1(\psi)(t) \coloneqq \bigl(\hat\B^*\hat\Pi\psi(t)\bigr)\hat\Nop^{\delta}(t)\psi(t),
\qquad
  F_2(\psi)(t) \coloneqq \hat\W^{\delta}\bigl(\psi(t)\bigr).
\]
In order to run the contraction argument in $W(0, \infty; \Omega)$, we estimate each one separately.
The next lemma bounds the {\em feedback term} $F_1$; its proof is essentially identical to \cite[Lem.~4.7]{breiten2018control}, so we only state it here.

\begin{lemma}[Feedback-term estimate]
  \label{lem:feedback-nonlinear}
  There exists $\widetilde C_1>0$ such that for all $y, z\in W(0, \infty; \Omega)$,
  \[
    \|F_1(y) - F_1(z)\|_{L^2(0,\infty;\,[H^1(\Omega)]^*)} \le \widetilde C_1 (\|y\|_{W(0, \infty; \Omega)} + \|z\|_{W(0, \infty; \Omega)}) \|y - z\|_{W(0, \infty; \Omega)}.
  \]
\end{lemma}

We now propose an estimate for $\hat{\W}^{\delta}$ of the same type.

\begin{lemma}[Nonlinear remainder estimate]\label{lem:remainder}
There exists \RevAdd{$\widetilde{C}_2 > 0$} such that for all $y, z\in W(0, \infty; \Omega)$,
\[
\|F_2(y)-F_2(z)\|_{L^2(0,\infty;[H^1(\Omega)]^*)} \le \RevAdd{\widetilde{C}_2}(\|y\|_{W(0, \infty; \Omega)} + \|z\|_{W(0, \infty; \Omega)}) \|y-z\|_{W(0, \infty; \Omega)}.
\]
\end{lemma}

\begin{proof}
  Set $u=\unitary^{-1}y$, $v=\unitary^{-1}z$.  Then
  \[
    \hat\W^{\delta}(y)-\hat\W^{\delta}(z) = e^{-2\delta t} \unitary\Bigl[\nabla \cdot \bigl(u \nabla W * (u-v)\bigr) +\nabla \cdot \bigl((u-v)\nabla W * v\bigr)\Bigr].
  \]
  Since $e^{-2\delta t} \le 1$, 
  \begin{align*}
    \|\hat{\W}^{\delta}(y) - \hat{\W}^{\delta}(z) \|_{L^2(0,\infty; [H^1(\Omega)]^*)} \le \; \|\unitary &\nabla \cdot (u \nabla W * (u - v))\|_{L^2(0,\infty; [H^1(\Omega)]^*)} \\[1ex]
    &+ \|\unitary \nabla \cdot ((u-v) \nabla W * v) \|_{L^2(0,\infty; [H^1(\Omega)]^*)}
  \end{align*}

  Consider the first term.
  For a fixed $t$ and any $\varphi \in H^1(\Omega)$ such that $\|\varphi\|_{H^1}=1$,
  \begin{align*}
    \bigl\langle\unitary\nabla \cdot(u\nabla W * (u-v)), \varphi\bigr\rangle &= \bigl\langle\nabla \cdot(u\nabla W * (u-v)), \unitary^*\varphi \bigr\rangle_{\bar\mu^{-1}} \\
    &= \int_\Omega \nabla \cdot\bigl(u\nabla W * (u-v)\bigr) \sqrt{\bar\mu} \frac{\varphi}{\bar\mu} \, dx \\
    &= -\int_\Omega u \nabla W * (u-v) \cdot \nabla \Bigl(\frac{\varphi}{\sqrt{\bar\mu}} \Bigr) dx \\
    &= -\int_\Omega u \nabla W*(u-v)\cdot \Bigl[\nabla\varphi +\nabla \bigl(\tfrac{V+W*\bar\mu}{2\sigma}\bigr) \varphi\Bigr] \frac{1}{\sqrt{\bar\mu}} \, dx \\
    &= -\int_\Omega \frac{u}{\sqrt{\bar\mu}} \frac{\nabla W*(u-v)}{\sqrt{\bar\mu}}\cdot \sqrt{\bar\mu}\Bigl[\nabla\varphi +\nabla \bigl(\tfrac{V+W*\bar\mu}{2\sigma}\bigr) \varphi\Bigr] \, dx.
  \end{align*}
  Hence, by Cauchy-Schwarz,
  \[
    \bigl\|\unitary\nabla \cdot(u\nabla W * (u-v))\bigr\|_{[H^1(\Omega)]^*} \le \left\|y \tfrac{\nabla W * (u-v)}{\sqrt{\bar\mu}}\right\|_{2}
    \,\left\|\sqrt{\bar\mu}\nabla\varphi
     +\sqrt{\bar\mu}\nabla\!\bigl(\tfrac{V+W*\bar\mu}{2\sigma}\bigr)\,\varphi
    \right\|_2.
  \]
  \RevAdd{Since $\bar\mu$ is continuous and strictly positive on the compact domain $\Omega$, it is bounded above and below.}
  Therefore,
  \[
  \left\|\sqrt{\bar\mu}\nabla\varphi + \sqrt{\bar\mu}\nabla \bigl(\tfrac{V+W*\bar\mu}{2\sigma}\bigr)\varphi \right\|_2 \le C_1\|\varphi\|_{H^1(\Omega)} = C_1,
  \]
  for some constant $C_1 > 0$, and, by Young's inequality,
  \begin{equation*}
      \begin{split}
            \left\|y \tfrac{\nabla W*(u-v)}{\sqrt{\bar\mu}} \right\|_2 &\le \|y\|_2 \left\|\tfrac{\nabla W*(u-v)}{\sqrt{\bar\mu}}\right\|_{\infty} \\
            &\le \|y\|_2 \left\|\tfrac{1}{\sqrt{\bar\mu}}\right\|_{\infty} \bigl\|\nabla W\bigr\|_{\infty} \bigl\|u-v\bigr\|_{1} \\
            &\le C_W \|y\|_2 \left\|\tfrac{1}{\sqrt{\bar\mu}}\right\|_{\infty} \bigl\|y-z\bigr\|_2,
      \end{split}
  \end{equation*}
  where \RevAdd{$C_W := \|W\|_{W^{2,\infty}(\Omega)}$, and} we use Cauchy-Schwarz to obtain 
  \[
  \|u-v\|_{1} = \int_{\Omega} |\sqrt{\bar\mu}(y-z)| \le \|y-z\|_2.
  \]
  \RevAdd{Since $\bar\mu$ is bounded below, $\|1/\sqrt{\bar\mu}\|_{\infty} < \infty$, and there exists a constant $C_* > 0$ such that}
  \[
  \|\unitary\nabla \cdot(u \nabla W * (u-v)) \|_{[H^1(\Omega)]^*} \le \RevAdd{C_*}\|y\|_2\|y-z\|_2.
  \]
  Integrating in time gives
  \[
  \|\unitary\nabla \cdot(u\nabla W*(u-v))\|_{L^2(0, \infty; [H^1(\Omega)]^*)} \le \RevAdd{C_*}\|y\|_{L^\infty(0,\infty; L^2(\Omega))} \|y-z\|_{L^2(0,\infty; L^2(\Omega))}.
  \]
    For the second term, we repeat the previous estimate with the substitutions $u' \coloneqq  u-v$, $v' \coloneqq  u-2v$.
    Because $u'-v'=v$, the expression $\nabla \cdot((u-v) \nabla W * v)$ is the same as $\nabla \cdot ( u' \nabla W * (u'-v'))$, \RevAdd{so} applying the bound proved for
    the first term with \((u',v')\) in place of \((u,v)\) therefore yields
    \[
    \|\unitary\nabla \cdot((u-v)\nabla W * v)\|_{L^2(0,\infty; [H^1(\Omega)]^*)} \le \RevAdd{C_*} \|z\|_{L^\infty(0, \infty; L^2(\Omega))}\|y-z\|_{L^2(0, \infty; L^2(\Omega))}.
    \]

  The continuous embeddings 
  \[
  W(0, \infty; \Omega) \hookrightarrow C_b([0,\infty);L^2(\Omega)) \text{ and } W(0, \infty; \Omega) \hookrightarrow L^2(0,\infty;L^2(\Omega))
  \] 
  from~\cite[Vol. I, Thm. 4.2]{lions2012non} imply the existence of a constant $C' > 0$ such that 
  \[
  \|y\|_{L^\infty(0, \infty; L^2(\Omega))} \le C' \|y\|_{W(0,\infty; \Omega)}, \quad \|y-z\|_{L^2(0, \infty; L^2(\Omega))} \le C' \|y-z\|_{W(0,\infty; \Omega)}
  \]
  so that the lemma follows after \RevAdd{setting $\widetilde{C}_2 \coloneqq  C_*(C')^2$}.
\end{proof}

Combining the three preceding estimates contributes to the concluding theorem.

\begin{theorem}[Existence via contraction]\label{thm:nonlin-existence}
    Let $C > 0$, $\widetilde{C}_1 > 0$, and $\widetilde{C}_2 > 0$ be the constants from \Cref{thm:max-reg-Hpi}, \Cref{lem:feedback-nonlinear}, and \Cref{lem:remainder}, respectively.  
    If the initial data $\psi_{0} \in L^{2}(\Omega)$ satisfies
    \[
    \|\psi_{0}\|_2 \le \frac{3}{16\,C^{2}\, (\widetilde{C}_1 + \widetilde{C}_2)},
    \]
    then the closed-loop equation~\eqref{eq:transformed_mckean_vlasov} with initial condition $\psi(0)=\mathcal{P}\psi_{0}$ admits a unique solution $\psi\in W(0, \infty; \Omega)$ satisfying
    \[
      \|\psi\|_{W(0, \infty; \Omega)} \le \frac{1}{4\,C\,(\widetilde{C}_1 + \widetilde{C}_2)}.
    \]
\end{theorem}

\begin{proof}
    We follow the proof from \cite[Thm. 4.8]{breiten2018control} closely.
    Define the map
    \[
      \mathcal{F}: B \to B, \quad B \coloneqq \Bigl\{ w \in W(0, \infty; \Omega) : \|w\|_{W(0, \infty; \Omega)} \le \tfrac1{4C(\widetilde{C}_1 + \widetilde{C}_2)}\Bigr\},
    \]
    by letting $\psi = \mathcal{F}(w)$ be the unique mild solution of
    \[
      \partial_t \psi + \Hcal_{\Pi} \psi = -(\hat{\B}^* \hat{\Pi}w)\hat{\Nop}^{\delta}(t) w + \hat{\W}^{\delta}(w),
      \quad
      \psi(0)=\mathcal{P}\psi_{0}.
    \]
    Denote $F \coloneqq -F_1 + F_2$ so that $\partial_t \psi + \Hcal_{\Pi} = F$.
    By \Cref{lem:feedback-nonlinear} and \Cref{lem:remainder},
    \begin{equation*}
        \begin{split}
            \|F(w)\|_{L^2(0,\infty; [H^1(\Omega)]^*)} &\le \widetilde C_1 \|w\|_{W(0, \infty; \Omega)}^2 + \widetilde C_2\|w\|_{W(0, \infty; \Omega)}^2 \\
            &= (\widetilde C_1 + \widetilde C_2)\|w\|_{W(0, \infty; \Omega)}^2 \\
            &\le \frac{1}{16 C^2 (\widetilde C_1 + \widetilde C_2)}.
        \end{split}
    \end{equation*}
    Hence \Cref{thm:max-reg-Hpi} yields for the mild solution $\psi=\mathcal F(w)$, the inequality
    \[
    \|\psi\|_{W(0, \infty; \Omega)} \le C \|F(w)\|_{L^2(0,\infty;[H^1(\Omega)]^*)} + C \|\psi_0\|_2 \le \frac1{16C(\widetilde C_1+\widetilde C_2)} + C \|\psi_0\|_2.
    \]
    By the assumption of the initial data, $\|\mathcal{F}(w)\|_{W(0, \infty; \Omega)} \le 2/(16\,C(\widetilde C_1+\widetilde C_2))$.
    Thus $\mathcal F$ maps $B$ into itself.

    We next verify that $\mathcal{F}$ is a contraction.
    To check the contraction, let $w_i \in B$ and $\psi_i = \mathcal{F}(w_i)$ for $i=1,2$.
    Then,
    \[
      \partial_t\psi_1-\partial_t\psi_2 + \Hcal_{\Pi}(\psi_1-\psi_2)
      = F(w_1) - F(w_2),
      \quad (\psi_1-\psi_2)(0)=0,
    \]
    so again by \Cref{thm:max-reg-Hpi}, \Cref{lem:feedback-nonlinear}, and \Cref{lem:remainder}
    \begin{equation*}
        \begin{split}
        \|\psi_{1}-\psi_{2}\|_{W(0, \infty; \Omega)} &\le C \Bigl(
          \Big\| (\hat{\B}^* \hat{\Pi}w_2)\hat{\Nop}^{\delta}(t) w_2 -(\hat{\B}^* \hat{\Pi}w_1)\hat{\Nop}^{\delta}(t) w_1 \Big\|_{L^{2}(0,\infty;[H^{1}(\Omega)]^*)} \\&\hspace{3.7cm} + \Big\| \hat{\W}^{\delta}(w_1) - \hat{\W}^{\delta}(w_2)\Big\|_{L^{2}(0,\infty;[H^{1}(\Omega)]^*)} \Bigr) \\
          &\hspace{-4mm}\le C(\widetilde{C}_1 + \widetilde{C}_2)\bigl(\|w_1\|_{W(0, \infty; \Omega)}+\|w_2\|_{W(0, \infty; \Omega)}\bigr)\,\|w_1-w_2\|_{W(0, \infty; \Omega)} \\
          &\hspace{-4mm}\le C(\widetilde{C}_1 + \widetilde{C}_2) \frac{2}{4C(\widetilde{C}_1 + \widetilde{C}_2)}\|w_1-w_2\|_{W(0, \infty; \Omega)} \\
          &\hspace{-4mm}= \frac{1}{2}\|w_1-w_2\|_{W(0, \infty; \Omega)}.
        \end{split}   
    \end{equation*}
    Hence $\mathcal{F}$ is a strict contraction on $B$, and Banach's fixed-point theorem implies the existence of a unique solution $\psi \in B$ with $\psi=\mathcal F(\psi)$, completing the proof of local existence and decay.
\end{proof}

A direct follow-up from this theorem is the following corollary.

\begin{corollary}[Local decay]\label{thm:local-decay}
  Let $u(t) = -\B^{*}\Pi y(t)$, $y(0) = \Pcal\mu_{0}$.  
  There exist $\varepsilon>0$, $C > 0$ such that
  \[
  \|y(t)\|_{\bar\mu^{-1}} \le C e^{-\delta t},
  \qquad \forall t\ge0,
  \]
  provided $\|y(0)\|_{\bar\mu^{-1}} \le \varepsilon$.
\end{corollary}

Along the closed loop $u(t) = -\B^* \Pi y(t)$, one has the identity
\[
\frac{d}{dt} \left( e^{2\delta t} \langle \Pi y(t), y(t) \rangle_{\bar\mu^{-1}} \right) = - e^{2\delta t} \left( \langle \mathcal{M} y(t), y(t) \rangle_{\bar\mu^{-1}} + |u(t)|_2^2 \right),
\]
hence 
\[
\int_0^{\infty} e^{2\delta t} |u(t)|_2^2 \, dt \le \langle \Pi y(0), y(0) \rangle_{\bar\mu^{-1}}.
\]
In particular, since the operator $\B^* \Pi$ is bounded, \Cref{thm:local-decay} yields 
\[
|u(t)|_2 \le \|\B^* \Pi\| \|y(t)\|_{\bar\mu^{-1}}  \le C \|\B^* \Pi\|  e^{-\delta t}
\]
for all $t > 0$, as long as $\|y(0)\|_{\bar\mu^{-1}}$ is sufficiently small.
\if{Additionally, 
\[
|u(t)|_2^2 = \sum_{j=1}^m \left( \langle \Pi y, \psi_j \rangle_{\bar\mu^{-1}} \right)^2
\]}\fi

\section{Numerical Experiments}
\label{sec:numerics}

To illustrate the performance of the Riccati-based feedback control law, we discretize the controlled \RevAdd{perturbation equation~\eqref{eq:abstract_formulation} on the torus associated with the controlled McKean-Vlasov equation~\eqref{eq:controlled_mckean_vlasov}, using a} spectral method (see, e.g., \RevAdd{\cite{mohammadi2015}}).
\RevAdd{In our experiments, we employ a truncated Fourier-spectral scheme, chosen for its natural handling of periodic boundary conditions, its high accuracy for smooth solutions and coefficients, the fact that the Fourier basis diagonalizes the Laplacian, and the use of Fast Fourier Transform to evaluate convolution terms and inner products.
However, the nonlinear remainder $\W$ in~\eqref{eq:operators_mckean_vlasov} still couples modes, so interactions among retained modes generate frequencies outside the truncated space.
At the discrete level, the feedback design also requires solving a coupled finite-dimensional Riccati equation rather than independent mode-by-mode problems.
Consequently, the truncated dynamics are not closed, and a rigorous convergence analysis of the closed-loop Fourier approximation is nontrivial.}
\RevAdd{Other numerical approaches for nonlinear Fokker-Planck equations include structure-preserving finite-difference schemes designed to preserve positivity, entropy dissipation, and large-time behavior~\cite{pareschi2018structure}.}

We truncate to $2L+1$ Fourier modes
\[
y(x,t) \approx \hat{y}_L(x,t) \coloneqq \sum_{k=-L}^L a_k(t) \phi_k(x), \quad \bar\mu(x) \approx \hat{\mu}_L(x) \coloneqq \sum_{k=-L}^L \bar\mu_k \phi_k(x),
\]
impose $a_0(t) \equiv 0$ and $\bar\mu_0 = 1/\sqrt{2\pi}$ to meet the mass conservation property, and thus end up with a system with $2L$ ordinary differential equations (ODEs) by projecting \RevAdd{equation~\eqref{eq:abstract_formulation}} onto the finite-dimensional space spanned by these modes.
In the torus $\T$, we consider $\phi_k(x) = e^{\mathrm{i}kx}/\sqrt{2\pi}$ so that $\langle \phi_j, \phi_k \rangle = \delta_{jk}$.
For each $i \in I \coloneqq  \{-L, \dots, -1, 1, \dots L\}$, we have
\[
\partial_t a_i = \sum_{k \in I} \Bigl\langle \Lcal\phi_k, \phi_i \Bigr\rangle a_k + \sum_{k \in I}\sum_{j=1}^m u_j \Bigl\langle \Nop_j \phi_k, \phi_i \Bigr\rangle a_k + \sum_{j=1}^m u_j\langle \phi_i, \B_j\rangle + F_i(a),
\]
where $F_i(a) \coloneqq -\langle \hat{y}_L \nabla W * \hat{y}_L, \nabla \phi_i \rangle$.
We then define the inner product matrices
\[
\begin{aligned}
  [L_V]_{ik} &= \bigl\langle \phi_k \nabla V, \nabla \phi_i \bigr\rangle, \\
  [D]_{ik} &= \bigl\langle \nabla\phi_k, \nabla\phi_i \bigr\rangle,\\
  [L_W]_{ik} &= \bigl\langle \phi_k \nabla W * \hat\mu_L + \hat\mu_L \nabla W * \phi_k, \nabla \phi_i \bigr\rangle,\\
  [N_j]_{i k} &= \langle \phi_k \nabla \alpha_j, \nabla \phi_i \rangle, \\ 
  [B_j]_{i}
  &=\langle \B_j,\phi_i\rangle = \bigl\langle \nabla \cdot(\hat\mu_L\,\nabla\alpha_j), \phi_i\bigr\rangle,
\end{aligned}
\]
for $i,k \in I$ so that, in vector form,
\begin{equation}
  \label{eq:nonlinear_finite_dimensional_mckean}
  \partial_t a = -\left[L_V + \sigma D + L_W + \sum_{j=1}^m u_j N_j\right] a  + B u + F(a).
\end{equation}
We also approximate the initial condition by $\hat{\mu}_0$ with Fourier coefficients $\bar{\mu}_{0,i}$ for $i \in I$ and $\bar{\mu}_{0,0} = 1/\sqrt{2\pi}$. 
Finally, we approximate $\Mcal$ by the matrix $M$ defined as
\[
M_{ik} = \langle \Mcal\phi_k, \phi_i\rangle, \quad i,k\in I
\]

With these definitions, set $A=-(L_V + \sigma D + L_W - \delta I)$.
We then solve the finite-dimensional algebraic Riccati equation 
\begin{equation}
  \label{eq:finite_dimensional_riccati_equation}
  A^*\Pi + \Pi A - \Pi B B^* \Pi + M = 0,
\end{equation}
for the (unique) self-adjoint positive-definite matrix $\Pi$.
The resulting optimal feedback is \RevAdd{$u = -B^* \Pi a$}.
\RevAdd{For each fixed truncation level $L$, the closed-loop system~\eqref{eq:nonlinear_finite_dimensional_mckean} after substituting $u$ with the feedback law is a finite-dimensional autonomous ODE with a polynomial right-hand side of the form
\[
\dot a = A_L a + Q_L(a,a),
\]
for a matrix $A_L$ and a bilinear map $Q_L$.
Consequently, the Galerkin coefficients $a$ are smooth in time on their interval of existence, and each feedback component $u_j(t)$ inherits the same time regularity.}

\RevAdd{The quality of the numerical control depends in part on the conditioning of~\eqref{eq:finite_dimensional_riccati_equation}, that is, on the sensitivity of the solution $\Pi$ to perturbations in $A$, $B$, and $M$. 
This local sensitivity is characterized by the Lyapunov operator $H \mapsto A_{\rm cl}^* H + H A_{\rm cl}$, where $A_{\rm cl} = A - BB^*\Pi$, so poor conditioning corresponds to this operator being close to singular~\cite{sun1998perturbation}.
By setting $\B_j = \psi_j$ for the directions satisfying $\Re\lambda_j \ge -\delta$ in \Cref{thm:hautus}, the columns of $B$ are the projections of the eigenfunctions $\psi_j$ of $\Lcal$ onto the Fourier basis, so weak actuation of the modes that need to be stabilized is not expected to be the source of ill-conditioning.
For the practical Fourier ansatz used below, however, this conclusion is no longer automatic.
Since the numerical section is only illustrative, we only check the Hautus stabilizability condition numerically for the finite-dimensional pair $(A,B)$, namely
\[
\operatorname{rank}[\lambda I-A, B] = 2L
\]
for all eigenvalues $\lambda$ of $A$ with $\Re \lambda \ge 0$, instead of pursuing a conditioning analysis.}

In practice, for the next examples, we bypass computing the first few eigenfunctions of $\Lcal$ and instead take a simple Fourier-basis ansatz for the control shapes. 
Concretely, we set $m = 4$ and
\[
\alpha_1(x) = \sin(x),\quad \alpha_2(x) = \cos(x),\quad \alpha_3(x) = \sin(2x),\quad \alpha_4(x) = \cos(2x).
\]
\RevAdd{We verify numerically that the truncated pair satisfies the Hautus condition.
This choice} yields stabilizing feedback in all our examples and provides a practical framework for computing eigenfunctions when the process is numerically complicated.
The controlled ODEs are then integrated in time using a Runge-Kutta of order 5(4).
For simplicity, we take $M = \nu I$, \RevAdd{with $\nu = 10^3$ or $\nu = 10^5$}.

\subsection{Noisy Kuramoto model}
\label{sec:noisy_kuramoto_model}

This model is a canonical McKean-Vlasov PDE \RevAdd{describing the collective dynamics of mean-field coupled phase oscillators on the circle, where attractive coupling competes with diffusion~\cite{kuramoto1984chemical}.
It} exhibits a well-studied synchronization transition (see, e.g.,~\cite{bertini2010dynamical}), \RevAdd{where} the uniform steady state loses stability at a critical coupling strength, and a continuum of nonuniform equilibria bifurcate thereafter. 
This rich phase-transition structure in a tractable model makes it a good proving ground to demonstrate both stabilization of unstable equilibria and redirection onto chosen synchronized distributions.
In our notation, the noisy Kuramoto model is written as $W(\theta) = -K\cos(\theta)$, $\sigma = 0.5$, and $V(\theta) = 0$. 
The critical coupling strength is $K=1$.
\RevAdd{For $K < 1$, the only steady state is the uniform density $\bar\mu = \tfrac1{2\pi}$, also called the fully incoherent state since it corresponds to the absence of synchronization; in this regime, the dynamics converge to $\bar\mu$.
This state becomes critical when $K = 1$, and for $K > 1$, there are infinitely many nonuniform synchronized densities, all translations of~\eqref{eq:steady_state_kuramoto}, and the uniform state becomes unstable.}

\RevAdd{\Cref{fig:kuramoto_steady_state_distributions} summarizes the synchronized steady states and their slow convergence near the bifurcation point.}
\RevAdd{The left panel} plots the stable density $\bar\mu$ given by~\eqref{eq:steady_state_kuramoto}, illustrating how increasing the coupling sharpens the peak.
\RevAdd{The right panel} shows the spectral gap $\lambda_{\rm gap}$ of the linearized operator $\mathcal{L}$ as a function of $K$ when $K > 1$, comparing the theoretical lower bound from~\cite{bertini2010dynamical} with our numerical computations over $K \in (1, 1.25]$.
\RevAdd{The small values of this gap close to $K=1$ indicate slow convergence toward the synchronized steady state in the uncontrolled dynamics.}

\begin{figure}[!htbp]
  \centering
  \begin{minipage}{0.48\textwidth}
    \centering
    \includegraphics[width=\textwidth]{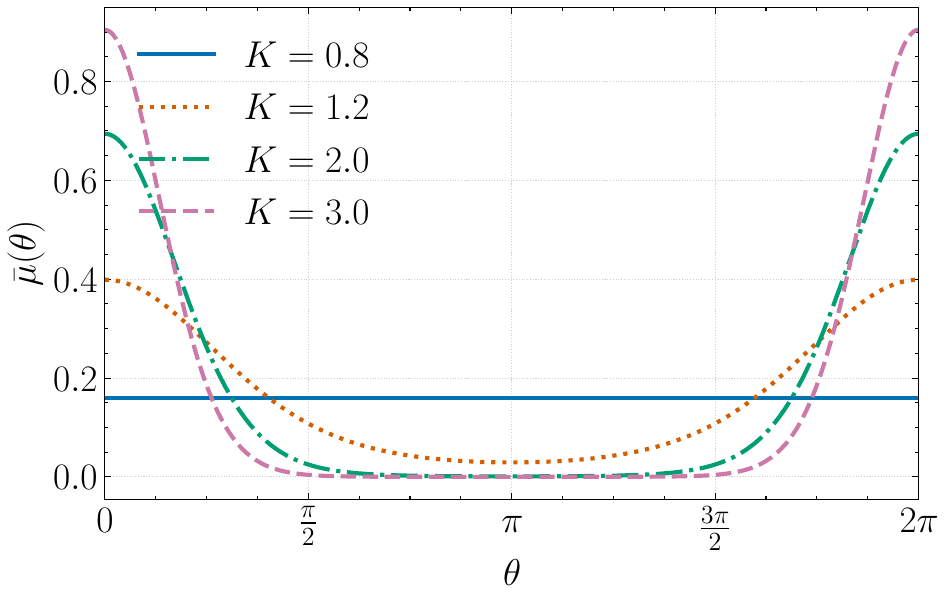}
  \end{minipage}
  \hfill
  \begin{minipage}{0.48\textwidth}
    \centering
    \includegraphics[width=\textwidth]{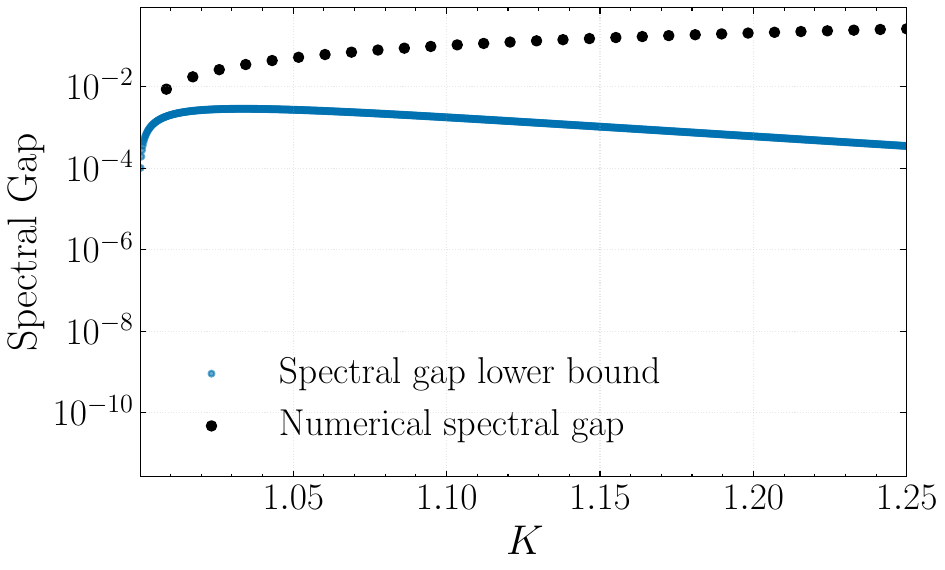}
  \end{minipage}
  \caption{{\bf Different steady states and spectral gap.} 
  \RevAdd{Left:} shape of the synchronized equilibrium densities $\bar\mu(\theta)$ as coupling $K$ increases.  
  \RevAdd{Right:} numerical estimates and theoretical lower bounds of the spectral gap $\lambda_{\rm gap}$ as a function of $K$.}
  \label{fig:kuramoto_steady_state_distributions}
\end{figure}

We solve the \RevAdd{controlled finite-dimensional system~\eqref{eq:nonlinear_finite_dimensional_mckean}} from a fixed initial distribution $\mu_0$, and compare the resulting trajectories \RevAdd{of $y$, defined in~\eqref{eq:y_definition}, in terms of the weighted norm} $\|y(t)\|_{\bar\mu^{-1}}$ with and without the feedback control.
\Cref{fig:kuramoto_control_panels} presents three distinct scenarios for the noisy Kuramoto model:

\begin{enumerate}
  \item[(a)] for $K < 1$, feedback accelerates convergence to the globally stable incoherent state $\bar\mu = (2\pi)^{-1}$; 

  \item[(b)] for $K > 1$, where $\bar\mu =  (2\pi)^{-1}$ becomes unstable, \RevAdd{the} control successfully stabilizes it;

  \item[(c)] for $K > 1$, we steer the system toward a \RevAdd{prescribed spatial translation of the steady state} distribution \RevAdd{given in~\eqref{eq:steady_state_kuramoto}, namely} $\bar\mu(\cdot + \theta_0)$ \RevAdd{for some fixed $\theta_0 \in \T$}, in contrast with the uncontrolled dynamics, which converge to a different equilibrium. 
\end{enumerate}

Fitting a linear trend to $\log_{10} \|y(t)\|_{\bar\mu^{-1}}$ over $t\in[0.2, 2]$ yields slopes of approximately $-3.5, -3.4, -3.2$, respectively, for each example, corresponding to an exponential decay rate $\delta \ge 3$ for all three examples.
As a comparison, the uncontrolled dynamics in the first example decay exponentially at a rate $0.01$.

\begin{figure}[!htbp]
  \centering
  \includegraphics[width=\textwidth]{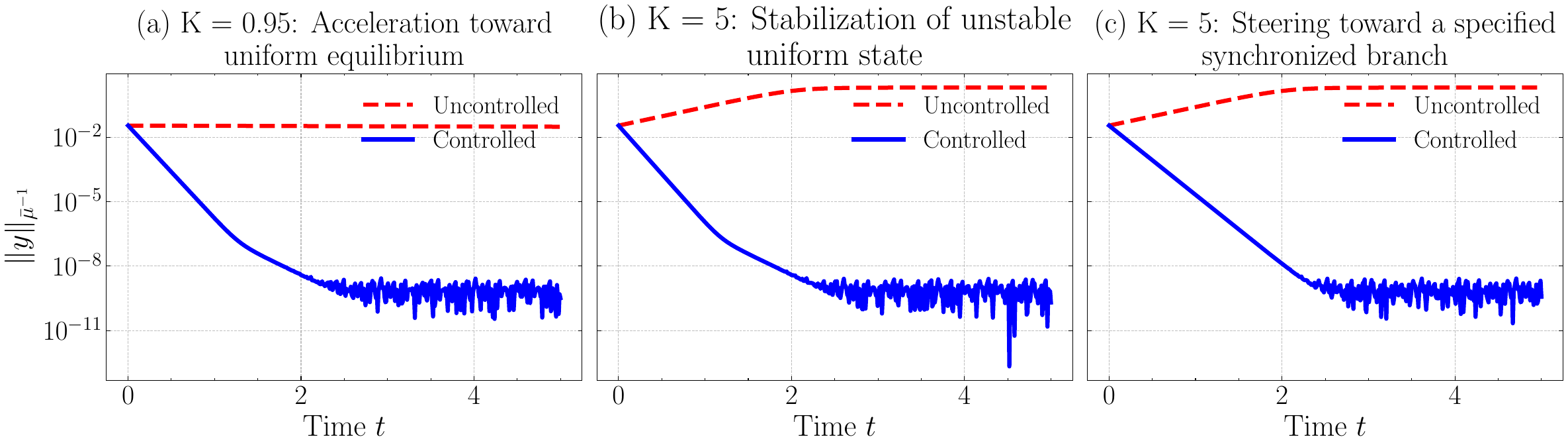}
  \caption{{\bf Results for the control for Noisy Kuramoto Model.}
    (a) Acceleration to the incoherent state ($K=0.95$): Riccati feedback (blue) vs.\ uncontrolled (red dashed).  
    (b) Stabilization of the unstable uniform state ($K=5$): feedback recovers exponential decay.  
    (c) Steering under a different stable steady state (different translation) ($K=5$): feedback still maintains decay.  
    All panels plot $\|y(t)\|_{\bar\mu^{-1}}$ on a log-scale.
    \RevAdd{In all three panels, $\nu = 10^3$.}}
  \label{fig:kuramoto_control_panels}
\end{figure}

Quantitative evidence is provided by \Cref{fig:kuramoto_changing_K}.
This figure illustrates how the feedback law enforces exponential decay for different values of $K$, both toward the unstable uniform state (scenario (a)) and the stable steady state (scenario (b)). 
In contrast, uncontrolled dynamics (scenario (c)) converge more slowly to the stable steady state and do not converge in the unstable case.
This result uses $\delta = 3$ for all examples.
\Cref{fig:kuramoto_spectra} shows the spectrum of the linearized operator $\mathcal{L}$ at $\bar\mu = 1/2\pi$ and the closed-loop operator $\mathcal{L} - \B \B^\ast \Pi$ for the noisy Kuramoto model with $K=0.5$, $K=0.95$, and $K=5$. 
For the cases where $K < 1$, all the real parts of the eigenvalues are negative but as close to the threshold $K_c=1$, the smaller the spectral gap is.
For $K > 1$, there is an eigenvalue with positive real part, which gives the instability of the uniform distribution.
In all cases, the feedback moves unstable or weakly stable modes into the half-plane $\Re(\lambda) \le -\delta$, achieving the prescribed exponential stabilization rate.
 
\begin{figure}[!htbp]
  \centering
  \includegraphics[width=\textwidth]{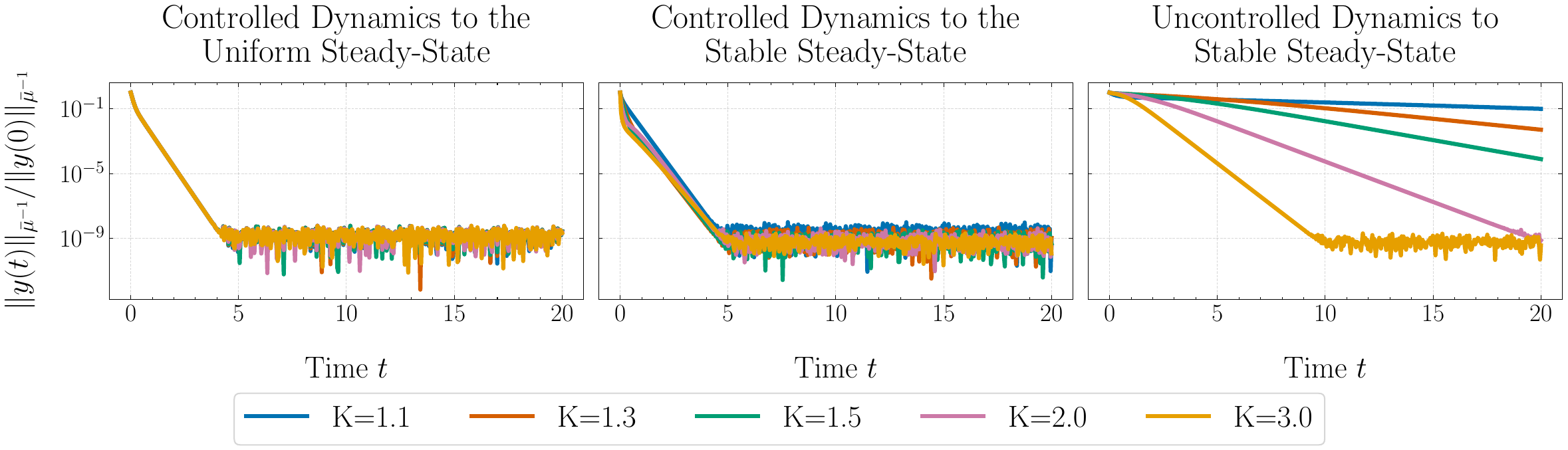}
  \caption{ {\bf Results for the controlled noisy Kuramoto model with different couplings.} 
  Log-scale evolution of the scaled norm $\|y(t)\|_{\bar\mu^{-1}} / \|y(0)\|_{\bar\mu^{-1}}$ for three scenarios: (a) feedback-controlled convergence to the uniform steady state, (b) feedback-controlled acceleration to the stable steady state, and (c) open-loop (uncontrolled) convergence to the same target.
  \RevAdd{In all controlled computations shown here, $\nu = 10^3$.}}
  \label{fig:kuramoto_changing_K}
\end{figure}

\RevAdd{\Cref{fig:kuramoto_controls_cost} complements the plots from~\Cref{fig:kuramoto_changing_K} by showing the control components in one representative run and the associated cumulative cost
\[
J(t) \coloneqq \frac12\int_0^t e^{2\delta s}\left(\langle y(s), \M y(s)\rangle_{\bar\mu^{-1}} + |u(s)|_2^2\right)\, ds.
\]
In the $K=3$ example on the left, one component dominates while the others remain close to zero for this representative initial condition and target.
The right panel shows that, for all displayed values of $K$, the cumulative cost quickly reaches a plateau in the controlled dynamics targeting the unstable uniform density, indicating that most of the control effort is concentrated near the initial transient.}

\begin{figure}[!htbp]
  \centering
  \includegraphics[width=0.8\textwidth]{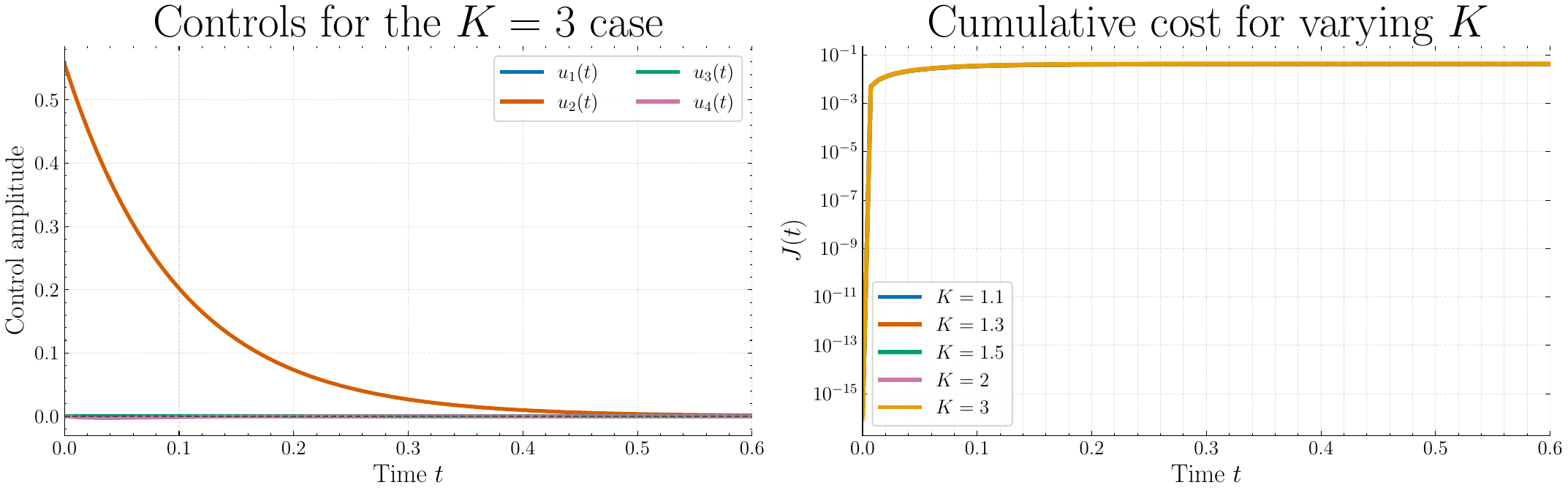}
  \caption{\RevAdd{{\bf Representative controls and cumulative cost in the noisy Kuramoto model.} 
  Left: control components $u_1(t),\dots,u_4(t)$ for a representative controlled run with $K=3$ and target $\bar\mu=(2\pi)^{-1}$.
  Right: cumulative cost $J(t)$ for controlled convergence to the uniform steady state and varying $K$.
  Here $\nu = 10^3$.}}
  \label{fig:kuramoto_controls_cost}
\end{figure}

\begin{figure}[!htbp]
\centering
    \includegraphics[width=\textwidth]{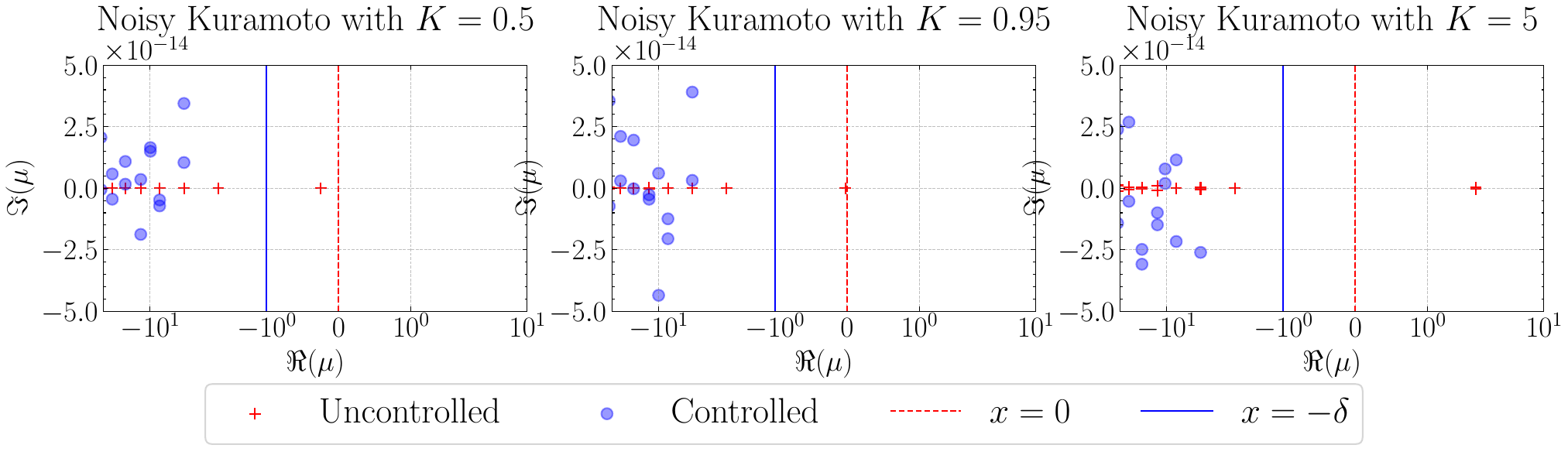}
    \caption{{\bf Spectral shift by the feedback control in the noisy Kuramoto model.} Spectra with real part at least $-30$ of the uncontrolled generator $\mathcal{L}$ (red crosses) and the closed-loop operator $\mathcal{L} - \B \B^\ast \Pi$ (blue circles) for the noisy Kuramoto model with three coupling strengths $K$. 
    The red dashed line marks $x=0$ (stability boundary) and the blue solid line marks $x=-\delta$ with $\delta=1$ (target decay rate).
    \RevAdd{Here $\nu = 10^3$.}}
    \label{fig:kuramoto_spectra}
\end{figure}

\subsection{Symmetry breaking example}

To demonstrate the applicability of our Riccati-based feedback control scheme, we evaluate it in three additional scenarios where either an external potential or a nonconvex interaction breaks the translation symmetry of the classical noisy Kuramoto model.  
Although each example features a different bifurcation mechanism, the same feedback design consistently restores, or, in the case of unstable targets, enables, uniform exponential decay.

\subsubsection*{(a) Cosine-potential perturbation}  
We add \RevAdd{the} periodic potential $V(x) = \cos(2x)$ to the noisy Kuramoto model \RevAdd{introduced in \Cref{sec:noisy_kuramoto_model}} with $K=1$.
In particular, $V$ attains its global minima twice at $x=\pi/2$ and $x=3\pi/2$ and its maxima at $x=0$ and $\pi$.
Thus, it creates two symmetric wells separated by a barrier.
At a high temperature $\sigma \gtrsim 0.78$, there exists a unique steady state $\bar\mu(x) \propto \exp(-V(x)/\sigma)$. 
As the temperature $\sigma$ decreases past a critical value ($\approx 0.77)$, two other steady states emerge, each with a peak at either $x=\pi/2$ or \RevAdd{$x=3\pi/2$}  (see~\cite{angeli2023well}).
In \Cref{fig:kuramoto_potential_control_heatmap}, we sweep $\sigma \in [0.6, 1.0]$ and \RevAdd{plot the time evolution of} $\|y\|_{\bar\mu^{-1}}$, where $\bar\mu \propto e^{-\cos(2x)/\sigma}$.
\RevAdd{The top panel shows that, under the feedback control, the perturbation decays rapidly to values below $10^{-8}$ across the entire range of temperatures considered even when the number of controls is fixed.
In contrast, the bottom panel shows that without control the decay depends strongly on $\sigma$, in the sense that for smaller $\sigma$ the norm remains large over the whole time window, while for larger $\sigma$ it decreases only gradually.}

\begin{figure}[!htbp]
  \centering
  \includegraphics[width=0.7\textwidth]{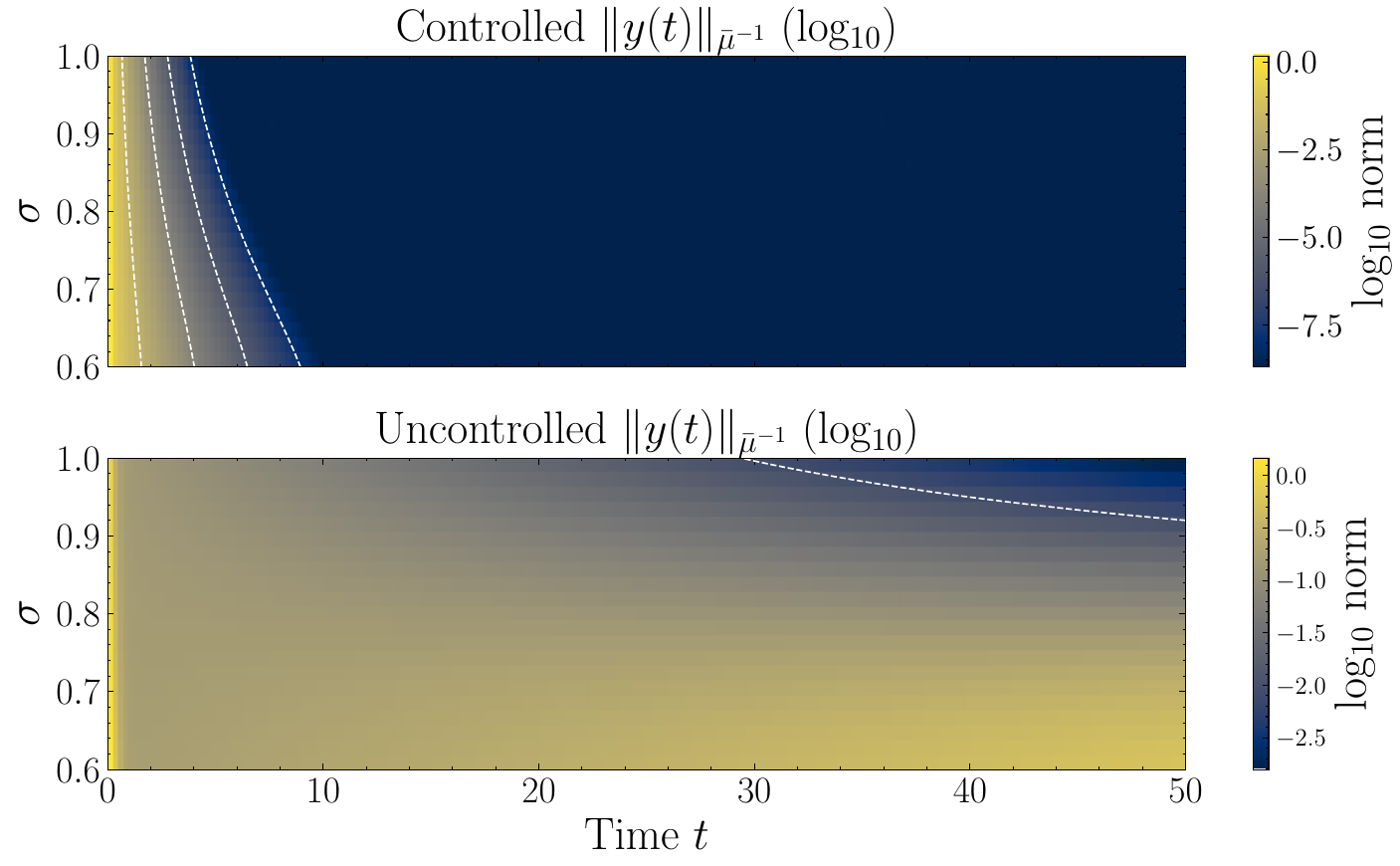}
  \caption{{\bf Heatmap of Convergence for cosine-potential perturbation.}
  Base-10 logarithm of the norm $\|y(t)\|_{\bar\mu^{-1}}$ versus time $t$ (horizontal axis) and temperature $\sigma$ (vertical axis) for the confining potential \(V(x)=\cos(2x)\) and interaction \(W(x)=-\cos x\).
  \textbf{Top:} controlled dynamics under the Riccati feedback law. 
  \textbf{Bottom:} uncontrolled dynamics. 
  White contour lines indicate levels $\log_{10}\|y\|_{\bar\mu^{-1}}\in\{-8,-6,-4,-2\}$.
  \RevAdd{Here $\nu = 10^5$.}}
  \label{fig:kuramoto_potential_control_heatmap}
\end{figure}

\subsubsection*{(b) O(2) model in a small magnetic field}

Next, we bias the noisy Kuramoto with $K=1$ by adding an \RevAdd{external field} $V(x) = -\eta \cos(x)$ with the field-strength parameter $\eta = 0.05$ so that $x=0$ \RevAdd{becomes the} preferred alignment \RevAdd{direction}.
As shown in~\cite{bertoli2024stability}, the uncontrolled dynamics undergo a bifurcation at a critical temperature. 
\Cref{fig:o2_sigma_sweep} compares the decay of the norm $\|y(t)\|_{\bar\mu^{-1}}$ for the controlled dynamics (left) and for the uncontrolled dynamics (right), across several temperatures $\sigma$.
The left panel demonstrates how the feedback control accelerates convergence to the steady state with a uniform exponential rate even close to the bifurcation threshold ($\sigma \approx 0.47$).
In contrast, the right panel shows the much slower (or stalled) decay of the uncontrolled dynamics.
In these experiments, we considered $\delta = 1$.

\begin{figure}[!hbtp]
  \centering
  \begin{subfigure}{\textwidth}
    \includegraphics[width=\textwidth]{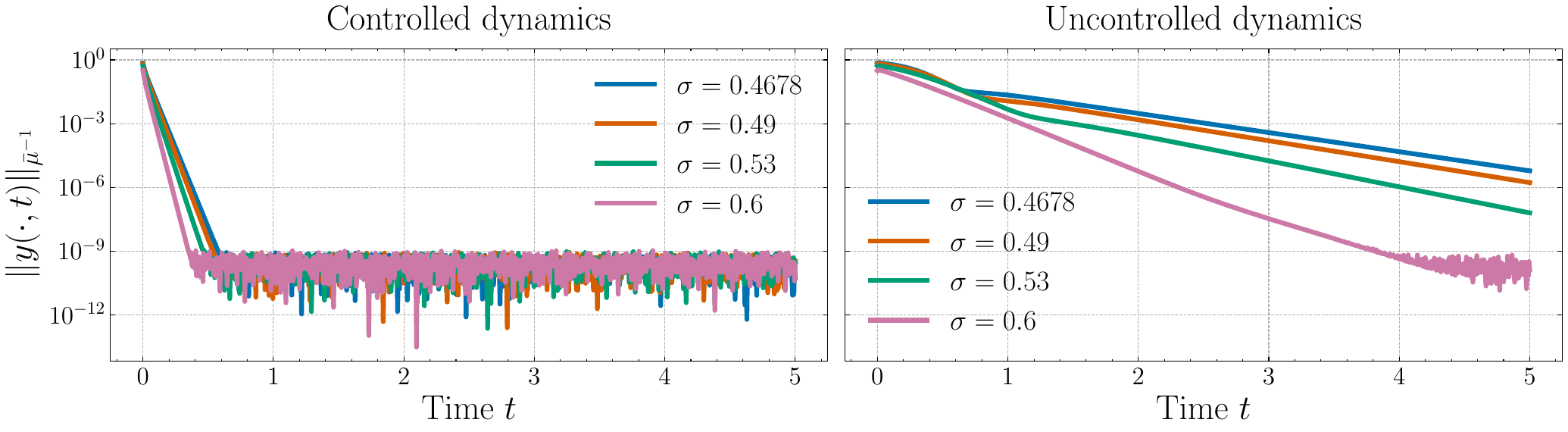}
    \caption{O(2) model in a magnetic field with \(\eta=0.05\) and \(K=1\).}
    \label{fig:o2_sigma_sweep}
  \end{subfigure}

  \begin{subfigure}{\textwidth}
    \includegraphics[width=\textwidth]{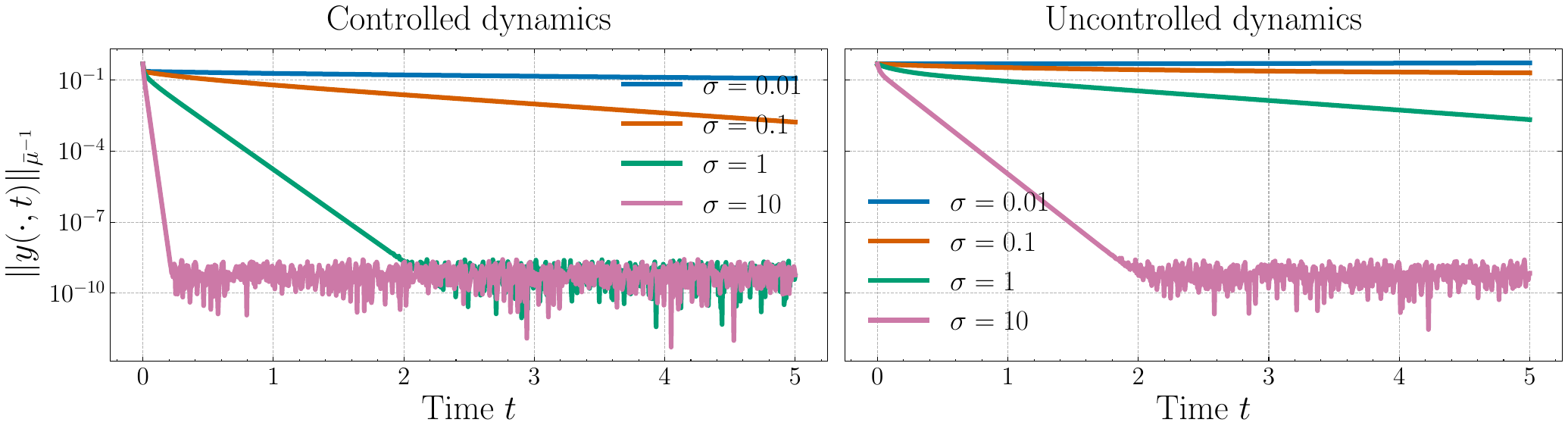}
    \caption{Attractive Von Mises interaction.}
    \label{fig:von_mises_control_different_sigmas}
  \end{subfigure}

  \caption{\textbf{Feedback control in two non-convex 1D examples.}  
    In both cases, the feedback control (left panels) enforces rapid exponential decay across all tested temperatures, whereas the uncontrolled dynamics (right panels) either stall or decay only for large \(\sigma\).
    \RevAdd{In both examples, $\nu = 10^5$.}}
  \label{fig:stacked_nonconvex1D}
\end{figure}

\subsubsection*{(c) Attractive Von Mises interaction}

We now turn to a pure mean-field coupling without a confining potential, choosing the interaction kernel 
\[
W(x) = -I_0(\theta)^{-1}\,\exp\!\bigl(\theta\cos x\bigr), \quad x\in[0,2\pi],
\]
where $I_0(\theta)$ denotes the modified Bessel function of order zero.
This model is a periodized analogue of a Gaussian attractive interaction and is called the attractive Von Mises interaction model.
It is known (see \cite[Cor. 5.14]{carrillo2020long}) that for high temperature $\sigma$, the uniform \RevAdd{density} is the unique steady state and, as the temperature decreases below a critical value, \RevAdd{it} loses stability, and a stable nonuniform steady state emerges. 
\Cref{fig:von_mises_control_different_sigmas} depicts the acceleration of convergence through control.
In our four-function ansatz $\{\alpha_j\}_{j=1}^4$, the maximum design rate we could numerically realize was $\delta\approx0.06$ (considering all examples), so the improvement in decay is modest for small $\sigma$.  
However, if we enrich the shape-function space, e.g., taking $8$ Fourier modes $\alpha_j$ instead of $4$, the solvable $\delta$ rises (in our tests up to $\delta \approx 0.2$), yielding faster convergence, which we do not depict here.

\subsection{Feedback stabilization on the two-dimensional torus}

\RevAdd{To illustrate that the Riccati-based feedback design extends beyond one dimension numerically, we consider a two-dimensional attractive von Mises-type interaction on the torus $\T^2$.
We take $V(x,y) \equiv 0$, $\sigma = 0.5$, the interaction kernel
\[
W(z_1,z_2) = -\frac{K}{Z_{\theta,\rho}}
\exp\left(\theta(\cos z_1 + \cos z_2) + \rho \cos z_1 \cos z_2\right),
\]
where
\[
Z_{\theta,\rho} = \int_{\T^2}
\exp\left(\theta(\cos z_1 + \cos z_2) + \rho \cos z_1 \cos z_2\right) \, dz_1 dz_2,
\]
and take the target equilibrium to be the uniform density $\bar\mu=(2\pi)^{-2}$.
This kernel is a two-dimensional analogue of the attractive von Mises interaction, with $\rho$ controlling the coupling between the two angular variables.}

On $\T^2$, we consider the orthonormal Fourier modes
\[
\phi_{k_x, k_y}(x,y) = \frac{1}{2\pi} e^{\mathrm{i}(k_x x + k_y y)}, \quad k_x, k_y = -L, \dots, L,
\]
\RevAdd{and one single state vector indexed by all nonzero modes $(k_x,k_y) \neq (0,0)$, where the zero mode is omitted because of mass conservation.
For the shape control functions, we pick the four coordinate modes and the four diagonal modes
\[
\begin{aligned}
  \alpha_1(x,y) = \cos x, \quad \alpha_2(x,y) = \sin x, &\quad \alpha_3(x,y) = \cos y, \quad \alpha_4(x,y) = \sin y, \\
  \alpha_5(x,y) = \cos(x+y), &\quad \alpha_6(x,y) = \sin(x+y),\\ 
  \alpha_7(x,y) = \cos(x-y), &\quad \alpha_8(x,y) = \sin(x-y).
\end{aligned}
\]
In the reported computation we use $\theta=1.5$, $\rho=0.75$, $\delta=1$, $\nu=10^5$, and $L=8$.
We choose $K = 30$ so that the uniform density is unstable for this truncation.
\Cref{fig:2D_example} shows the resulting norm decay, the feedback controls, the initial density, the final uncontrolled density, and the final controlled density.}
This figure confirms that the same Riccati feedback control strategy achieves exponential stabilization with a rate parameter at least $\delta = 1$.

\begin{figure}[!htbp]
  \centering
  \begin{subfigure}{0.4\textwidth}
    \centering
    \includegraphics[width=\linewidth]{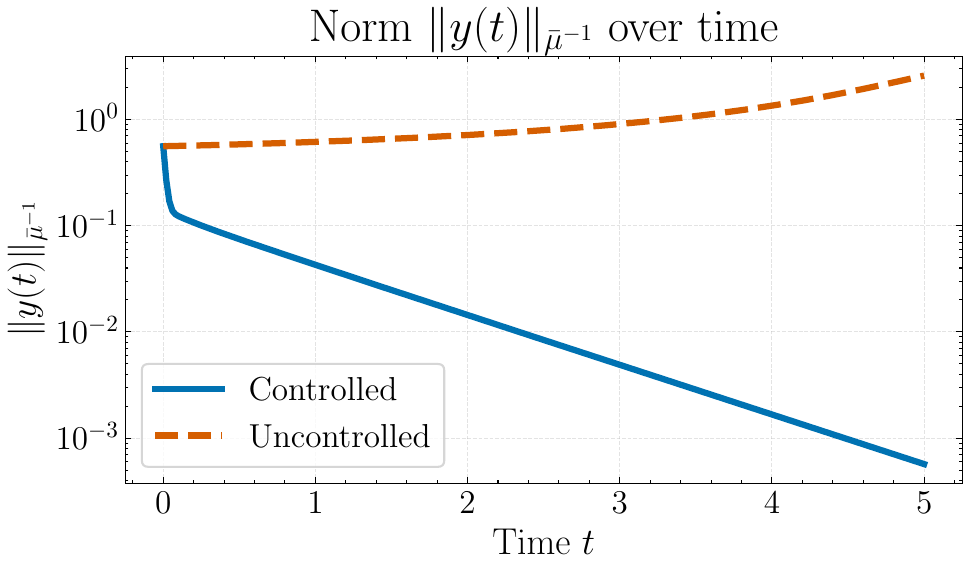}
  \end{subfigure}
  \begin{subfigure}{0.4\textwidth}
    \centering
    \includegraphics[width=\linewidth]{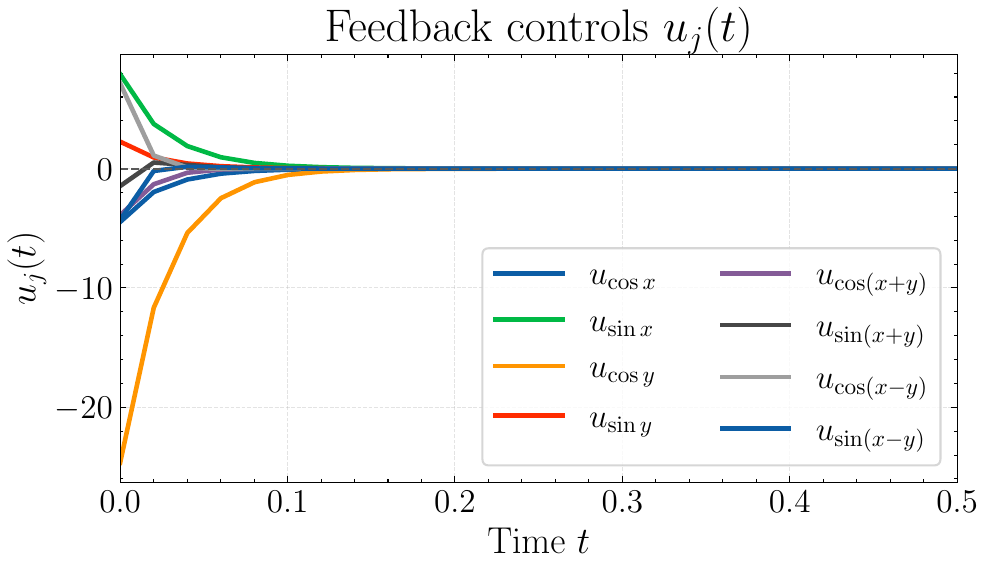}
  \end{subfigure}

  \begin{subfigure}{0.32\textwidth}
    \centering
    \includegraphics[width=\linewidth]{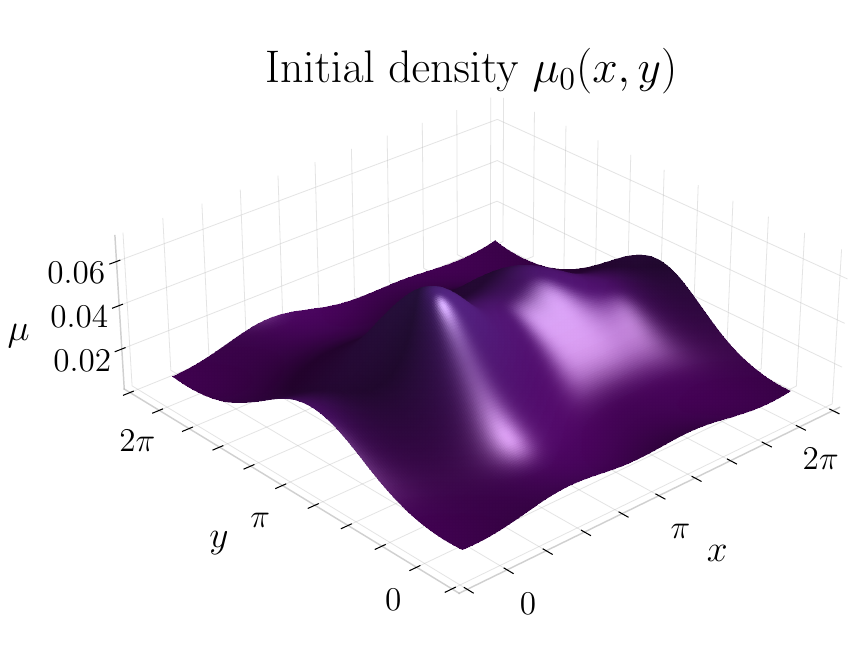}
  \end{subfigure}
  \hfill
  \begin{subfigure}{0.32\textwidth}
    \centering
    \includegraphics[width=\linewidth]{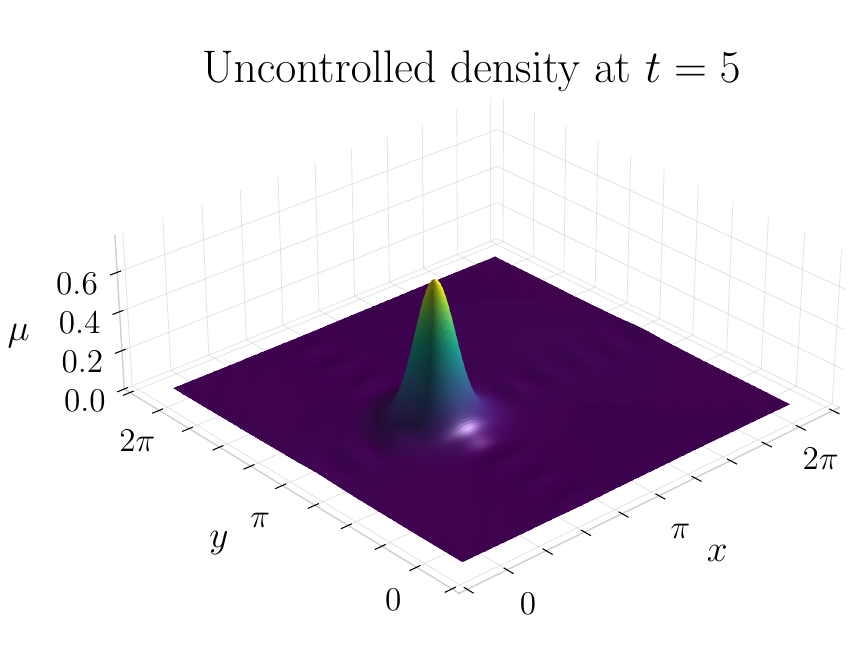}
  \end{subfigure}
  \hfill
  \begin{subfigure}{0.32\textwidth}
    \centering
    \includegraphics[width=\linewidth]{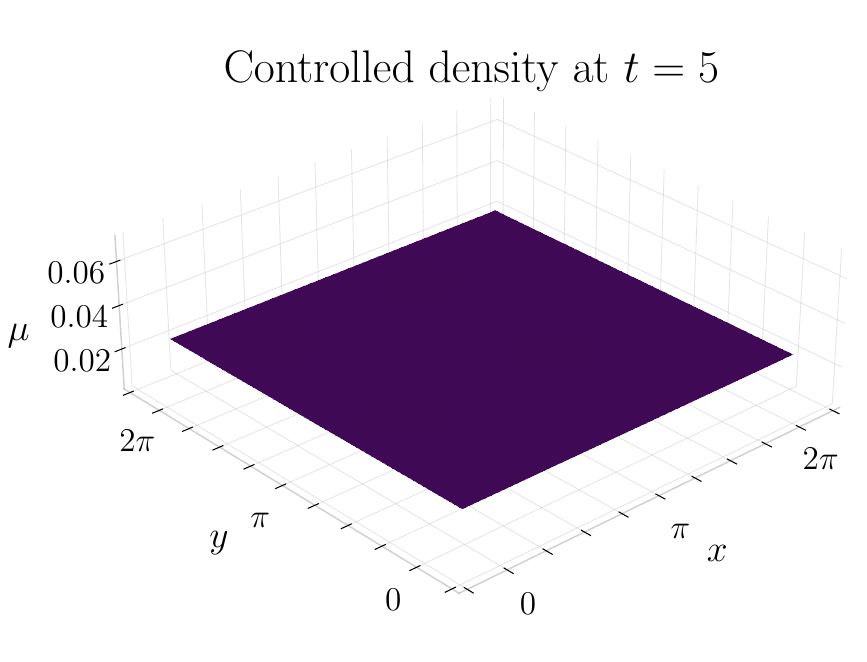}
  \end{subfigure}
  \caption{\textbf{2D torus example.}
    \RevAdd{2D attractive von Mises-type example with $K=30$, $\theta=1.5$, $\rho=0.75$, and feedback weight $\nu=10^5$.
    The top row reports the time evolution of $\|y(t)\|_{\bar\mu^{-1}}$ and the feedback controls $u_j(t)$.
    The bottom row compares the initial distribution $\mu_0$ with the uncontrolled and controlled densities at $t=5$.}}
  \label{fig:2D_example}
\end{figure}

\section{Conclusions}
\label{sec:conclusions}

We have proposed and rigorously analyzed a Riccati-based feedback design for a nonlocal and nonlinear Fokker-Planck (McKean-Vlasov) equation on the torus.  
By mapping the linearized dynamics via a ground-state transform to a Schr\"odinger-type operator with compact resolvent, we verified an infinite-dimensional Hautus criterion, established the corresponding operator-Riccati equation, and proved local exponential stabilization of the PDE under the resulting feedback law by providing energy estimates for the nonlinear terms.

\RevAdd{The numerical experiments illustrate that the same four-mode Fourier ansatz yields exponential decay across qualitatively different regimes in one dimension (noisy Kuramoto, cosine-potential perturbation, O(2) model, von Mises interaction), including both acceleration of convergence and redirection between steady states. 
Moreover, in our experiments, enlarging the shape-function space increases the achievable decay rate $\delta$.
We further demonstrated that this approach extends to two dimensions via a tensor-product Fourier discretization.
However, because the nonlinear term $\W$ couples modes, the truncated dynamics are not closed under the nonlinear Fourier dynamics, and a rigorous convergence analysis of the closed-loop Fourier approximation remains nontrivial. 
In addition, for the practical truncated systems considered here, the Hautus stabilizability condition is checked numerically for each finite-dimensional pair $(A,B)$.}

These results establish a practical and theoretically grounded framework for direct PDE-feedback control of nonlocal mean-field dynamics.  
Future work will address (i) scalable solvers for high-dimensional problems; (ii) robustness of the control \RevAdd{with respect to perturbations in} the model parameters; \RevAdd{(iii) analysis of the closed-loop Fourier approximation, including conditioning of the truncated Riccati equation and convergence of the controlled solution as the truncation level increases}; and (iv) extension of the feedback design to the underlying particle system.

\appendix

\section{Technical operator-theoretic results}\label{app:operator}

\begin{lemma}\label{lem:perturb-pure}
    Let $T$ be a self-adjoint closed operator on a Hilbert space $H$ with compact resolvent, and let $A : D(A) \subseteq H \to H$ be a bounded operator in $H$, so that $D(A) \supset D(T)$.
    Then $S \coloneqq T+A$ is closed and also has compact resolvent. 
\end{lemma}

\begin{proof}
    By Kato's perturbation theorem~\cite[Thm. IV-3.17]{kato2013perturbation}, we need to verify that there exists some $\zeta \in \rho(T)$ such that 
    \begin{equation}
        \label{eq:desired_inequality}
        \|A\|  \|R(\zeta, \RevAdd{T})\| < 1,
    \end{equation}
    \RevAdd{where $R(x,T) := (xI-T)^{-1}$ is the resolvent of $T$ at $x$}.
    By~\cite[Thm. V-3.16]{kato2013perturbation}, the self-adjointness of $T$ implies that $R(x,T)$ is defined for all $x \in \C \setminus \R$ and $\|R(a+b\mathrm{i}, T)\| \le \frac{1}{|b|}$ for all $b > 0$ and $a \in \R$, where $\mathrm{i}$ is the imaginary number.
    Taking $\zeta = a + \mathrm{i}(\|A\| + 1)$ leads to the inequality~\eqref{eq:desired_inequality}.
\end{proof}

\begin{proof}[Proof of \Cref{lem:K-bounded}]\label{proof:K-bounded}
    Operator $\K$ is defined as $\mathcal{K}[\varphi](x) = \int_{\Omega} k(x,y) \varphi(y) \, dy$.
    An integral operator is Hilbert-Schmidt iff its kernel lies in $L^2(\Omega \times \Omega)$.
    Because $W\in W^{2,\infty}(\Omega)$, there exists $C_W > 0$ such that 
    \[
    \| \nabla W \|_{\infty} \le C_W,  \quad \| \Delta W \|_{\infty} \le C_W.
    \]
    Additionally, notice that $\nabla \bar\mu = -\bar\mu\nabla (V + W*\bar\mu)/\sigma$ and define 
    \[
    a(x) = \sqrt{\bar\mu(x)}\left(\sigma^{-1}|\nabla(V(x) + (W*\bar\mu)(x))| + 1 \right), \quad b(y) = C_W\sqrt{\bar\mu(y)}
    \]
    so that $|k(x,y)| \le a(x) b(y)$.
    It is clear that $b \in L^2(\Omega)$ with $\|b\| = C_W$ because $\bar\mu$ is a probability density.
    Furthermore,
    \begin{equation*}
        \begin{split}
            \|a\| &\le \frac{1}{\sigma}\|\sqrt{\bar\mu} \nabla V\| + \frac{1}{\sigma}\|\sqrt{\bar\mu} \nabla W * \bar\mu\| + \|\sqrt{\bar\mu}\| \\
            &\le \frac{1}{\sigma \sqrt{Z}} \|e^{-V/(2\sigma)} \nabla V\| + \frac{C_W}{\sigma} + 1, \\
            &=\frac{2}{\sqrt{Z}} \|\nabla (e^{-V/(2\sigma)})\| + \frac{C_W}{\sigma} + 1,
        \end{split}    
    \end{equation*}
    which implies that $a \in L^2(\Omega)$ since $\nabla \phi \in L^2(\Omega)$ for $\phi \coloneqq e^{-V/(2\sigma)}$ because $\nabla V \in L^2(\Omega)$ and $e^{-V/(2\sigma)}$ is bounded.
    Therefore, 
    \RevAdd{\[
    \iint_{\Omega} |k(x,y)|^{2} \, dx \, dy \le \|a\|^{2}\|b\|^{2}<\infty.
    \]}
    Hence $\mathcal{K}$ is Hilbert-Schmidt and therefore compact on $L^{2}(\Omega)$.
\end{proof}

\section{Full well-posedness proof}
\label{app:full-wellposedness-proof}

The goal of this appendix is to provide a complete, step-by-step proof of the global existence and uniqueness of solutions to the McKean-Vlasov equation on the $d$-dimensional torus under the assumptions stated below.
While the main article outlines the argument and states the result, here we include all intermediate estimates, functional-analytic tools, and auxiliary lemmas used in the proof.

\subsection{Preliminaries}

Let $\Omega \coloneq (\R / 2\pi \mathbb{Z})^d$ be the $d$-dimensional periodic torus. 
For $1 \le p \le \infty$ and a Banach space $X$, we write $L^p(0,T;X)$ for the Bochner space of $p$-integrable maps $[0,T] \to X$. 
We denote by $L^p(\Omega)$ the space of $p$-integrable periodic functions on $\Omega$, and by $W^{k,p}(\Omega)$ the periodic Sobolev spaces, with $H^k(\Omega) = W^{k,2}(\Omega)$. 
We use $L^1_{+}(\Omega) \coloneq \{f \in L^1(\Omega) : f \ge 0\}$, $L^2_{0}(\Omega) \coloneq \left\{f \in L^2(\Omega) : \int_{\Omega} f = 0\right\}$, and for a positive weight $w > 0$ we set $L^p(\Omega, w) \coloneq \{f : w^{1/p} f \in L^p(\Omega)\}$ when $p < \infty$.
The inner product in $L^2(\Omega)$ is denoted by $\langle \cdot,\cdot\rangle$, and the inner product in $L^2(\Omega, w)$ by $\langle \cdot,\cdot\rangle_{w}$. 
The space $H^{-1}(\Omega)$ denotes the dual of $H^1(\Omega)$, with respect to $L^2(\Omega)$ as the pivot space. 
We write $\nabla = (\partial_{x_1},\dots,\partial_{x_d})$, $\nabla \cdot F = \sum_{i=1}^d \partial_{x_i}F_i$, and $\Delta = \nabla \cdot\nabla$.

We consider the periodic McKean-Vlasov equation
\begin{align}
    \partial_t \rho = \sigma \Delta \rho + \nabla \cdot\bigl(\rho\,\nabla v[\rho]\bigr), \quad \text{in} \; \Omega \times (0,T),
    \label{sup:eq:mv_equation}
\end{align}
where $v[\rho] = V + W * \rho$, $(W*\rho)(x) \coloneq \int_\Omega W(x-y)\,\rho(y)\,dy$, with periodic boundary conditions on $\Omega$, and  $\rho(\cdot,0) = \rho_0 \in \Pac$.
Periodicity ensures conservation of mass
\begin{equation}
    \label{sup:eq:normalization_cond}
    \int_{\Omega}\rho(x,t) \, dx = \int_{\Omega}\rho_0(x) \, dx =1, \qquad \forall t \in [0,T].
\end{equation}

Throughout, we assume $W \in W^{2,\infty}(\Omega)$ is integrable and coordinate-wise even, and $V \in W^{1,\infty}(\Omega) \cap W^{2,\max\{d,2+\epsilon\}}(\Omega)$ for some $\epsilon > 0$.

\begin{definition}
    A function $\rho \in L^{2}(0,T;H^{1}(\Omega)) \cap L^{\infty}(0,T;L^{2}(\Omega))$ with $\partial_t \rho \in L^{2}(0,T;H^{-1}(\Omega))$ is a {\em weak solution} of~\eqref{sup:eq:mv_equation} if, for every $\eta\in L^{2}(0,T;H^{1}(\Omega))$,
	\begin{equation}
		\int_{0}^{T}\left\langle \partial_t \rho(t),\eta(t) \right\rangle \, dt + \int_{0}^{T} \int_{\Omega} \left(\sigma \nabla \rho + \rho \nabla v[\rho] \right) \cdot \nabla \eta \, dx \, dt = 0,
	\end{equation}
	and $\rho(0) = \rho_{0}$ (see \Cref{sup:thm:Evans}).
\end{definition}

\subsection{Probability density property\label{sup:sec:apriori}}

Let us assume that there exists a solution $\rho\in C^{1}\left(0,\infty;C^{2}\left(\Omega\right)\right)$, to equation~\eqref{sup:eq:mv_equation} to obtain a priori energy estimates so that we can deduce the solution's nonnegativity. 

\begin{lemma}
    \label{sup:lem:L1}
    Suppose $\rho \in C^{1}\left(0,\infty; C^{2}\left(\Omega\right)\right)$ is a solution of~\eqref{sup:eq:mv_equation} with $\rho_{0}\in L^{1}\left(\Omega\right)$, then $\|\rho(t)\|_{1}\le\|\rho_0\|_{1}$ for all $t \ge 0$. 
\end{lemma}

\begin{proof}
    Let $\epsilon > 0$ and define $\chi_{\epsilon}$ to be a convex $C^{2}$-approximation of the absolute value function satisfying $|r| \le \chi_{\epsilon}(r)$, such as
	\begin{equation}
            \label{sup:eq:chi_def}
            \chi_{\epsilon}\left(r\right) = \begin{cases} \left|r\right| & \left|r\right|>\epsilon \\ 
            -\frac{r^{4}}{8\epsilon^{3}}+\frac{3r^{2}}{4\epsilon}+\frac{3\epsilon}{8} & \left|r\right|\le\epsilon
            \end{cases}. 
	\end{equation}
	Multiplying $\chi_{\epsilon}^{\prime}(\rho)$ to \eqref{sup:eq:mv_equation} and integrating by parts over $\Omega$, we have
	\begin{equation}
            \begin{split}
                \frac{d}{dt} \int_{\Omega} \chi_{\epsilon} \left(\rho(x,t)\right)\, dx + \sigma \|&\nabla \rho(t) \left[\chi_{\epsilon}^{\prime\prime}\left(\rho(t)\right)\right]^{1/2}\|^{2} = \\
                &-\int_{\Omega} \rho(x,t) \chi_{\epsilon}^{\prime\prime} \left(\rho(x,t)\right) \nabla v[\rho](x,t) \cdot \nabla \rho(x,t)\, dx.
            \end{split}
	\end{equation}
By Cauchy-Schwarz and Young's inequality, 
	\begin{align*}
            \frac{d}{dt} \int_{\Omega} \chi_{\epsilon} \left(\rho(x,t)\right)\, dx &+ \sigma \left\Vert \nabla \rho(t)\left[\chi_{\epsilon}^{\prime\prime}\left(\rho(t)\right)\right]^{1/2}\right\Vert^{2}\nonumber \\
            & \le \left\Vert \rho(t) \nabla v[\rho](t)\left[\chi_{\epsilon}^{\prime\prime}\left(\rho(t)\right)\right]^{1/2}\right\Vert \left\Vert \nabla\rho(t)\left[\chi_{\epsilon}^{\prime\prime}\left(\rho(t)\right)\right]^{1/2}\right\Vert\nonumber \\
            & \le \frac{1}{\sigma}\left\Vert \rho(t) \nabla v[\rho](t)\left[\chi_{\epsilon}^{\prime\prime}\left(\rho(t)\right)\right]^{1/2}\right\Vert^{2} + \sigma \left\Vert \nabla\rho(t)\left[\chi_{\epsilon}^{\prime\prime}\left(\rho(t)\right)\right]^{1/2}\right\Vert^{2},
	\end{align*}
	where we take $p=q=2$ and $\epsilon = \sqrt{\sigma/2}$ in Young's inequality.
	Therefore, 
	\begin{align*}
		\frac{d}{dt} \int_{\Omega} \chi_{\epsilon}\left(\rho\left(x,t\right)\right)\, dx & \le \frac{1}{\sigma}\left\Vert \rho(t) \nabla v[\rho](t)\left[\chi_{\epsilon}^{\prime\prime}\left(\rho(t)\right)\right]^{1/2}\right\Vert^{2}\nonumber \\
		& \le \frac{1}{\sigma}\left\Vert \nabla v[\rho](t)\right\Vert _{\infty}^{2}\left\Vert \rho(t)\left[\chi_{\epsilon}^{\prime\prime}\left(\rho(t)\right)\right]^{1/2}\right\Vert^{2}.
	\end{align*}
        Notice that $\left\Vert \nabla v[\rho](t)\right\Vert _{\infty}\le \|\nabla V\|_{\infty} + \|\nabla W \|_{\infty} \|\rho(t)\|_{1}$.
	Additionally,
	\begin{align}
            \label{sup:eq:second_one}
            \left\Vert \rho(t)\left[\chi_{\epsilon}^{\prime\prime}\left(\rho(t)\right)\right]^{1/2}\right\Vert ^{2} &= \int_{\Omega}\rho^{2}\chi_{\epsilon}^{\prime\prime}\left(\rho\right)\mathbf{1}_{\left\{ \left|\rho\right|>\epsilon\right\} }\, dx + \int_{\Omega}\rho^{2}\chi_{\epsilon}^{\prime\prime}\left(\rho\right)\mathbf{1}_{\left\{ \left|\rho\right|\le\epsilon\right\} }\, dx \nonumber \\
            &= \int_{\Omega}\frac{3\rho^{2}\left(\epsilon^{2}-\rho^{2}\right)}{2\epsilon^{3}}\mathbf{1}_{\left\{ \left|\rho\right|\le\epsilon\right\} }dx = \int_{\Omega} \frac{3\epsilon}{2} \, dx =  3\cdot 2^{d-1}\pi^d \epsilon,
	\end{align}
        where we used that for $\left|\rho\right|>\epsilon$, $\chi_{\epsilon}^{\prime\prime}\left(\rho\right)=0$ by construction.
	Therefore, if $y_{\epsilon}(t) \coloneq \int_{\Omega} \chi_{\epsilon}(\rho(x,t)) \, dx$, 
	\begin{equation*}
		\begin{split}
			\frac{d}{dt}y_{\epsilon}(t)  &\le K\epsilon \left(\|\nabla V\|_{\infty}^2 + 2\|\nabla V\|_{\infty}\|\nabla W\|_{\infty}\|\rho(t)\|_{1} + \|\nabla W\|_{\infty}^2\|\rho(t)\|_{1}^{2} \right) \\
			&\le K \epsilon \left(\|\nabla V\|_{\infty}^2 + 2\|\nabla V\|_{\infty}\|\nabla W\|_{\infty}y_{\epsilon}(t)  + \|\nabla W\|_{\infty}^2y_{\epsilon}(t)^{2} \right)
		\end{split}
	\end{equation*}
	for a constant $K > 0$.
	Applying the comparison principle, we get
	\begin{equation*}
            y_{\epsilon}(t) \le \frac{1}{\frac{1}{y_\epsilon(0) + \|\nabla V\|_{\infty}} - K\epsilon\|\nabla W\|_{\infty} t} - \|\nabla V\|_{\infty}, \quad t < \frac{1}{K\epsilon\|\nabla W\|_{\infty}}\frac{1}{y_{\epsilon}(0) +  \|\nabla V\|_{\infty}}.
	\end{equation*}
	First, notice that $y_{\epsilon}(t) \to \|\rho(t)\|_{1}$ as $\epsilon \to 0$.
	Therefore, taking the limit $\epsilon\rightarrow0$ yields
	\begin{equation*}
		\|\rho(t)\|_{1} \le \|\rho_{0}\|_{1},
	\end{equation*}
	for every $t\geq 0$, as required. 
\qquad\end{proof}

\Cref{sup:lem:L1} also establishes the nonnegativity of
$\rho$ indirectly.

\begin{corollary}\label{sup:cor:positivity}
    If $\rho \in C^{1}\left(0,\infty;C^{2}\left(\Omega\right)\right)$ is a solution of \eqref{sup:eq:mv_equation}, with $\rho_{0} \in \Pac$, then $\rho(t) \in \Pac$ for all $t \ge 0$. 
\end{corollary}

\begin{proof}
    Since $\int_{\Omega} \rho_{0}(x) \, dx = 1$, the normalization condition~\eqref{sup:eq:normalization_cond} is satisfied for $t>0$. 
    Applying \Cref{sup:lem:L1}, we have
    \begin{equation}
        1=\int_{\Omega} \rho(x,t)dx \le \|\rho(t)\|_{1} \le \|\rho_0\|_{1} = 1.
    \end{equation}
    Hence, $\|\rho(t)\|_{1}=1$. 
    But
    \begin{align*}
	1 &= \int_{\Omega} \rho(x,t) \, dx = \int_{\Omega} \rho \mathbf{1}_{\left\{ \rho\geq0\right\} } \, dx + \int_{\Omega} \rho\mathbf{1}_{\left\{ \rho<0\right\}}\, dx,\\
	1 &= \int_{\Omega} \left|\rho(x,t)\right| \, dx = \int_{\Omega} \rho\mathbf{1}_{\left\{ \rho\geq0\right\} } \, dx -\int_{\Omega}\rho\mathbf{1}_{\left\{ \rho<0\right\} } \, dx.
    \end{align*}
    These equations imply that $\int_{\Omega} \rho\mathbf{1}_{\left\{ \rho<0\right\} }dx=0$, and hence, $\rho(t)\geq0$ a.e. in $\Omega$. 
	By continuity, $\rho(t)\geq0$ in $\Omega$ for all $t\geq0$. 
\end{proof}

\subsection{Iterative scheme and regularity}

To construct a weak solution of the McKean-Vlasov equation, we first define a sequence of linear and local parabolic problems whose solutions we will prove later converge to the solution of the nonlinear dynamics. 
This iterative scheme replaces the nonlocal term $\nabla v[\rho]$ in each step with the corresponding quantity from the previous iterate, yielding a family $(\rho_n)_{n \ge 1}$ of smooth periodic solutions. 
In this subsection, we establish the existence of these approximations and derive useful uniform $L^1/L^2$ and $H^1/H^2$ bounds.

Let $T > 0$ and consider a sequence of linear parabolic equations 
\begin{equation}
    \label{sup:eq:sequence_evolution}
    \partial_t \rho_{n} - \sigma \Delta \rho_{n} = \nabla \cdot \left(\rho_{n} \nabla v[\rho_{n-1}] \right), \quad \text{in} \; \Omega \times (0,T),
\end{equation}
with periodic boundary conditions on $\Omega$ and $\rho_n(\cdot, 0) = \rho_0$ for all $n \ge 1$, and $\rho_0(\cdot, t) = \rho_0$ for all $t > 0$.
Assume initially that $\rho_0 \in C^{\infty}(\Omega) \cap \Pac$.
For the next lemma, consider the operator $\mathcal{G}_{\rho} y \coloneq \nabla \cdot \bigl(y \nabla v[\rho]\bigr)$ for a fixed $\rho 
 \in \Pac$.

\begin{lemma}
\label{sup:lem:G_rel_bound}
    For every $y \in D(\Delta) = H^2(\Omega)$ and $\epsilon > 0$, there exists a constant $C_\epsilon > 0$ such that
    \begin{equation}
        \label{sup:eq:G_bound}
        \|\mathcal{G}_{\rho} y\| \le \epsilon\,\bigl(\|\nabla V\|_{\infty} + \|\nabla W\|_{\infty} \|\rho\|_{1}\bigr) \|\Delta y\|
        + C_\epsilon\,\|y\|.
    \end{equation}
    In particular, $\mathcal{G}_{\rho}$ is $\Delta$-bounded for all $\epsilon > 0$.
\end{lemma}

\begin{proof}
    Let $y \in H^2(\Omega)$ and $\epsilon > 0$.
    By the product rule and Hölder’s inequality,
    \[
    \|\mathcal{G}_{\rho} y\| \le \|\nabla v[\rho]\|_{\infty} \|\nabla y\| + \|\Delta v[\rho]\|_{p} \|y\|_q, \quad \frac{1}{p} + \frac{1}{q} = \frac{1}{2}.
    \]
    From the assumptions on $V$ and $W$, and Young's convolution inequality,
    \begin{align*}
        \| \nabla v[\rho] \|_{\infty} &\le M_{\infty} \coloneq \|\nabla V\|_{\infty} + \|\nabla W\|_{\infty} \|\rho\|_{1}, \\
        \| \Delta v[\rho] \|_{p} &\le M_p \coloneq \|\Delta V\|_{p} + (2\pi)^{d/p}\|\Delta W\|_{\infty} \|\rho\|_{1}.
    \end{align*}
    By using Gagliardo--Nirenberg interpolation inequality and $p = \max\{d, 2+\epsilon\}$, 
    \[
    \|y\|_q \le C\|\nabla y\|_{2}^{n/p} \|y\|_2^{1-n/p} + C\|y\|_2 \le \frac{Cn}{p} \|\nabla y \|_2 + C\left(2+\frac{n}{p}\right)\|y\|_2.
    \]
    Using Cauchy-Schwarz and Young's inequality for products,
    \[
    \|\nabla y\| \le \epsilon \|\Delta y\| + \frac{1}{4\epsilon}\|y\| \implies \|\mathcal{G}_{\rho} y\| \le \epsilon\left(M_{\infty} + M_p \frac{Cn}{p}\right)\|\Delta y\| + C_\epsilon\,\|y\|,
    \]
    with
    \[
    C_\epsilon \coloneq M_p C \left(2 + \frac{n}{p}\right) + \frac{M_{\infty} + M_p Cn/p}{4\epsilon}.
    \]
This is precisely~\eqref{sup:eq:G_bound}. Since $\epsilon>0$ is arbitrary, the relative bound can be made arbitrarily small.
\end{proof}

\begin{proposition}
\label{sup:prop:analytic_sg}
    For any $\sigma > 0$, the operator $\sigma\Delta + \mathcal{G}_{\rho}$ generates an analytic semigroup $T(t)$ on $L^2(\Omega)$. 
    Moreover, for every $\alpha \ge 0$ and $t>0$,
    \[
    T(t) : L^2(\Omega) \to D\bigl((\sigma\Delta+\mathcal{G}_{\rho})^\alpha\bigr),  \quad  \|(\sigma\Delta+\mathcal{G}_{\rho})^\alpha\,T(t)\|_{\mathbb{L}(L^2)} \le M_\alpha\,t^{-\alpha}.
    \]
    In particular, for each $k \in \mathbb{N}$,
    \[
    u_0 \in L^2(\Omega) \Longrightarrow u(t) = T(t)u_0 \in H^{2k}(\Omega), \quad  \|\Delta^k u(t)\| \le C_k t^{-k} \|u_0\|.
    \]
\end{proposition}

\begin{proof}
    The operator $\mathcal{G}_{\rho}$ is $\Delta$-bounded for any $\epsilon > 0$ by \Cref{sup:lem:G_rel_bound}. 
    Since $\Delta$ with periodic boundary conditions generates an analytic semigroup on $L^2(\Omega)$, \cite[Theorem 3.2.1]{pazy2012semigroups} implies that $\sigma \Delta + \mathcal{G}_{\rho}$ does as well. 
    The domain regularity and decay estimates for the semigroup follow from \cite[Theorem 2.6.13]{pazy2012semigroups}.
\end{proof}

\begin{corollary}
    \label{sup:cor:smooth_iterates}
    Let $\rho_0 \in C^\infty(\Omega)$, and define the sequence $(\rho_n)_{n \ge 1}$ by the iterative scheme~\eqref{sup:eq:sequence_evolution}. 
    Then, $\rho_n \in C^\infty(0,T; C^\infty(\Omega))$ for all $n \ge 1$.
\end{corollary}

\begin{proof}
    For $n=1$, \eqref{sup:eq:sequence_evolution} is a linear parabolic equation with coefficients in $C^\infty(\Omega)$, so \Cref{sup:prop:analytic_sg} implies $\rho_1 \in C^\infty((0,T); C^\infty(\Omega))$. The same reasoning applies inductively to each $n>1$ since $\rho_{n-1}$ is smooth and, therefore, $\nabla v[\rho_{n-1}]$ is smooth, preserving the regularity of the solution $\rho_n$.
\end{proof}

Given the sequence of smooth functions $(\rho_n)_{n \ge 1}$, we record two additional properties of $\rho_n$: an $L^1$-norm bound and preservation of the probability-density property whenever $\rho_0$ is a probability density.

\begin{proposition}
    \label{sup:prop:unif_L1_est}
    Let $T>0$ and suppose $(\rho_{n})_{n \ge 1}$ satisfy \eqref{sup:eq:sequence_evolution} with $\rho_{0}\in C^{\infty}(\Omega)$.
    Then, $\|\rho_{n}(t)\|_{1} \le \|\rho_0\|_{1}$ for all $0\le t \le T$ and $n \ge 1$. 
\end{proposition}

\begin{proof}
    Since we know that $\rho_{n}(t)\in C^{\infty}\left(0,T;C^{\infty}(\Omega)\right)$ for all $n \ge 1$ and all $0 \le t \le T$, we can proceed exactly as in the proof of \Cref{sup:lem:L1}.
    We obtain
    \begin{align*}
        \frac{d}{dt} \int_{\Omega} \chi_{\epsilon}(\rho_n(x,t)) \, dx \le K\epsilon &(\|\nabla V\|_{\infty}^2 + 2\|\nabla V\|_{\infty}\|\nabla W\|_{\infty}\|\rho_{n-1}(t)\|_{1} \\ 
        &+ \|\nabla W\|_{\infty}^2\|\rho_{n-1}(t)\|_{1}^{2} ).
    \end{align*}
By induction, suppose $\|\rho_{n-1}(t)\|_1 \le \|\rho_0
	\|_1$ (for $n=1$ this is trivially true), then we have that
	\begin{align*}
		\int_{\Omega}\chi_{\epsilon}(\rho_{n}(x,t)) \, dx \le \left(\int_{\Omega}\chi_{\epsilon}\left(\rho_{0}\left(x\right)\right)dx\right) + K\epsilon \left(\|\nabla V\|_{\infty}^2 + \|\nabla W\|_{\infty}\|\rho_0\|_{1} \right)^2 t	
	\end{align*}
	Hence we take the limit $\epsilon \to 0$ to obtain
	\begin{equation}
		\left\Vert \rho_{n}(t)\right\Vert _{1}\le\|\rho_0\|_{1},
	\end{equation}
	which concludes the induction argument.
\end{proof}

Using \Cref{sup:prop:unif_L1_est} and an identical proof of \Cref{sup:cor:positivity} leads to the following corollary.

\begin{corollary}
    \label{sup:cor:positivity_n}
    Let $T>0$ and suppose $(\rho_{n})_{n \ge 1}$ satisfy~\eqref{sup:eq:sequence_evolution} with $\rho_{0} \in C^{\infty}(\Omega) \cap \Pac$. 
    Then, $\rho_{n}(t) \in \Pac$ for all $0 \le t \le T$ and for all $n \ge 1$. 
\end{corollary}

The next two propositions then provide two a priori estimates.

\begin{proposition}
    \label{sup:prop:unif_L2_est}
    Under the settings of \Cref{sup:cor:positivity_n}, there exists a constant $C(T)>0$ such that 
    \begin{equation*}
        {\left\Vert \rho_{n} \right\Vert}_{L^\infty\left(0,T;L^2(\Omega)\right)} + {\left\Vert \rho_{n} \right\Vert}_{L^2(0,T;H^1(\Omega))} \le C(T)\|\rho_0\|.
    \end{equation*}
\end{proposition}

\begin{proof}
    We multiply \eqref{sup:eq:sequence_evolution} by $\rho_{n}$ and integrate by parts. 
    This gives us 
	\begin{align*}
		\frac{1}{2} \frac{d}{dt}\|\rho_{n}(t)\|^{2} + \sigma \|\nabla \rho_{n}(t)\| ^{2} 
		& \le \int_{\Omega}\left|\rho_{n}\nabla \rho_{n} \cdot \nabla v[\rho_{n-1}]\right|dx\nonumber \\
		& \le \|\nabla \rho_{n}(t)\|  \|\rho_{n}(t)\nabla v[\rho_{n-1}](t)\|  \nonumber \\
		& \le \frac{\sigma}{2}\|\nabla \rho_{n}(t)\| ^{2} +\frac{1}{2\sigma}\|\rho_{n}(t)\| ^{2}\|\nabla v[\rho_{n-1}](t)\| _{\infty}^{2} \nonumber \\
		 \le \frac{\sigma}{2}\|&\nabla \rho_{n}(t)\| ^{2} +\frac{1}{2\sigma}\|\rho_{n}(t)\| ^{2}\| \left(\|\nabla V\|_{\infty} + \|\nabla W\|_{\infty} \|\rho_{n-1}\|_1\right)^2 \nonumber \\
		&= \frac{\sigma}{2}\|\nabla \rho_{n}(t)\| ^{2} +\frac{1}{2\sigma}\|\rho_{n}(t)\| ^{2}\| \left(\|\nabla V\|_{\infty} + \|\nabla W\|_{\infty}\right)^2,
	\end{align*}
	which implies
	\begin{equation*}
		\frac{d}{dt} \| \rho_{n}(t)\|^{2} + \sigma \| \nabla \rho_{n}(t)\|^{2} \le \frac{\left(\|\nabla V\|_{\infty} + \|\nabla W\|_{\infty} \right)^2}{\sigma} \|\rho_{n}(t)\|^{2}, 
	\end{equation*}
	and by Grönwall's lemma,
        \begin{equation}
            \label{sup:eq:rhonCub}
            \|\rho_{n}(t)\| ^{2} \le \exp\left\{ R t\right\}\|\rho_0\|^{2},
	\end{equation}
	for all $0\le t\le T$, where we set $R \coloneq \frac{\left(\|\nabla V\|_{\infty} + \|\nabla W\|_{\infty} \right)^2}{\sigma}$, and
	\begin{equation}
		\begin{split}
			\int_{0}^{T} \|\rho_{n}(t)\|^{2} + \| \nabla \rho_{n}(t)\|^{2} \, dt \le \left(\frac{R+\sigma}{\sigma}\frac{e^{RT} - 1}{R} + 1\right) \|\rho_0\|^2,
		\end{split}
	\end{equation}
	which combined with \eqref{sup:eq:rhonCub} gives the result.
\end{proof}

\begin{proposition}
	\label{sup:prop:unif_L2_est_higher}
	Under the settings of \Cref{sup:cor:positivity_n}, there exists a constant $C(T)>0$ such that 
	\begin{equation*}
	{\left\Vert \rho_{n} \right\Vert}_{L^\infty\left(0,T;H^1(\Omega)\right)} + {\left\Vert \rho_{n} \right\Vert}_{L^2(0,T;H^2(\Omega))} \le C(T) \|\rho_0\|_{H^1(\Omega)}.
	\end{equation*}
\end{proposition}

\begin{proof}
    Let $C_1(T)$ be the constant given by~\Cref{sup:prop:unif_L2_est}.
    Multiplying~\eqref{sup:eq:sequence_evolution} by $-\Delta\rho_{n}$ and integrating by parts over $\Omega$, it follows from Cauchy-Schwarz and Young's inequality, that
    \begin{align*}
        \frac{1}{2}\frac{d}{dt}\|\nabla &\rho_{n}(t)\|^{2} + \sigma \|\Delta \rho_{n}(t)\|^{2} \le \int_{\Omega} \left|\nabla \cdot \left(\rho_{n} \nabla v[\rho_{n-1}]\right) \Delta \rho_{n} \right|dx  \\
        & \le \frac{\sigma}{2}\|\Delta \rho_{n}(t)\|^{2} + \frac{1}{2\sigma}\|\rho_n(t) \Delta v[\rho_{n-1}]\|^{2} + \frac{1}{2\sigma}\|\nabla \rho_n(t) \cdot \nabla v[\rho_{n-1}]\|^{2} \nonumber \\
        & \le \frac{\sigma}{2}\|\Delta \rho_{n}(t)\|^{2} + \frac{1}{2\sigma}\|\Delta v[\rho_{n-1}]\|_{p}^2\|\rho_n(t)\|_q^{2} + \frac{1}{2\sigma} \|\nabla v[\rho_{n-1}]\|_{\infty}^2 \|\nabla \rho_n(t)\|^{2} \nonumber \\
        &\le \frac{\sigma}{2}\|\Delta \rho_{n}(t)\|^{2} + \frac{R'}{2}\|\rho_n(t)\|_q^{2} + \frac{R}{2}\|\nabla \rho_n(t)\|^{2} \nonumber \\
        &\le \frac{\sigma}{2}\|\Delta \rho_{n}(t)\|^{2} + C\frac{R'}{2}\|\rho_n(t)\|_{H^1(\Omega)}^{2} + \frac{R}{2}\|\nabla \rho_n(t)\|^{2}, \nonumber
    \end{align*}
    with $R' = \frac{\left(\|\Delta V\|_{p} + (2\pi)^{d/p}\|\Delta W\|_{p}\right)^2}{\sigma}$ and applying Gagliardo–Nirenberg for the last inequality.
    It follows that
    \begin{align*}
         \frac{d}{dt}\|\nabla \rho_{n}(t)\|^{2} + \sigma\|\Delta \rho_{n}(t)\|^{2} &\le (R + R'C)\|\nabla \rho_{n}(t)\|^{2} + R'C\|\rho_{n}(t)\|^{2} \\
         &\le (R + R'C)\|\nabla \rho_{n}(t)\|^{2} + R'C e^{Rt}\|\rho_0\|^{2}.
    \end{align*}
    Grönwall's lemma gives for, $0 \le t \le T$,
    \[
    \|\nabla \rho_{n}(t)\|^{2} \le e^{(R+R'C) t}\Bigl(\|\nabla \rho_0\|^{2} + R'C\|\rho_{0}\|^{2} \frac{e^{Rt} - 1}{R+R'C}\Bigr).
    \]
    Hence
    \[
    \sup_{0\le t\le T}\|\nabla\rho_{n}(t)\|
    \le C_2(T)\,\bigl(\|\nabla\rho_{0}\|+\|\rho_{0}\|\bigr)
    \]
    for a constant $C_2(T)$.
    Next, integrate the inequality
    \[
    \frac{d}{dt}\|\nabla\rho_{n}\|^{2} +\sigma\|\Delta\rho_{n}\|^{2} \le (R+R'C)\,\|\nabla\rho_{n}\|^{2}+R'C\|\rho_{n}\|^{2}
    \]
    over \([0,T]\).
    Using the bounds just obtained and \(\|\rho_{n}(t)\|^{2} \le e^{Rt}\|\rho_{0}\|^{2}\), one finds
    \begin{align*}
        \int_{0}^{T}\|\Delta\rho_{n}(t)\|^{2}\,dt &\le \frac{R+R'C}{\sigma}\int_{0}^{T}\|\nabla\rho_{n}\|^{2}dt + \frac{R'C}{\sigma}\int_{0}^{T}\|\rho_{n}\|^{2} \, dt + \frac{1}{\sigma}\|\nabla\rho_{0}\|^{2} \\
        &\le C_{3}^2(T)\,\bigl(\|\nabla\rho_{0}\|^{2}+\|\rho_{0}\|^{2}\bigr)
    \end{align*}
    for a constant $C_3(T) > 0$.

    Using the fact that $\|\Delta u\|_2^2 = \sum_{i,j} \|\partial_{x_i,x_j} u\|_2^2$ and combining these two estimates with \Cref{sup:prop:unif_L2_est}, we obtain
    \[
    \|\rho_{n}\|_{L^{\infty}(0,T;H^1(\Omega))} + \|\rho_{n}\|_{L^{2}(0,T;H^2(\Omega))} \le C(T)\|\rho_{0}\|_{H^1(\Omega)},
    \]
    as claimed.  
    This completes the proof.
\end{proof}

\subsection{Convergence}

With the uniform estimates above, we can now show that $\rho_{n}$ converges strongly to a limit in $L^1(0, T; L^1(\Omega))$ and weakly in $L^2(0, T; H^1(\Omega))$. 

\begin{lemma}
    \label{sup:lem:Cauchy}
    Under the settings of \Cref{sup:cor:positivity_n}, there exists a limit function $\rho \in L^{1}(0,T; L^{1}(\Omega))$ such that $\rho_{n} \rightarrow \rho$ in $L^{1}\left(0,T;L^{1}(\Omega)\right)$.
\end{lemma}

\begin{proof}
	We set $\phi_{n}=\rho_{n} - \rho_{n-1}$ for $n\geq 1$. 
	For $n\geq 2$, the evolution equation for $\phi_{n}$ reads as
	\begin{equation*}
            \partial_t \phi_{n} - \sigma \Delta \phi_{n} = \nabla \cdot   \left(\phi_{n}\nabla v[\rho_{n-1}]+\rho_{n-1} \nabla W * \phi_{n-1}\right).
	\end{equation*}
        Let $\epsilon > 0$. 
        Multiplying the equation above by $\chi^{\prime}_{\epsilon} \left(\phi_{n}\right)$ (see~\eqref{sup:eq:chi_def}) and integrating by parts yields
    \begin{align}
        \label{sup:eq:phi_n_main1}
        \frac{d}{dt} \int_{\Omega} &\chi_{\epsilon}\left(\phi_{n}\right) dx
        + \sigma \left\|\left[\chi^{\prime\prime}_{\epsilon}\left(\phi_n(t)\right)\right]^{1/2} \nabla \phi_{n}(t)\right\|^{2} \\ 
        &\le \int_{\Omega} |\chi^{\prime\prime}_{\epsilon}\left(\phi_n\right) \phi_{n} \nabla \phi_{n} \cdot \nabla v[\rho_{n-1}]|dx + \int_{\Omega}|\chi^{\prime}_{\epsilon}\left(\phi_n\right) \nabla \cdot \left(\rho_{n-1} \nabla W * \phi_{n-1}\right)|dx\nonumber \\
        &\le \sigma \left\|\left[\chi^{\prime\prime}_{\epsilon}\left(\phi_n(t)\right)\right]^{1/2}\nabla \phi_{n}(t)\right\|^{2} + \frac{1}{4\sigma} \|\nabla v[\rho_{n-1}](t)\|^{2}_{\infty}
        \left\|\left[\chi^{\prime\prime}_{\epsilon}\left(\phi_n(t)\right)\right]^{1/2} \phi_{n}(t)\right\|^2 \nonumber \\
        &\quad + \int_{\Omega}|\chi^{\prime}_{\epsilon}\left(\phi_n\right) \nabla \cdot \left(\rho_{n-1} \nabla W * \phi_{n-1}\right)|dx \nonumber.
    \end{align}
    By \Cref{sup:cor:positivity_n}, $\|\nabla v[\rho_{n-1}](t)\|_{\infty} \le \sigma R\|\rho_{n-1}(t)\|_{1} = \sigma R$. 
    Also, similar to \eqref{sup:eq:second_one} from the proof of \Cref{sup:lem:L1}, we have 
    \begin{equation*}
        \|\left[\chi^{\prime\prime}_{\epsilon} \left(\phi_n(t)\right)\right]^{1/2}\phi_{n}(t)\|^{2} \le C \epsilon,
    \end{equation*}
    for some constant $C > 0$.
    To estimate the last integral term in~\eqref{sup:eq:phi_n_main1}, observe that $|\chi^{\prime}_{\epsilon}| \le \max\{2, 1+3\epsilon\}/2$, then
    \begin{align*}
        \int_{\Omega}|\chi^{\prime}_{\epsilon}&\left(\phi_n\right) \nabla \cdot (\rho_{n-1} \nabla W * \phi_{n-1})| \, dx \\ 
        &\le K \int_{\Omega}|\nabla \rho_{n-1} \cdot \nabla W * \phi_{n-1}| \, dx + K\int_{\Omega}|\rho_{n-1} \Delta W * \phi_{n-1} | \, dx \\
        &\le K (2\pi)^{d/2} \|\nabla W\|_{\infty} \|\phi_{n-1}\|_1 \| \nabla \rho_{n-1} \| + K (2\pi)^{d/2} \|\Delta W\|_{\infty} \|\phi_{n-1} \|_1 \|\rho_{n-1}\| \\
        &\le K (2\pi)^{d/2} (\|\nabla W\|_{\infty} + \|\Delta W\|_{\infty}) \|\phi_{n-1} \|_1 \|\rho_{n-1}\|_{H^1(\Omega)}.
    \end{align*}
    Hence, it follows that, as $\epsilon \to 0$, \eqref{sup:eq:phi_n_main1} becomes 
    \begin{align}
        \label{sup:eq:phi_n_main2}
        \frac{d}{dt} \|\phi_{n}(t) \|_{1} &\le K (2\pi)^{d/2} (\|\nabla W\|_{\infty} + \|\Delta W\|_{\infty}) \|\phi_{n-1}\|_1 \|\rho_{n-1}\|_{H^1(\Omega)} \\
        &\le C\left(\rho_0 ;1,T\right) \|\phi_{n-1} (t) \|_{1} \nonumber.
    \end{align}
    In the last line, we used \Cref{sup:prop:unif_L2_est_higher} and the shorthand 
    \begin{equation*}
        C\left(\rho_0 ;1,T\right) \coloneq C(T) K (2\pi)^{d/2} (\|\nabla W\|_{\infty} + \|\Delta W\|_{\infty}) \|\rho_0\|_{H^{1}(\Omega)}.
    \end{equation*}
    Now, for $N \ge 2$ we define $y_{N}(t) \coloneq \sum_{n=2}^{N}\|\phi_{n}(t)\|_{1}$.
    By~\eqref{sup:eq:phi_n_main2} and \Cref{sup:cor:positivity_n},
    \begin{equation*}
        \frac{d}{dt}y_{N}(t) \le C\left(\rho_0; 1, T\right) \left(y_N(t) + \|\phi_{1}(t)\|_{1} - \|\phi_{N}(t)\|_{1} \right) \nonumber \le C\left(\rho_0 ;1,T\right) \left(y_N(t) + 4 \right). 
    \end{equation*}
    Moreover, $y_{N}(0) = 0$.
    Thus, by Gr\"onwall's inequality,
    \begin{equation*}
        y_{N}(t) \le 4 (e^{C\left(\rho_0 ;1,T\right) T} - 1),
    \end{equation*}
    uniformly in $N$ and $t$.
    Furthermore, for each $t$, $y_{N}(t)$ is a bounded monotone sequence in $N$, hence there exists 
    \begin{equation*}
        y_{\infty}(t) = \sum_{n=2}^{\infty}\|\phi_{n}(t)\|_{1}\le 4 (e^{C\left(\rho_0 ;1,T\right) T} - 1),
    \end{equation*}
    such that $y_{N}(t) \uparrow y_{\infty}(t)$ pointwise in $t$.
    By the monotone convergence theorem, 
    \begin{equation*}
        \int_{0}^{T}y_{N}(t) \, dt \uparrow \int_{0}^{T} y_{\infty}(t)\, dt \le 4 T (e^{C\left(\rho_0 ;1,T\right) T} - 1).
    \end{equation*}
    This result implies that $(\rho_{n})_{n \ge 1}$ is a Cauchy sequence in $L^{1}(0,T;L^{1}(\Omega))$.
    Indeed, for any $\epsilon > 0$ we can pick $N \ge 2$ such that $\int_{0}^{T} y_{\infty}(t) \, dt - \int_{0}^{T} y_{N}(t) \, dt < \epsilon$. 
    Hence, for all $M\geq 1$, 
    \begin{align*}
        \|\rho_{N+M}-\rho_{N}\|_{L^{1}(0,T;L^{1}(\Omega))} 
        &= \int_{0}^{T}\|\rho_{N+M}(t)-\rho_{N}(t)\|_{1} \, dt = \int_{0}^{T}\left\|\sum_{n=N+1}^{N+M} \phi_{n}(t)\right\|_{1} \, dt \\
        &\le \int_{0}^{T}\sum_{n=N+1}^{N+M} \|\phi_{n}(t)\|_{1} \, dt = \int_{0}^{T}y_{N+M}(t) \, dt - \int_{0}^{T} y_{N}(t) \, dt \\
        &\le \int_{0}^{T} y_{\infty}(t) \, dt - \int_{0}^{T} y_{N}(t) \, dt \le\epsilon.
    \end{align*}
    Therefore, $(\rho_{n})_{n \ge 1}$ is a Cauchy sequence and there exists $\rho \in L^{1}(0,T;L^{1}(\Omega))$ such that $\rho_{n}\rightarrow\rho$ in $L^{1}(0,T;L^{1}(\Omega))$. 
\end{proof}

The next result states that we can extract from $(\rho_{n})_{n \ge 1}$ a subsequence that converges weakly in smaller spaces. 

\begin{lemma}
    \label{sup:lem:weakconv}
    We have $\rho \in L^{2}(0,T;H^{1}(\Omega)) \cap L^{\infty}(0,T;L^{2}(\Omega))$ satisfying $\partial_t \rho \in L^{2}(0,T;H^{-1}(\Omega))$ and
    \begin{equation*}
        \|\rho\|_{L^{\infty}(0,T;L^{2}(\Omega))} + \|\rho\|_{L^{2}(0,T;H^{1}(\Omega))} + \|\partial_t\rho\|_{L^{2}(0,T;H^{-1}(\Omega))} \le C(T) \|\rho_0\| .
    \end{equation*}
    Moreover, there exists a subsequence $(\rho_{n_{k}})_{k \ge 1}$ such that 
    \[
    \rho_{n_{k}} \rightharpoonup \rho \text{ in } L^{2}(0,T;H^{1}(\Omega)) \quad \text{ and } \quad \partial_t \rho_{n_{k}} \rightharpoonup \partial_t \rho \text{ in } L^{2}(0,T;H^{-1}(\Omega)).
    \]
\end{lemma}

\begin{proof}
    From \Cref{sup:prop:unif_L2_est}, we have 
    \begin{equation}
        \|\rho_n\|_{L^{\infty}(0,T;L^{2}(\Omega))} + \|\rho_n\|_{L^{2}(0,T;H^{1}(\Omega))} \le C(T)\|\rho_0\|.
    \end{equation}
    Next, from the evolution equation of $\rho_{n}$, we have $\partial_t \rho_{n} = \nabla \cdot \left(\nabla v[\rho_{n-1}]\rho_{n} + \sigma \nabla \rho_{n}\right)$.
    Hence, 
    \begin{equation*}
    \begin{split}
        \|\partial_t \rho_n\|_{L^2(0, T; H^{-1}(\Omega))}^2 &\le \int_0^T \|\nabla v[\rho_{n-1}](t) \rho_n(t) + \sigma \nabla \rho_n(t)\|^2 \, dt \\
        &\le 2\int_0^T \|\nabla v[\rho_{n-1}](t) \rho_n(t)\|_2^2 \, dt + 2\sigma \int_0^T \|\nabla \rho_n(t)\|_2^2 \, dt \\
        &\le 2R^2\sigma^2\int_0^T \|\rho_n(t)\|^2 \, dt + 2\sigma \int_0^T \|\nabla \rho_n(t)\|^2 \, dt \\
        &\le K \|\rho_n\|_{L^2(0,T; H^1(\Omega))}
    \end{split}
    \end{equation*}
    for some constant $K > 0$.
    Therefore, we have the uniform estimate 
    \begin{equation}
        \label{sup:eq:unif_3_ineq}
        \|\rho_{n}\|_{L^{\infty}(0,T;L^{2}(\Omega))} + \|\rho_{n}\|_{L^{2}(0,T;H^{1}(\Omega))} + \|\partial_t \rho_{n}\|_{L^{2}(0,T;H^{-1}(\Omega))}\le C(T)(1+K)\|\rho_0\| .
    \end{equation}
    Bounded sequences in reflexive Banach spaces guarantee the existence of a subsequence $(\rho_{n_k})_{k \ge 1}$ such that 
    \begin{equation*}
        \begin{cases}
            \rho_{n_{k}}\rightharpoonup\rho & \mbox{in }L^{2}(0,T;H^{1}(\Omega)), \\
            \partial_t \rho_{n_{k}} \rightharpoonup \partial_t \rho & \mbox{in }L^{2}(0,T;H^{-1}(\Omega)).
        \end{cases}
    \end{equation*}
    By the lower-semicontinuity of the norm, $\rho \in L^{2}(0,T;H^{1}(\Omega)) \cap L^{\infty}(0,T;L^{2}(\Omega))$, with $\partial_t \rho \in L^{2}(0,T;H^{-1}(\Omega))$ and they satisfy the same estimate of~\eqref{sup:eq:unif_3_ineq}.
\end{proof}

Following \cite[Theorem 3, Section 5.9]{evans2022partial}, we can deduce from \Cref{sup:lem:weakconv} the following result:

\begin{theorem}
    \label{sup:thm:Evans}
    Suppose $\rho\in L^{2}(0,T;H^{1}(\Omega))$
    with $\partial_t \rho \in L^{2}(0,T;H^{-1}(\Omega))$,
    then up to a set of measure zero $\rho\in C(0,T;L^{2}(\Omega))$. 
    Further, the mapping
    \begin{equation*}
        t\mapsto\|\rho(t)\|^{2}
    \end{equation*}
    is absolutely continuous, with 
    \begin{equation*}
        \frac{d}{dt}\|\rho(t)\|^{2} = 2 \langle \partial_t \rho(t), \rho(t)\rangle_{H^{-1}, H^{1}}, \quad \text{a.e. } 0 \le t\le T.
    \end{equation*}
\end{theorem}

\begin{proof}
    The proof is identical to the proof in \cite{evans2022partial}. 
    The only difference here is that we consider $H^{1}$ and $H^{-1}$, instead of $H_{0}^{1}$ and $H^{-1}$.
    Since periodic conditions still guarantee integration by parts without boundary terms, all proofs follow through. 
\qquad\end{proof}

\subsection{Existence and uniqueness of weak solutions}

Now, we are ready to prove the existence of a weak solution to~\eqref{sup:eq:mv_equation}.

\begin{theorem}[Existence and uniqueness]
    \label{sup:thm:existence}
    Let $\rho_{0}\in C^{\infty}(\Omega) \cap \Pac$. 
    Then, there exists a unique weak solution $\rho$ to equation~\eqref{sup:eq:mv_equation} with the estimate 
    \[
    \|\rho\|_{L^{\infty}(0,T;L^{2}(\Omega))} + \|\rho\|_{L^{2}(0,T;H^{1}(\Omega))} + \|\partial_t \rho\|_{L^{2}(0,T;H^{-1}(\Omega))} \le C(T)\|\rho_0\|.
    \]
\end{theorem}

\begin{proof}
    Let $(\rho_n)_{n \ge 1}$ be a sequence that converges weakly to $\rho$ in $L^2(0, T; H^1(\Omega))$.
    For each $\eta \in L^{2}(0,T;H^{1}(\Omega))$, we multiply equation~\eqref{sup:eq:sequence_evolution} by $\eta$ and integrate over $\Omega_{T}$ to obtain
    \begin{equation*}
        \begin{split}
            \int_{0}^{T} \left\langle \partial_t \rho_{n}(t), \eta(t)\right\rangle \, dt &+ \sigma \int_{0}^{T} \int_{\Omega} \nabla \rho_{n}(t) \cdot \nabla \eta(t) \, dx \, dt \\
            &+ \int_{0}^{T} \int_{\Omega} \rho_{n}(t) \nabla v[\rho_{n-1}](t) \cdot \nabla \eta(t) \, dx \, dt = 0.
        \end{split}
    \end{equation*}
    Notice that 
    \begin{equation}
        \label{sup:eq:conv0}
        \begin{split}
            \int_0^T \int_{\Omega} \rho_n \nabla v[\rho_{n-1}] &\cdot \nabla \eta \, dx\,dt = \int_0^T \int_{\Omega} (\rho_{n} - \rho) \nabla v[\rho_{n-1}] \cdot \nabla \eta \, dx \, dt \\
            &+ \int_0^T \int_{\Omega} \rho \nabla v[\rho] \cdot \nabla \eta \, dx \, dt + \int_0^T \int_{\Omega} \rho \nabla v[\rho_{n-1}] \cdot \nabla \eta \\ 
            &- \int_0^T \int_{\Omega} \rho \nabla v[\rho] \cdot \nabla \eta \, dx \, dt \\
            &= \int_0^T \int_{\Omega} (\rho_{n} - \rho) \nabla v[\rho_{n-1}] \cdot \nabla \eta \, dx \, dt + \int_0^T \int_{\Omega} \rho \nabla v[\rho] \cdot \nabla \eta \, dx \, dt \\
            &+ \int_0^T \int_{\Omega} \rho \nabla W * (\rho_{n-1} - \rho) \cdot \nabla \eta \, dx \, dt.
        \end{split}
    \end{equation}
    From \Cref{sup:lem:Cauchy}, we know that $\rho_n \to \rho$ in $L^1(0, T; L^1(\Omega))$, so 
    \[
    \|\nabla v[\rho_{n}] - \nabla v[\rho]\|_{L^\infty(0, T; L^\infty(\Omega))}\le \|\nabla W\|_{\infty} \|\rho_n - \rho\|_{L^1(0, T; L^1(\Omega))} \to 0
    \]
    so that 
    \begin{align*}
        \int_0^T \int_{\Omega} |(\nabla v[\rho_{n-1}] - \nabla v[\rho]) &\cdot \nabla \eta|^2 \, dx \, dt \\ 
        &\le \|\nabla v[\rho_{n}] - \nabla v[\rho]\|^2_{L^\infty(0, T; L^\infty(\Omega))} \|\eta\|^2_{L^2(0, T; H^1(\Omega))} \to 0.
    \end{align*}
    Thus, a weak convergence of $\rho_n$ plus a strong convergence of $\nabla v[\rho_{n-1}] \cdot \nabla \eta$ implies
    \[
    \int_0^T \int_{\Omega} (\rho_n - \rho) \nabla v[\rho_{n-1}] \cdot \nabla \eta \, dx \, dt \to 0.
    \]
    Also, 
    \begin{align*}
        \Bigl|\int_0^T \int_{\Omega} \rho \nabla W * &(\rho_{n-1} - \rho) \cdot \nabla \eta \, dx \, dt \Bigr| \\
        &\le \|\nabla W\|_{\infty} \|\rho_{n-1} - \rho\|_{L^1(0, T; L^1(\Omega))} \|\rho\|_{L^2(0, T; L^2(\Omega))} \|\eta\|_{L^2(0,T; H^1(\Omega))},
    \end{align*}
    which also converges to $0$ when $n \to \infty$.
    
    Taking the limit ($n \to \infty$) on \eqref{sup:eq:conv0} leads to
    \begin{equation}
        \label{sup:eq:conv_1}
        \int_{0}^{T} \int_{\Omega} \rho_{n} \nabla v[\rho_{n-1}] \cdot \nabla \eta \, dx \, dt \rightarrow \int_{0}^{T} \int_{\Omega} \rho \nabla v[\rho] \cdot \nabla \eta \, dx \, dt.
    \end{equation}
    By the weak convergence results established in \Cref{sup:lem:weakconv}, we also have
    \begin{align*}
        \label{sup:eq:conv_2}
        \int_{0}^{T} \left\langle \partial_t \rho_{n}(t), \eta(t)\right\rangle \, dt &\rightarrow \int_{0}^{T}\left\langle \partial_t \rho(t),\eta(t) \right\rangle \, dt, \\ \int_{0}^{T}\int_{\Omega} \nabla \rho_{n} \cdot \nabla \eta \, dx \, dt &\rightarrow \int_{0}^{T} \int_{\Omega} \nabla \rho \cdot \nabla \eta \, dx \, dt.
    \end{align*}
    Putting together~\eqref{sup:eq:conv_1} and~\eqref{sup:eq:conv_2}, we obtain in the limit $n \rightarrow \infty$, 
    \begin{equation}
        \label{sup:eq:weak_form}
        \int_{0}^{T} \left\langle \partial_t \rho(t),\eta(t)\right\rangle \, dt + \int_{0}^{T} \int_{\Omega} \left(\sigma \nabla \rho + \rho \nabla v[\rho]\right)\cdot \nabla \eta \, dx \, dt = 0, 
    \end{equation}
    for every $\eta \in L^{2}(0,T; H^{1}(\Omega))$.
    
    Finally, we have to show that $\rho(0) = \rho_0$. 
    Pick some $\eta \in C^1(0,T;H^1(\Omega))$ with $\eta(T) = 0$.
    Then, we have from~\eqref{sup:eq:weak_form} that 
    \begin{equation}
        \label{sup:eq:comp_1}
        -\int_{0}^{T} \left\langle \rho(t), \partial_t \eta(t)\right\rangle \, dt + \int_{0}^{T} \int_{\Omega} \left(\sigma \nabla \rho + \rho \nabla v[\rho]\right)\cdot \nabla \eta \, dx \, dt = \langle \rho(0), \eta(0) \rangle. 
    \end{equation}
    Similarly, we also have 
    \begin{equation}
        \label{sup:eq:comp_2}
        -\int_{0}^{T} \left\langle \rho_n(t), \partial_t \eta(t)\right\rangle \, dt + \int_{0}^{T} \int_{\Omega} \left(\sigma \nabla \rho_n+ \rho_n \nabla v[\rho_{n-1}]\right)\cdot \nabla \eta \, dx \, dt = \langle \rho_0, \eta(0) \rangle.
    \end{equation}
    Taking the limit $n \rightarrow \infty$ and comparing~\eqref{sup:eq:comp_1} to~\eqref{sup:eq:comp_2}, $\langle \rho(0), \eta(0) \rangle = \langle \rho_0, \eta(0) \rangle$.
    Since $\eta$ is arbitrary, we conclude that $\rho(0) = \rho_0$. 
    This completes the proof of the existence of a weak solution. 
    
    Now, we prove its uniqueness. 
    Let $\rho_{\dagger}$ and $\rho_{\ddagger}$ be weak solutions to~\eqref{sup:eq:mv_equation} and set $\xi = \rho_{\dagger} - \rho_{\ddagger}$.
    Then, for every $\eta \in L^{2}(0,T; H^{1}(\Omega))$, we have 
    \begin{align*}
        \int_{0}^{T}\left\langle \partial_t \xi (t),\eta(t) \right\rangle \, dt &+ \sigma \int_{0}^{T} \int_{\Omega} \nabla \xi \cdot \nabla \eta \, dx \, dt \\
        &+ \int_{0}^{T} \int_{\Omega} \left(\xi \nabla v[\rho_{\dagger}] + \rho_{\ddagger} \nabla W * \xi \right) \cdot \nabla \eta \,dx \, dt = 0.
    \end{align*}
    By the same renormalization argument from \Cref{sup:lem:L1}, $\|\rho_{\dagger}\|_1 \le 1$ and $\|\rho_{\ddagger}\|_1 \le 1$, which implies
    \begin{align*}
        \left|\int_{0}^{T}\int_{\Omega}\xi \nabla v[\rho_{\dagger}] \cdot \nabla \eta \, dx \, dt\right| &\le \|\nabla v[\rho_{\dagger}]\|_{L^{\infty}(0,T;L^{\infty}(\Omega))} \|\nabla \eta\|_{L^{2}(0,T;L^{2}(\Omega))} \|\xi\|_{L^{2}(0,T;L^{2}(\Omega))} \\
        &\le R\sigma \|\nabla \eta\|_{L^{2}(0,T;L^{2}(\Omega))}\|\xi\|_{L^{2}(0,T;L^{2}(\Omega))} \\
        &\le \frac{\sigma}{2}\|\nabla \eta\|_{L^{2}(0,T;L^{2}(\Omega))}^{2} + \frac{R^2 \sigma}{2} \|\xi\|_{L^{2}(0,T;L^{2}(\Omega))}^{2},
    \end{align*}
    and 
    \begin{align*}
        \Biggl|\int_{0}^{T}\int_{\Omega} \rho_{\ddagger} \nabla W * &\xi \cdot \nabla \eta \, dx \, dt\Biggr| \\
        &\le \|\rho_{\ddagger}\|_{L^{\infty}(0,T;L^{2}(\Omega))} \|\nabla \eta\|_{L^{2}(0,T;L^{2}(\Omega))} \|\nabla W * \xi \|_{L^{2}(0,T;L^{\infty}(\Omega))} \\
        &\le (2\pi)^{d/2} \|\rho_{\ddagger}\|_{L^{\infty}(0,T;L^{2}(\Omega))} \|\nabla \eta\|_{L^{2}(0,T;L^{2}(\Omega))} \|\nabla W\|_{\infty} \|\xi\|_{L^2(0, T; L^2(\Omega))} \\
        &\le C(T) \|\rho_0\| R \sigma (2\pi)^{d/2} \|\nabla \eta\|_{L^{2}(0,T;L^{2}(\Omega))}\|\xi\|_{L^{2}(0,T;L^{2}(\Omega))} \\
        &\le \frac{\sigma}{2}\|\nabla \eta\|_{L^{2}(0,T;L^{2}(\Omega))}^{2} + \frac{C(T)^2 \|\rho_0\|^2 R^2 \sigma (2\pi)^d}{2} \|\xi\|_{L^{2}(0,T;L^{2}(\Omega))}^{2},
    \end{align*}
    where the third inequality uses similar arguments from \Cref{sup:prop:unif_L2_est} to bound $\|\rho_{\ddagger}\|_{L^{\infty}(0,T; L^2(\Omega))}$.
    Grouping these inequalities, we obtain 
    \begin{equation*}
    \begin{split}
        \int_0^T \langle \partial_t \xi(t), \eta(t) \rangle \, dt &+ \sigma \int_0^T \int_{\Omega} \nabla \xi \cdot \nabla \eta \, dx \, dt \\
        &\le \sigma \|\nabla \eta\|_{L^2(0, T; L^2(\Omega))}^2 + (K_1 + K_2(T)\|\rho_0\|^2)\|\xi\|_{L^2(0,T; L^2(\Omega))},
    \end{split}
    \end{equation*}
    for some constants $K_1 > 0$ and $K_2(T) > 0$.
    
    Setting $\eta = \xi$ and applying \Cref{sup:thm:Evans} leads to 
    \begin{align*}
        \int_{0}^{T} \frac{1}{2}\frac{d}{dt}\|\xi(t)\|^{2} \, dt \le \left(K_1 + K_2(T)\|\rho_0\|^{2}\right) \int_0^T \|\xi(t)\|_2^{2} \, dt,
    \end{align*}
    which is equivalent to
    \[
    \|\xi(T)\|^2 - \|\xi(0)\|^2 \le 2(K_1 + K_2(T)\|\rho_0\|^2) \int_0^T \|\xi(t)\|^2 \, dt.
    \]
    Grönwall's inequality implies 
    \[
    \|\xi(T)\|^2 \le \|\xi(0)\|^2 K(T)
    \]
    for another constant $K(T) > 0$.
    But $\|\xi(0)\| = \|\rho_{0} - \rho_{0}\| = 0$, so $\rho_{\dagger}(t) = \rho_{\ddagger}(t)$ for all $t$.
    
    Finally, the energy estimate is from \Cref{sup:lem:weakconv}. 
\end{proof}

Throughout this section, we assumed that the initial condition is smooth. 
We can relax this condition to $\rho_{0} \in L^{2}(\Omega)$ by mollifying the initial data. 

\begin{theorem}
    \label{sup:thm:existence-1}
    Let $\rho_{0}\in L^{2}(\Omega) \cap \Pac$.
    Then, there exists a unique weak solution $\rho$ to equation~\eqref{sup:eq:mv_equation} with the estimate
    \[
    \|\rho\|_{L^{\infty}(0,T;L^{2}(\Omega))} + \|\rho\|_{L^{2}(0,T;H^{1}(\Omega))} + \|\partial_t \rho\|_{L^{2}(0,T;H^{-1}(\Omega))}\le C(T)\|\rho_0\|.
    \]
\end{theorem}

\begin{proof}
    Let $\epsilon>0$ and consider the modified problem of~\eqref{sup:eq:sequence_evolution}
    \begin{align}
        \label{sup:eq:sequence_evolution_mollify}
        \partial_t \rho_{n}^{\epsilon} - \sigma \Delta \rho_{n}^{\epsilon} = \nabla \cdot \left(\rho_{n}^{\epsilon} \nabla v[\rho_{n-1}^{\epsilon}] \right), \quad \text{in } \Omega \times (0,T),
    \end{align}
    with periodic boundary conditions on $\Omega$, and $\rho_{n}^{\epsilon}(\cdot, 0) = \rho^{\epsilon}_{0} \coloneq  \varphi_{\epsilon} * \rho_0$.
    Here,
    \[
    \varphi_\varepsilon(x):=\varepsilon^{-d}\sum_{k\in\mathbb{Z}^d}
    \varphi\!\left(\frac{x+k}{\varepsilon}\right)
    \]
    with $\varphi \in  C_c^\infty(\mathbb{R}^d)$, $\varphi \ge 0$, and $\|\varphi\|_1  = 1$, so that $\varphi_\epsilon \ge 0$ and $\|\varphi_\epsilon\|_{1}=1$.
    With the mollification, the function $\rho_{0}^{\epsilon}$ is smooth and we can apply \Cref{sup:thm:existence} to conclude that there exists a unique weak solution $\rho^{\epsilon} \in {L^{\infty}(0,T;L^{2}(\Omega))}\cap {L^{2}(0,T;H^{1}(\Omega))}$, with $\partial_t \rho^{\epsilon} \in {L^{2}(0,T;H^{-1}(\Omega))}$ to equation~\eqref{sup:eq:sequence_evolution_mollify} with the estimate 
    \begin{equation*}
        \|\rho^{\epsilon}\|_{L^{\infty}(0,T;L^{2}(\Omega))} + \|\rho^{\epsilon}\|_{L^{2}(0,T;H^{1}(\Omega))} + \|\partial_t \rho^{\epsilon}\|_{L^{2}(0,T;H^{-1}(\Omega))} \le C(T) \|\rho^{\epsilon}_{0}\|.
    \end{equation*}
    For all $\epsilon > 0$, $\|\rho_0^{\epsilon}\| \le \|\varphi_{\epsilon}\|_1 \|\rho_0\| = \|\rho_0\|$. 
    Hence, there exists a sequence $(\epsilon_k)_{k \ge 1}$, with
    \begin{eqnarray}
            \label{sup:eq:weak_conv_mollify}
        \begin{cases}
        \rho^{\epsilon_k} \rightharpoonup \rho & \mbox{in }L^{2}(0,T;H^{1}(\Omega)),\\
        \partial_t \rho^{\epsilon_k} \rightharpoonup \partial_t \rho & \mbox{in }L^{2}(0,T;H^{-1}(\Omega)),
        \end{cases}
    \end{eqnarray}
    as $k\rightarrow \infty$.
    Additionally, 
    \begin{equation}
        \|\rho\|_{L^{\infty}(0,T; L^{2}(\Omega))} + \|\rho\|_{L^{2}(0,T; H^{1}(\Omega))} + \|\partial_t \rho\|_{L^{2}(0,T; H^{-1}(\Omega))}\le C(T)\|\rho_0\|,
    \end{equation}
    so that $\rho\in {L^{\infty}(0,T;L^{2}(\Omega))}\cap {L^{2}(0,T; H^{1}(\Omega))}$, with $\partial_t \rho \in {L^{2}(0,T; H^{-1}(\Omega))}$.
    It remains to prove that $\rho$ is in fact a weak solution to~\eqref{sup:eq:mv_equation}. 
    
    We have that 
    \begin{equation*}
        \int_{0}^{T} \left\langle \partial_t \rho^{\epsilon_k}(t), \eta(t)\right\rangle \, dt + \sigma \int_{0}^{T} \int_{\Omega} \nabla \rho^{\epsilon_k} \cdot \nabla \eta \, dx \, dt + \int_{0}^{T}\int_{\Omega} \rho^{\epsilon_k} \nabla v[\rho^{\epsilon_k}] \cdot \nabla \eta \, dx \, dt = 0.
    \end{equation*} 
    Using~\eqref{sup:eq:weak_conv_mollify}, we can replace $\rho^{\epsilon_k}$ by $\rho$ in the first two integrals above in the limit $k \rightarrow \infty$. 
    Moreover, as in~\eqref{sup:eq:conv0}, 
    \begin{equation}
        \label{sup:eq:111}
        \begin{split}
            \int_0^T \int_{\Omega} \rho^{\epsilon_k} \nabla v[\rho^{\epsilon_k}] &\cdot \nabla \eta \, dx \, dt = \int_0^T \int_{\Omega} (\rho^{\epsilon_k} - \rho) \nabla v[\rho^{\epsilon_k}] \cdot \nabla \eta \, dx \, dt \\
            &+ \int_0^T \int_{\Omega} \rho \nabla v[\rho] \cdot \nabla \eta \, dx \, dt + \int_0^T \int_{\Omega} \rho \nabla W * (\rho^{\epsilon_k} - \rho) \cdot \nabla \eta \, dx \, dt.
        \end{split}
    \end{equation}
    
    Using similar arguments to those of \Cref{sup:thm:existence} and noticing that $\rho^{\epsilon_k} \to \rho$ strongly in $L^1(0, T; L^1(\Omega))$, we get that \eqref{sup:eq:111} follows the same convergence of \eqref{sup:eq:conv0}.
    Thus, we obtain
    \begin{equation*}
        \int_{0}^{T}\left\langle \partial_t\rho(t),\eta(t)\right\rangle \, dt + \sigma\int_{0}^{T} \int_{\Omega} \nabla\rho \cdot \nabla \eta \, dx \, dt + \int_{0}^{T} \int_{\Omega} \rho \nabla v[\rho] \cdot \nabla \eta \, dx \, dt=0.
    \end{equation*} 
    To show that $\rho(0) = \rho_0$, we again take $\eta \in C^1\left(0,T;H^1(\Omega)\right)$ with $\eta(T)=0$. 
    Since $\rho_0^{\epsilon_k} \rightarrow \rho_0$ in $L^2(\Omega)$, we have (cf. expressions~\eqref{sup:eq:comp_1} and~\eqref{sup:eq:comp_2}) 
    \begin{equation*}
        \langle \rho(0), \eta(0) \rangle = \langle \rho_{0}, \eta(0) \rangle.
    \end{equation*}
    Since $\eta$ is arbitrary, we have $\rho(0) = \rho_0$. 
    The uniqueness follows from exactly the same argument in the proof of \Cref{sup:thm:existence}. 
\qquad\end{proof}

\subsection{Higher Regularity\label{sup:sec:regularity}}

Now the goal is to improve the regularity of the weak solution to~\eqref{sup:eq:mv_equation}.
In this section, we assume that $V, W \in W^{k, \infty}(\Omega)$ for a sufficiently high $k \in \mathbb{N}$.
As in the previous section, we always mollify $\rho_{0}$ by $\varphi_{\epsilon}$ so that the resulting evolution equations~\eqref{sup:eq:sequence_evolution} admit smooth solutions.
This allows us to differentiate the equation as many times as required, and we take the limit $\epsilon \rightarrow 0$ at the end.
For simplicity of notation, we drop the $\epsilon$ superscripts on $\rho_n$ and implicitly perform the limit at the end. 

First, we prove a useful estimate. 

\begin{proposition}
    \label{sup:prop:induction_estimate}
    Let $k \in \mathbb{N}$, $\varphi \in H^{k}(\Omega)$ and $\rho \in L^1(\Omega)$.
    Assume $V, W \in W^{k+1, \infty}(\Omega)$.
    Then, for any multiindex $\alpha$ with $|\alpha| = k$,
    \begin{equation*}
        \|\partial^{\alpha}_x (\varphi \nabla v[\rho])\| \le C\|\varphi\|_{H^{k}(\Omega)},
    \end{equation*}
\end{proposition}

\begin{proof}
    Leibniz differentiation formula says that 
    \begin{align*}
        [\partial^{\alpha}_x (\varphi \nabla v[\rho])]_i &= \left[\sum_{|\beta|\le|\alpha|} \genfrac(){0pt}{0}{\alpha}{\beta} (\partial^{\beta}_x \varphi)(\partial_x^{\alpha-\beta}\nabla v[\rho])\right]_i \\ 
        &= \left[\sum_{|\beta|\le|\alpha|} \genfrac(){0pt}{0}{\alpha}{\beta} (\partial^{\beta}_x \varphi)(\partial_x^{\alpha-\beta+e_i}(V + W * \rho))\right]_i
    \end{align*}
    so that, applying the triangle inequality and H\"older's inequality yields
    \begin{equation*}
    \begin{split}
        \|\partial^{\alpha}_x (\varphi \nabla v[\rho])\| &\le \sum_{|\beta|\le|\alpha|} \genfrac(){0pt}{0}{\alpha}{\beta} \sum_{i=1}^d \|\partial^{\beta}_x \varphi\|_2\|\partial_x^{\alpha-\beta+e_i}(V + W * \rho)\|_{\infty} \\
        &\le \|\varphi\|_{H^k(\Omega)}\sum_{|\beta|\le|\alpha|} \genfrac(){0pt}{0}{\alpha}{\beta} \sum_{i=1}^d \left(\|\partial_x^{\alpha-\beta+e_i}V\|_{\infty} + \|\partial_x^{\alpha-\beta+e_i}W * \rho\|_{\infty}\right) \\
        &\le \|\varphi\|_{H^k(\Omega)} \sum_{|\beta|\le|\alpha|} \genfrac(){0pt}{0}{\alpha}{\beta} \sum_{i=1}^d\left(\|\partial_x^{\alpha-\beta+e_i}V\|_{\infty} + \|\partial_x^{\alpha-\beta+e_i}W\|_{\infty} \|\rho\|_{1}\right).
    \end{split}
    \end{equation*}
    Given that $\rho \in L^1(\Omega)$, we can set the constant
    \[
    C \coloneq \sum_{|\beta|\le|\alpha|} \genfrac(){0pt}{0}{\alpha}{\beta} \sum_{i=1}^d\left(\|\partial_x^{\alpha-\beta+e_i}V\|_{\infty} + \|\partial_x^{\alpha-\beta+e_i}W\|_{\infty} \|\rho\|_{1}\right).
    \]
\end{proof}

Now, we assume that $\rho_0 \in H^k(\Omega)$ for some $k \geq 0$ and prove the corresponding regularity of $\rho$. 

\begin{theorem}
    \label{sup:thm:x_regular} 
    Let $k \in \mathbb{N}$ and assume that $V, W \in W^{k+1, \infty}(\Omega)$ and $\rho_{0} \in H^{k}(\Omega) \cap \Pac$.
    Then the unique solution to~\eqref{sup:eq:mv_equation} satisfies 
    \begin{equation*}
        \rho \in L^{2}(0,T;H^{k+1}(\Omega)) \cap L^{\infty}(0,T;H^{k}(\Omega)),
    \end{equation*}
    with the estimate 
    \begin{equation*}
        \|\rho\|_{L^{2}(0,T; H^{k+1}(\Omega))} + \|\rho\|_{L^{\infty}(0,T; H^{k}(\Omega))} \le C\left(\rho_{0};k,T\right) \coloneq C(T) \|\rho_0\|_{H^{k}(\Omega)},
    \end{equation*}
    for a constant $C(T) > 0$.
\end{theorem}

\begin{proof}
    We prove the statements by proving uniform estimates on $\rho_{n}$
    by induction on $k$. 
    The base case $k=0$ is provided in \Cref{sup:prop:unif_L2_est}.
    The case $k=1$ is given in \Cref{sup:prop:unif_L2_est_higher}. 
    Suppose for some $k \geq 1$,
    \begin{equation}
        \label{sup:eq:inductive_hyp}
        \|\rho_{n}\|_{L^{2}(0,T;H^{k+1}(\Omega))} + \|\rho_{n}\|_{L^{\infty}(0,T;H^{k}(\Omega))} \le C\left(\rho_{0};k,T\right),
    \end{equation}
    for all $n$. 
    We apply $\partial^{\alpha}_x \coloneq \partial^{\alpha_1}_{x_1} \cdots \partial^{\alpha_d}_{x_d}$ with $|\alpha|=\alpha_1+\cdots+\alpha_d = k$ to equation \eqref{sup:eq:sequence_evolution} and multiply it by $-\Delta \partial_x^{\alpha} \rho_{n}$ and integrate over $\Omega$ to get
    \begin{align*}
        \frac{1}{2}\frac{d}{dt}\|\nabla \partial_x^{\alpha}\rho_{n}(t)\|^{2} + \sigma \|\Delta \partial_{x}^{\alpha}&\rho_{n}(t)\|^{2} \le \int_{\Omega}\left| \Delta \partial_{x}^{\alpha} \rho_{n}(t) \nabla \cdot \partial_{x}^{\alpha}\left(\rho_{n}\nabla v[\rho_{n-1}]\right)(t)\right|dx\nonumber \\
        &\le \frac{\sigma}{2}\|\Delta \partial_{x}^{\alpha}\rho_{n}(t)\|^{2} + \frac{1}{2\sigma}\|\nabla \cdot \partial_{x}^{\alpha}\left(\rho_{n}\nabla v[\rho_{n-1}]\right)(t)\|^{2}.
    \end{align*}
    Using \Cref{sup:prop:induction_estimate} with $\varphi = \rho_{n}(t)$ and $\rho = \rho_{n-1}(t)$, we have
    \[
    \frac{d}{dt}\|\nabla \partial_{x}^{\alpha}\rho_{n}(t)\|^{2} + \sigma\|\Delta \partial_{x}^{\alpha}\rho_{n}(t)\|^{2} \le K\|\rho_{n}(t)\|_{H^{k+1}(\Omega)}^{2}
    \]
    for a constant $K > 0$.
    Summing over $|\alpha| = k$ yields
    \[
    \frac{d}{dt} A_n(t) + \sigma B_n(t) \le \tilde{K} \|\rho_n(t)\|^2_{H^{k+1}(\Omega)},
    \]
    where $A_n(t) \coloneq \sum_{|\alpha|=k}\|\nabla\partial^\alpha\rho_n(t)\|_2^{2}$, $B_n(t) \coloneq \sum_{|\alpha|=k}\|\Delta\partial^\alpha\rho_n(t)\|_2^{2}$, and $\tilde{K} > 0$ is a constant.
    Integrating this equation from $0$ to $t$ leads to 
    \begin{align*}
        A_n(t) + \sigma \int_0^t B_n(s) \, ds &\le \tilde{K} \|\rho_n\|^2_{L^2(0,t; H^{k+1})} + A_n(0) \\
        &\le \tilde{K} \|\rho_n\|^2_{L^2(0,T; H^{k+1}(\Omega))} + d\sum_{\|\beta\| = k+1} \|\partial_x^{\beta} \rho_0\|^2,
    \end{align*}
    where we used the fact that 
    \[
    \sum_{\|\beta\| = k+1} \|\partial_x^{\beta} \rho_n(t)\|^2 \le A_n(t) = \sum_{|\alpha|=k} \sum_{i=1}^d \|\partial^{\alpha+e_i}_x \rho_n(t)\|^2 \le d \sum_{\|\beta\| = k+1} \|\partial_x^{\beta} \rho_n(t)\|^2,
    \]
    which can be seen because for each $\beta$, there are $m(\beta) \coloneq |\{i = 1,\dots,d | \beta_i \ge 1\}|$ possible multiindices $\alpha$ that generate the same vector by setting $\alpha = \beta - e_i$.
    Then, using the induction step,
    \[
    \sum_{\|\beta\| = k+1} \|\partial_x^{\beta} \rho_n(t)\|^2 \le \tilde{K} C(T) \|\rho_0\|_{H^k(\Omega)} + d(\|\rho_0\|_{H^{k+1}(\Omega)} - \|\rho_0\|_{H^k(\Omega)}),
    \]
    which leads to
    \[
    \sup_{0 \le t \le T} \|\rho_n\|_{H^{k+1}(\Omega)} \le \tilde{C}(T)\|\rho_0\|_{H^{k+1}(\Omega)}
    \]
    for some constant $\tilde{C}(T)$.
    
    We also obtain that 
    \[
    \int_0^T B_n(s) \, ds \le \frac{\tilde{K} C(T)}{\sigma} \|\rho_0\|_{H^k(\Omega)} + \frac{d}{\sigma}(\|\rho_0\|_{H^{k+1}(\Omega)} - \|\rho_0\|_{H^k(\Omega)}).
    \]
    Notice that 
    \[
    B_n(s) = \sum_{|\alpha|=k} \|\Delta \partial^{\alpha}_x \rho_n(t)\|^2 = \sum_{|\alpha|=k} \sum_{i,j=1}^d \|\partial^{\alpha + e_i + e_j}_x \rho_n(t)\|^2  \ge \sum_{\|\beta\| = k+2} \|\partial_x^{\beta} \rho_n(t)\|^2,
    \]
    where the second equality is due to the periodic boundary conditions.
    Therefore,
    \[
    \|\rho_n\|_{L^2(0,T; H^{k+2}(\Omega))} \le \tilde{\tilde{C}}(T)\|\rho_0\|_{H^{k+1}(\Omega)}.
    \]
    This completes the induction. 
    Taking limits, we obtain
    \begin{equation*}
        \rho \in L^{2}(0,T;H^{k+2}(\Omega)) \cap L^{\infty}(0,T;H^{k+1}(\Omega)),
    \end{equation*}
    with the correct estimate.
\end{proof}

We next improve the regularity in the time domain.

\begin{theorem}
    \label{sup:thm:reg_xt} 
    Let $k \in \mathbb{N}$ and assume that $V,W \in W^{2k + 1, \infty}(\Omega)$ and $\rho_{0} \in H^{2k}(\Omega) \cap \Pac$.
    Then,
    \begin{enumerate}
        \item[(i)] For every $0 \le m\le k$, the unique solution to \eqref{sup:eq:mv_equation} satisfies 
        \[
        \frac{d^{m}\rho}{dt^{m}} \in L^2(0,T; H^{2k-2m+1}(\Omega)) \cap L^{\infty}(0, T; H^{2k-2m}(\Omega)),
        \]
        with the estimate
        \[
        \sum_{m=0}^{k}\left(\left\|\frac{d^{m}\rho}{dt^{m}}\right\|_{L^{2}(0,T;H^{2k-2m+1}(\Omega))}+\left\|\frac{d^{m}\rho}{dt^{m}}\right\|_{L^{\infty}(0,T;H^{2k-2m}(\Omega))}\right) \le D\left(\rho_{0};k,T\right),
        \]
        where 
        \[
        D\left(\rho_{0};k,T\right) \coloneq \left(\sum_{j=0}^{k}C\left(\rho_{0};2k,T\right)^{2^{j+1}}\right)^{1/2}.
        \]
    
        \item[(ii)]
        $\displaystyle
        \frac{d^{k+1}\rho}{dt^{k+1}}\in L^{2}(0,T;H^{-1}(\Omega))$ such that $\displaystyle \left\|\frac{d^{k+1}\rho}{dt^{k+1}}\right\|_{L^{2}(0,T;H^{-1}(\Omega))} \le D\left(\rho_{0};k,T\right).
        $
    \end{enumerate}
\end{theorem}

\begin{proof}
    Let us prove that for all $M \le k$
    \begin{align}
        \label{sup:eq:inductive_hyp-1}
        \sum_{m=0}^{M} \Bigl( \left\|\frac{d^{m}\rho_{n}}{dt^{m}} \right\|_{L^{2}(0,T;H^{2k-2m+1}(\Omega))} &+ \left\|\frac{d^{m}\rho_{n}}{dt^{m}}\right\|_{L^{\infty}(0,T;H^{2k-2m}(\Omega))}\Bigr) \\ 
        &\le C\left(\rho_{0};k,M,T\right) \coloneq \left(\sum_{j=0}^{M}C\left(\rho_{0};2k,T\right)^{2^{j+1}}\right)^{1/2} \nonumber
    \end{align}
    by induction on $M$ up to $k$.
    The case $M=0$ is \Cref{sup:thm:x_regular}. 
    Suppose the estimate \eqref{sup:eq:inductive_hyp-1} holds for some $0 \le M < k$. 
    Differentiating~\eqref{sup:eq:sequence_evolution} $M$ times with respect to $t$ and using the Leibniz rule, we have
    \begin{align*}
        \rho_{n}^{\left(M+1\right)} &= \sigma \Delta \rho_{n}^{\left(M\right)} + \nabla \cdot \left(\rho_{n}\nabla v[\rho_{n-1}]\right)^{\left(M\right)} \\ 
        &= \sigma \Delta \rho_{n}^{\left(M\right)} + \sum_{m=0}^{M} \genfrac(){0pt}{0}{M}{m} \nabla \cdot \left(\rho_{n}^{\left(m\right)} \nabla v\Bigr[\rho_{n-1}^{\left(M-m\right)}\Bigr]\right).
    \end{align*}
    where we used the shorthand $\rho_{n}^{\left(m\right)}:=\partial^{m}\rho_{n}/\partial t^{m}$.
    Thus, we have
    \begin{align*}
        \|\rho_{n}^{\left(M+1\right)}(t)\|_{H^{2k-2M-1}(\Omega)}^{2} \le C_{1}\|&\rho_{n}^{\left(M\right)}(t)\|_{H^{2k-2M+1}(\Omega)}^{2} \\
        &+ C_2\sum_{m=0}^{M}\|\rho_{n}^{\left(m\right)}(t) \nabla v\Bigl[\rho_{n-1}^{\left(M-m\right)}\Bigr](t)\|_{H^{2k-2M}(\Omega)}^{2}
    \end{align*}
    for some constants $C_1, C_2 > 0$.
    
    Using \Cref{sup:prop:induction_estimate} with $\varphi = \rho_{n}^{\left(m\right)}(t)$
    and $\rho = \rho_{n-1}^{\left(M-m\right)}(t)$, we have
    \begin{align*}
        \|\rho_{n}^{\left(M+1\right)}(t)\|_{H^{2k-2M-1}(\Omega)}^{2} &\le C_{1}\|\rho_{n}^{\left(M\right)}(t)\|_{H^{2k-2M+1}(\Omega)}^{2} \\
        &\quad\quad + C_2\sum_{m=0}^{M}\|\rho_{n}^{\left(m\right)}(t)\nabla v\Bigl[\rho_{n-1}^{\left(M-m\right)}\Bigr](t)\|_{H^{2k-2M}(\Omega)}^{2} \\
        & \le C_{1}\|\rho_{n}^{\left(M\right)}(t)\|_{H^{2k-2M+1}(\Omega)}^{2} + \tilde{C}_2\sum_{m=0}^{M}\|\rho_{n}^{\left(m\right)}(t)\|_{H^{2k-2M}(\Omega)}^{2} \\
        & \le C_{1}\|\rho_{n}^{\left(M\right)}(t)\|_{H^{2k-2M+1}(\Omega)}^{2} +\tilde{C}_2\sum_{m=0}^{M}\|\rho_{n}^{\left(m\right)}(t)\|_{H^{2k-2m+1}(\Omega)}^{2},
    \end{align*}
    where the constant $\tilde{C}_2$ incorporates \Cref{sup:prop:induction_estimate} constant for all the different values of $\alpha$.
    Integrating over time then gives
    \begin{align*}
        \|\rho_{n}^{\left(M+1\right)}(t)\|_{L^2(0, T; H^{2k-2M-1}(\Omega))}^{2} &\le C_{1}\|\rho_{n}^{\left(M\right)}(t)\|_{L^2(0, T; H^{2k-2M+1}(\Omega))}^{2} \\
        &\quad\quad + \tilde{C}_2 \sum_{m=0}^{M}\|\rho_{n}^{\left(m\right)}(t)\|_{L^2(0, T; H^{2k-2m+1}(\Omega))}^{2}.
    \end{align*}
    Since $0\le m, M-m\le M$, we can apply the inductive hypothesis
    \eqref{sup:eq:inductive_hyp-1} to conclude that 
    \begin{align*}
        \|\rho_{n}^{\left(M+1\right)}\|_{L^{2}(0,T;H^{2k-2M-1}(\Omega))}^{2} & \le C_1 C\left(\rho_{0};k,M,T\right)^{2} + \tilde{C}_2 C\left(\rho_{0};k,M,T\right)^2 \\
        & \le C\left(\rho_{0};k,M+1,T\right),
    \end{align*}
    with the appropriate constant.
    A similar inequality for $\|\rho_{n}^{\left(M+1\right)}\|_{L^{\infty}(0,T;H^{2k-2M-2}(\Omega))}^{2}$ can be obtained, which completes the induction on $M$ up to $k$. 
    Putting $M=k$ into \eqref{sup:eq:inductive_hyp-1} and taking limits proves part (i).
    
    To prove the second part, notice that 
    \begin{equation*}
        \rho_{n}^{\left(k+1\right)}=\nabla \cdot \left(\sigma \nabla {\rho_{n}}^{\left(k\right)} + \sum_{m=0}^{k} \genfrac(){0pt}{0}{k}{m} \left(\rho_{n}^{(m)} \nabla v\Bigl[\rho_{n-1}^{(k-m)}\Bigr]\right)\right).
    \end{equation*}
    Hence,
    \begin{align*}
        \|\rho_{n}^{\left(k+1\right)}\|_{L^{2}(0,T;H^{-1}(\Omega))}^{2} & \le \left\|\sigma \nabla \rho_{n}^{\left(k\right)} + \sum_{m=0}^{k} \genfrac(){0pt}{0}{k}{m}\left(\rho_{n}^{\left(m\right)} \nabla v\Bigl[\rho_{n-1}^{\left(k-m\right)}\Bigr] \right) \right\|_{L^{2}(0,T;L^{2}(\Omega))}^{2} \\
        & \le C_{1}\|\rho_{n}^{\left(k\right)}\|_{L^{2}(0,T;H^{1}(\Omega))}^{2}  + C_2\sum_{m=0}^{k}\|\rho_{n}^{\left(m\right)}\|_{L^{2}(0,T;L^{2}(\Omega))}^{2} \\
        & \le D\left(\rho_{0};k,T\right).
    \end{align*}
    Taking limits then proves part (ii).
\end{proof}

\begin{corollary}\label{sup:cor:last_reg}
    Let $T>0$ and fix $k = \ceil*{2 + d/4}$.
    Assume $V, W \in W^{2k + 1, \infty}(\Omega)$ and $\rho_0 \in H^{2k}(\Omega) \cap \Pac$.
    Then the unique solution to~\eqref{sup:eq:mv_equation} satisfies
    \[
    \rho\in C^{1}\left(0,T;C^{2}(\Omega)\right).
    \]
\end{corollary}

\begin{proof}
    Applying \Cref{sup:thm:x_regular}, we get that 
    \[
    \rho \in L^{\infty}(0, T; H^{2k}(\Omega)) \cap L^2(0, T; H^{2k+1}(\Omega))
    \]
    and applying \Cref{sup:thm:reg_xt}, we get that 
    \[
    \partial_t \rho \in L^{\infty}(0,T; H^{2k-2}(\Omega)) \cap L^2(0, T; H^{2k-1}(\Omega)), \quad \partial_{tt} \rho \in L^2(0, T; H^{2k-3}(\Omega)).
    \]
    Then, \Cref{sup:thm:Evans} implies $\partial_t \rho \in C(0, T; H^{2k-2}(\Omega))$ and $\rho \in C(0, T; H^{2k}(\Omega))$.
    Because $2k - 2 > 2 + d/2$, Sobolev embedding yields 
    \[
    H^{2k-2}(\Omega) \hookrightarrow C^2(\Omega), \quad H^{2k}(\Omega) \hookrightarrow C^2(\Omega)
    \]
    hence $\rho \in C^1(0, T; C^2(\Omega))$.
\end{proof}

We conclude with the following theorem, excluding the regularity at $t = 0$.

\begin{theorem}
    Let $T > 0$ and fix $r = 4 + \floor*{d/2}$.
    Assume $V, W \in W^{r, \infty}(\Omega)$ and $\rho_0 \in \Pac$.
    Then the weak solution $\rho$ of~\eqref{sup:eq:mv_equation} is classical for every $\tau > 0$,
    \[
    \rho \in C(\tau, T; C^2(\Omega)) \cap C^1(\tau, T; C(\Omega)).
    \]
\end{theorem}

\begin{proof}
    The function $\rho$ satisfies the Duhamel representation, i.e.,
    \begin{equation}
        \label{sup:eq:duhamel}
        \rho(t) = S(t-s)\rho(s) + \int_s^t S(t-\theta) \nabla \cdot \big(\rho(\theta) \nabla v[\rho(\theta)]\big)\,d\theta
    \end{equation}
    for $0 \le s \le t \le T$, where $S(t) \coloneqq e^{\sigma t \Delta}$ is the heat semigroup on $\Omega$ with periodic boundary conditions.
    Well-known energy estimates~\cite[Sections 2.5, 2.6, 4.3, and 7.2]{pazy2012semigroups} show that 
    \begin{equation}
        \label{sup:eq:semigroup}
        \|S(t) f\|_{H^m(\Omega)} \le \|f\|_{H^m(\Omega)} \qquad
        \|S(t)\nabla\cdot F\|_{H^m(\Omega)} \le C t^{-1/2}\|F\|_{H^m(\Omega)}
    \end{equation}
    for all integer $m \ge 0$ and $t > 0$, and some constant $C > 0$ that does not depend on $F$.
    We also have the smoothing property that $\|S(t) f\|_{H^{m}(\Omega)} \le C t^{-m/2-d/4} \|f\|_{1}$ by using Young's inequality for convolution on the fundamental solution of the heat equation. 
    
    Fix $\tau > 0$ and an integer $m \in [0,r-1]$. 
    Using \eqref{sup:eq:duhamel} with $s \in [0,t)$ and \eqref{sup:eq:semigroup},
    \begin{align}
        \label{sup:eq:volterra}
        \|\rho(t)\|_{H^m(\Omega)} &\le \|\rho(s)\|_{H^m(\Omega)} + C \int_s^t (t-\theta)^{-1/2} \|\rho(\theta) \nabla v[\rho(\theta)]\|_{H^m(\Omega)} \, d\theta \\
        &\le \|\rho(s)\|_{H^m(\Omega)} + C_m \int_s^t (t-\theta)^{-1/2} \|\rho(\theta)\|_{H^m(\Omega)} \, d\theta \nonumber,
    \end{align}
    where $C_m$ incorporates the constant from \Cref{sup:prop:induction_estimate}.
    
    For an interval $I$, define $M_m(I) \coloneq \sup_{t \in I} \|\rho(t)\|_{H^m(\Omega)}$.
    Take $0 \le h \le \tau/2$ and set $s = t - h$ for the inequality~\eqref{sup:eq:volterra}. 
    Then, 
    \[
    M_m([\tau, \tau + h]) \le M_m([\tau-h, \tau]) +C_m M_m([\tau-h, \tau+h]) \int_{t-h}^t (t-\theta)^{-1/2} \, d\theta.
    \]
    Since $\int_{t-h}^t (t-\theta)^{-1/2}\, d\theta = 2\sqrt{h}$ and by taking $h > 0$ small such that $2C_m \sqrt{h} \le 1/2$,
    \[
    M_m([\tau,\tau+h]) \le M_m([\tau-h,\tau]) + \frac{1}{2} \max\{M_m([\tau-h, \tau]), M_m([\tau, \tau+h])\}
    \]
    which implies 
    \[
    M_m([\tau,\tau+h]) \le 2 M_m([\tau-h,\tau]).
    \]
    Given that $h > 0$, we can iterate finitely many times to obtain, for some $C > 0$,
    \begin{equation}\label{eq:boundedMm}
        \sup_{t\in[\tau,T]} \|\rho(t)\|_{H^m(\Omega)} \le C \sup_{t \in [\tau/2,\tau]} \|\rho(t)\|_{H^m(\Omega)}.
    \end{equation}
    
    To bound $\|\rho(t)\|_{H^m(\Omega)}$ over $[\tau/2, \tau]$, we return to~\eqref{sup:eq:duhamel} and obtain 
    \begin{align*}
        \|\rho(t)\|_{H^m(\Omega)} &\le C_1 t^{-m/2-d/4}\|\rho_0\|_1 + C_2\int_0^t (t-s)^{-1/2}\|\rho(s)\|_{H^m(\Omega)} \, ds \\
        &\le C_1 (\tau/2)^{-m/2-d/4}+ C_2\int_0^t (t-s)^{-1/2}\|\rho(s)\|_{H^m(\Omega)} \, ds
    \end{align*}
    for some constants $C_1 > 0$ and $C_2 > 0$.
    The fractional Grönwall inequality says that
    \[
    \|\rho(t)\|_{H^m(\Omega)} \le  C_1 (\tau/2)^{-m/2-d/4} E_{1/2}[C_2 \Gamma(1/2)(t^{1/2})], 
    \]
    where $E_{\beta}$ is the Mittag-Leffler function.
    In particular, $\sup_{t\in[\tau/2,\tau]}\|\rho(t)\|_{H^m}\le C_{\tau,m}$ for some constant $C_{\tau,m}$.
    Hence
    \begin{equation}\label{eq:uniformHm}
        \sup_{t\in[\tau,T]}\|\rho(t)\|_{H^m(\Omega)} \le C_{m,\tau,T} .
    \end{equation}
    Because $m\le r-1$ was arbitrary, \eqref{eq:uniformHm} holds for all $m=0,1,\dots,r-1$.

    Take $m_*=r-1$. 
    By our choice of $r$, we have $m_* > 2 + d/2$, so the Sobolev embedding theorem gives $H^{m_*}(\Omega) \hookrightarrow C^2(\Omega)$.
    Therefore, $\rho\in L^\infty(\tau, T; C^2(\Omega))$.
    Regarding the time derivative, write it again as $\partial_t \rho = \nabla \cdot (\sigma \nabla \rho + \rho \nabla v[\rho])$.
    Because $\rho \in L^{\infty}(\tau, T; H^{m_*}(\Omega))\hookrightarrow L^{\infty}(\tau, T; C^2(\Omega))$, we have $\partial_t \rho \in L^{\infty}(\tau, T; H^{m_*-2}(\Omega)) \hookrightarrow L^{\infty}(\tau, T; C(\Omega))$.
    Continuity in time follows from~\eqref{sup:eq:duhamel}. 
    So
    \[
    \rho \in C(\tau, T; C^2(\Omega)), \quad \partial_t \rho \in C(\tau, T; C(\Omega)).
    \]
    
    This proves the theorem.
\end{proof}

\bibliographystyle{siamplain}
\bibliography{references}
\end{document}